\documentclass[
]{amsart}

\usepackage[utf8]{inputenc} 
\usepackage[T1]{fontenc}

\usepackage{amsmath}
\usepackage{amsfonts}
\usepackage{amssymb}
\usepackage{amsthm}
\usepackage{setspace}
\usepackage{amsrefs}
\usepackage{mathrsfs} 
\usepackage{graphicx}
\usepackage[shortlabels]{enumitem}
\usepackage[normalem]{ulem}

\usepackage[usenames]{xcolor}

\usepackage{hyperref}
\hypersetup{hidelinks}

\makeatletter
\def\@tocline#1#2#3#4#5#6#7{\relax
  \ifnum #1>\c@tocdepth 
  \else
    \par \addpenalty\@secpenalty\addvspace{#2}%
    \begingroup \hyphenpenalty\@M
    \@ifempty{#4}{%
      \@tempdima\csname r@tocindent\number#1\endcsname\relax
    }{%
      \@tempdima#4\relax
    }%
    \parindent\z@ \leftskip#3\relax \advance\leftskip\@tempdima\relax
    \rightskip\@pnumwidth plus4em \parfillskip-\@pnumwidth
    #5\leavevmode\hskip-\@tempdima
      \ifcase #1
       \or\or \hskip 1em \or \hskip 2em \else \hskip 3em \fi%
      #6\nobreak\relax
    \hfill\hbox to\@pnumwidth{\@tocpagenum{#7}}\par
    \nobreak
    \endgroup
  \fi} 
\makeatother

\title[]{Global Newlander-Nirenberg theorem on domains with finite smooth boundary in complex manifolds
}
\author[]{Xianghong Gong$^{\dagger}$}

 \address{Department of Mathematics,
 University of Wisconsin-Madison, Madison, WI 53706}
 \email{gong@math.wisc.edu}
 
\author[]{Ziming Shi$^{\ddagger}$}

\address{Department of Mathematics,
	University of California-Irvine, Irvine, CA 92697}
\email{zimings3@uci.edu}

\thanks{$^{\dagger}$Partially supported by  NSF grants DMS-2054989 and DMS-2349865. $^{\ddagger}$Partially supported by AMS-Simons Travel Grant.}


\makeatletter
\newcommand*\bigcdot{\mathpalette\bigcdot@{.75}}
\newcommand*\bigcdot@[2]{\mathbin{\vcenter{\hbox{\scalebox{#2}{$\m@th#1\bullet$}}}}}
\makeatother

\newtheorem{thm}{Theorem}[section]
\newtheorem{cor}[thm]{Corollary}
\newtheorem{prop}[thm]{Proposition}
\newtheorem{lemma}[thm]{Lemma}

\theoremstyle{definition}

\newtheorem{defn}[thm]{Definition}

\newtheorem{rem}[thm]{Remark}

\newtheorem{prob}[thm]{Problem}

\renewcommand{\th}[1]{\begin{thm}\label{#1}}
\renewcommand{\eth}{\end{thm}}
\newcommand{\co}[1]{\begin{cor}\label{#1}}
\newcommand{\eco}{\end{cor}}
\renewcommand{\le}[1]{\begin{lemma}\label{#1}}
\newcommand{\ele}{\end{lemma}}
\newcommand{\pr}[1]{\begin{prop}\label{#1}}
\newcommand{\epr}{\end{prop}}
\newcommand{\pf}[1]{\begin{proof}#1 \end{proof}}

\newcommand{\ga}{\begin{gather}}
\newcommand{\ega}{\end{gather}}
\newcommand{\gan}{\begin{gather*}}
\newcommand{\egan}{\end{gather*}}
\newcommand{\al}{\begin{align}}
\newcommand{\eal}{\end{align}}
\newcommand{\aln}{\begin{align*}}
\newcommand{\ealn}{\end{align*}}
\newcommand{\eq}[1]{\begin{equation}\label{#1}}
\newcommand{\eeq}{\end{equation}}

\newcommand{\DD}[2]{\frac{\partial #1}{\partial #2}}
\newcommand{\f}[2]{\frac{#1}{#2}}

\newcommand{\ci}{~\cite}

\newcommand{\cc}{{\mathbb C}}
\newcommand{\nn}{{\mathbb N}}

\newcommand{\rr}{{\mathbb R}}

\newcommand{\pp}{{\mathbb P}}

\newcommand{\ov}{\overline}

\newcommand{\dist}{\operatorname{dist}}
\newcommand{\dom}{\operatorname{Dom}}
\newcommand{\id}{\operatorname{id}}

\newcommand{\Span}{\operatorname{span}}

\renewcommand{\dbar}{\overline\partial}

\newcommand{\supp}{\operatorname{supp}}

\newcommand{\cL}{\mathcal}

\newcommand{\all}{\alpha}

\newcommand{\gm}{\gamma}

\newcommand{\del}{\delta}
\newcommand{\Del}{\Delta}
\newcommand{\var}{\varphi}
\newcommand{\e}{\varepsilon}
\newcommand{\ve}{\varepsilon}
\newcommand{\om}{\omega}
\newcommand{\Om}{\Omega}

\newcommand{\la}{\lambda}
\newcommand{\La}{\Lambda}

\newcommand{\pd}{\partial}

\newcommand{\yh}{\frac{1}{2}}

\newcommand{\jq}[1]{\langle #1\rangle}

\newcommand{\re}[1]{(\ref{#1})}
\newcommand{\rea}[1]{$(\ref{#1})$}
\newcommand{\rl}[1]{Lemma~\ref{#1}}
\newcommand{\nrc}[1]{Corollary~\ref{#1}}
\newcommand{\rp}[1]{Proposition~\ref{#1}}
\newcommand{\rt}[1]{Theorem~\ref{#1}}

\newcommand{\rla}[1]{Lemma~$\ref{#1}$}

\newcommand{\rpa}[1]{Proposition~$\ref{#1}$}
\newcommand{\rta}[1]{Theorem~$\ref{#1}$}

\newcommand{\db}{\dbar}

\newcommand{\pa}{\partial}
\newcommand{\na}{\nabla}

\newcommand{\wti}{\widetilde}

\newcommand{\Hc}{\mathcal{H}}

\newcommand{\Kc}{\mathcal{K}}
\newcommand{\Nc}{\mathcal{N}}

\newcounter{pp}
\newcommand{\bpp}{\begin{list}{$\hspace{-1em}(\alph{pp})$}{\usecounter{pp}}}
\newcommand{\epp}{\end{list}}

\newcounter{ppp}
\newcommand{\bppp}{\begin{list}{$\hspace{-1em}(\roman{ppp})$}{\usecounter{ppp}}}
\newcommand{\eppp}{\end{list}}

\def\beq{\begin{equation}}
\def\eeq{\end{equation}}

\def\a{\alpha}
\def\b{\beta}

\def\var{\varphi}

\def\ve{\varepsilon}
\def\th{\theta}
\def\g{\gamma}

\def\C{\mathbb{C}}

\def\Om{\Omega}

\def\o\{\O\}
\def\H{\mathcal{H}}
\def\BH{\mathbb{H}}
\def\L{\mathcal{L}}
\def\lap{\triangle}
\def\V{\mathcal{V}}
\def\om{\omega}
\def\ll{\left|}
\def\d{\partial}
\def\we{\wedge}
\def\lapla{\Delta}
\def\bo{\square}
\def\Div{\text{div}}
\def\grad{\triangledown}
\def\S{\mathscr{S}}

\newcommand{\norm}[1]{\lVert #1\rVert}

\newcommand{\znorm}[1]{\lvert #1\rvert}
\newcommand{\pair}[1]{\left<#1\right>}

\newcommand{\red}[1]{{\textcolor{red}{#1}}} 

\newcommand{\Sp}{\operatorname{span}} 

\begin{document}

\begin{abstract}
Let $M$ be a relatively compact $C^2$ domain in a complex manifold $\mathcal M$ of dimension $n$. Assume that $H^{1}(M,\Theta)=0$ where $\Theta$ is the sheaf of germs of holomorphic tangent fields of $M$. Suppose that the Levi-form of the boundary of $M$ has at least 3 negative eigenvalues or at least $n-1$ positive eigenvalues pointwise. We first construct a homotopy formula for $\Theta$-valued $(0,1)$-forms on $\overline M$. We then apply a Nash-Moser iteration scheme to show that if a formally integrable almost complex structure of the H\"{o}lder-Zygmund class $\Lambda^r$ on $\overline M$ is sufficiently close to the complex structure on $ M$ in the H\"{o}lder-Zygmund norm $\Lambda^{r_0}(\overline M)$ for some $r_0>5/2$, then there is a diffeomorphism $F$ from $\overline M$ into $\mathcal M$ that transforms the almost complex structure into the complex structure on $F(M)$, where $F \in \Lambda^s(M)$ for all $s<r+1/2$.
\end{abstract}

 \maketitle

\tableofcontents

\section{Introduction}\label{sec  introduction}
\setcounter{thm}{0}\setcounter{equation}{0}

We consider a complex manifold $\cL M$ of complex dimension $n>1$. Let $\Theta$ is the sheaf of germs of holomorphic tangent fields of $\cL M$. Let $M= \{ \rho<0 \}$ be a relatively compact domain with $C^2$ boundary in   $\cL M$.
In this paper, we will prove the following.
\begin{thm} \label{Global NN}  
Let
$M$ be a relatively compact domain with $C^2$ boundary in   $\cL M$ with $H^1(M,\Theta)=0$. Suppose that at each point of $bM$ the Levi form of $\rho$ has $n-1$ positive eigenvalues or at least 3 negative eigenvalues.   Fix $r_0>5/2$. 
There exists $\del>0$ satisfying the following: if $A\in C^1_{(0,1)}(\ov M,\Theta)$ defines a formally integrable $($see definition below$)$ almost complex structure $H$ on $\ov M$,    $\|A\|_{\Lambda^{r_0}(M)}<\del$ and $A\in\Lambda^r(M)$ with $r_0<r\leq\infty$, then there exists an embedding $F\in\La^{s}(M)$ for all $s<r+1/2$ from $\ov M$ into $\cL M$ that transforms $H$ into the complex structure on $F(\overline M)$.  
\end{thm}

Here $\Lambda^r(M)$ denotes the H\"older--Zygmund space defined in Section~\ref{sec  Settings}. See \rt{Global NN mv} for the full version of \rt{Global NN}. The rest of the introduction provides the motivation of the main theorem  and related open problems. 

Given a complex manifold $\cL M$, let $T\cL M$ and $T\cL M\otimes\cc$ denote   respectively its real and complex tangent bundles.  Assume that $M\Subset \cL M$ is a relatively compact domain with $C^2$ boundary. The complex structure on $\cL M$ or its restriction on $M$ will be denoted by $H_{st}$.
An almost complex structure   on the closure $\ov M$ is defined by a complex subbundle $H$ of $T_\cc \ov M:=(T\cL M\otimes\cc)|_{\ov M}$   such that $T_\cc\ov M=H\oplus\ov H$.
For each $p\in U_k\cap \ov M$, $H_p$ has a basis $$
X_{k,\ov\a}= \sum_{\beta=1}^n\left\{ A_{k,\ov\a}^{\beta}(z_k)\frac{\partial}{\pd z_k^\beta}+B_{k,\ov\a}^{\ov\beta}(z_k)\frac{\pd}{\pd \ov z_k^\beta}\right\}, \quad \a=1,\dots, n,
$$
when $z_k=(z_k^1,\dots, z_k^n)$ are holomorphic coordinates  on $U_k\subset\cL M$ with $\ov M\subset\cup U_k$. The almost complex structure will also be denoted by $H$.
We always assume that  $H$ has  {\it finite distance} from the standard complex structure on $\ov M$, i.e.   $(B_{k,\ov\all}^{\ov\beta})_{1\leq\all,\beta\leq n}$ are invertible. By changing bases, we take $(B_{k,\ov\all}^{\ov\beta})$ as the identity matrix. This determines $A_{k,\ov\all}^\beta$ uniquely and $A:=\sum_{\all,\beta}A^\beta_{k,\ov\all}(z_k)d\bar z^\all_k\otimes\pd/{\pd z_k^\beta}$ is a global section of $\Theta$-valued $(0,1)$ form on $\ov M$.
$H$  is {\it formally integrable} if it is closed under the Lie bracket $[H,H]\subset H$, or equivalently, if $\db A=-\f{1}{2}[A,A]$ (see Section~\ref{sec  Settings} for details). 

We say that the almost complex structure $H$ is \emph{locally integrable} on the interior of $M$ if for each $p\in M$, there exist a neighborhood $U$ of $p$ in $M$ and a  diffeomorphism $F$ from $U$ into $\cL M$ such that $F_\ast (H) = H_{st}$, i.e. $F_*X_{k,\ov\a}$ are in the span of $\{\frac{\partial}{\pd z_k^\beta};1\leq \beta\leq n\}$. Assuming the existence of $F$, since $\Span \{ \pp{}{\ov z_k^\all} \}$ is closed under the Lie bracket, we know that $H$ is closed under the Lie bracket. In other words, if $H$ is locally integrable on $M$ then it is formally integrable on $M$. 

The celebrated Newlander-Nirenberg theorem \cites{MR88770} states that the converse is true locally. Namely, at each interior point $p$ of $M$, if $H$ is formally integrable near $p$, then there exists a neighborhood $U_p$ of $p$ in $M$ on which $H$ is integrable. In this case, we refer the corresponding diffeomorphism $F$ from $U_p \Subset M$ into $\cL M$ as a {\it local embedding} for $H$.  
This classical theorem now has several proofs due to Kohn~\ci{MR1045639}, Malgrange~\ci{MR0253383} and Webster~\ci{MR999729}. See Yao~\ci{MR4211866} for counter-examples on sharpness of the coordinate maps. The reader is also referred to the optimal regularity results in Street~\ci{MR4027743} on the elliptic structures in general. 

The global analogue of the above result, known as the global Newlander-Nirenberg problem, can be stated as follows. Suppose the almost complex structure $H$ is formally integrable on $\ov M$ and is sufficiently close to the complex structure on $\cL M$, is $H$ integrable on $\ov M$? i.e. does there exist a global embedding $F$ for $H: \ov M \to \cL M$?   
Here by "sufficiently close" we mean that the $A$ given above is small in certain norms.   
A similar local boundary version of the Newlander-Nirenberg problem  asks whether a formally integrable $H$ admits an embedding near boundary point of the domain $M$.   

It turns out that certain geometric assumption on the boundary is needed for the global or boundary Newlander-Nirenberg theorem to hold.
We now describe this phenomena. 
Following H\"ormander~\ci{MR0179443},
we say that the domain $M$ (or its boundary $bM$) satisfies the  \emph{condition $a_q$}, if $M = \{ \rho(p)<0\colon p\in\cL M\}$ has a $C^2$ definition function $\rho$,  
and the Levi-form of $\rho$ at each boundary point $p \in U_k \cap bM$, defined by $L_p(\rho,t) = \sum_{\alpha,\beta=1}^n \frac{\pa^2 \rho}{\pa z_k^\alpha \pa {\ov z_k^\beta}}(p) t^\alpha \ov{t^\beta}$, has at least $n-q$ positive or $q+1$ negative eigenvalues on $T^{(1,0)}_p( b M):=\left\{\sum_{\all} t^\all\DD{}{z_k^\all}\big|_p\colon \sum_{\all} t^\all\DD{\rho}{z_k^\all}(p)=0\right\}$.  In \cite{MR0992454}, Hill showed that if the Levi form on the boundary of the domain has exactly $1$ negative eigenvalue, then there are many formally integrable smooth almost complex structures for which local embedding for the structures fails to exist at each boundary point.

We now describe positive results. 
In  pioneering work~\cites{MR0477158, MR594711, MR656198},   Hamilton proved the existence of a $C^\infty$ global embedding for $H$ on $\ov M$ when both $H, b M$ are $C^\infty$ and $M$ is  $a_1$. Hamilton's result has to assume that the perturbation $A$ is small in some norm $\| \cdot \|_{\La^{r_0}}$ with unspecified (large) $r_0$.  
Catlin  \cite{MR959270} and Hanges--Jacobowitz~\cite{MR980299} solved a local version of Newlander--Nirenberg problem near boundary points of a strong pseudoconvex domain. It is worth emphasizing that all these results are only for $C^\infty$ category for the structures and the boundary of the domains, and the proofs are based on $\dbar$-Neumann type methods.

When the domain is strongly pseudoconvex in $\C^n$, Gan and the first-named author~\ci{GG} showed the existence of   global solutions that  have a loss of derivatives with respect to the finite smoothness of the given almost complex structure. This result has been improved significantly in the  very recent work of the second-named author~\ci{shiNN}. 

Therefore,  \rt{Global NN} proves the existence of solution for the global Newlander--Nirenberg problem on $a_1^+$ domains (i.e. the ones satisfying $a_1,a_2$ conditions) with the almost  1/2 regularity gain.   

To prove \rt{Global NN}, we will employ the homotopy formula methods together with a modified Nash-Moser type iteration. 
The idea is to find a sequence of embedding $F^j: M_j \to M_{j+1}$, with $M_0 = M$, so that the the error $A^j$ converges to $0$. We then show that the maps $\wti F_j:= F^j \circ F^{j-1} \circ \cdots \circ F^0$ converge to the embedding $F$ which is the solution. 

These techniques were originally employed by Webster \cites{MR999729, MR1128608, MR995504} to investigate the classical Newlander--Nirenberg theorem, the CR vector bundle problem, and the more challenging
 local CR embedding problem.  
These techniques, together with a more precise interior regularity estimate for Henkin's integral solution operators of $\dbar_b$ on strong pseudoconvex real hypersurfaces, were successfully used by Gong--Webster \cites{MR2742034, MR2829316, MR2868966}  to obtain a sharp version of CR vector bundle problem and the local CR embedding problem; see~\cite{MR3961327}  also for a parameter version of Frobenius--Nirenberg theorem
by using similar techniques.
The scheme of the proof of \rt{Global NN} is similar to the ones in Gan--Gong~\cite{MR3961327} and Shi~\ci{shiNN}. However, we mention an important new feature in the present work. Here we must  deal with the failure of {\it local} homotopy formula for $a_1$ domains near  the concave side of the boundary. Noticeably, the lack of homotopy formulas with good estimates was shown by Nagel--Rosay~\ci{MR1016447}  in the situation for local CR embedding problem in dimension $5$. In our setting, it was proved by Laurent-Thi\'{e}baut and Leiterer~\ci{MR1621967}*{Prop.~0.7} that there is  no {\it local} homotopy formula with good estimates for $(0,q)$ forms near a boundary point of the concave side of some $a_q$ domain.

For the proof  \rt{Global NN}, we construct a global homotopy formula on $a_1^+$ domains.
Let us introduce  notations. 
Note that for an $a_q$ domain $M$,  $ b M$ is the disjoint union of $b_{n-q}^+M$ and $b_{q+1}^-M$, where $b_{n-q}^+M$ (resp. $b_{q+1}^-M$) is the union of the components of $bM$ on which $L_\zeta (\rho, t)$ has at least $(n-q)$ positive (resp. $(q+1)$ negative) eigenvalues at each $\zeta\in b_{n-q}^+M$ (resp. $b^-_{q+1}M$). Define
$$|\rho|_a^*:=\max\{|\rho\|_{C^2(\cL M)},\|\rho\|_{\Lambda^a(\cL N)}\}$$
 for a fixed neighborhood  $\cL N$ of $b_{q+1}^-M$ in $\cL M$.  
We say that $M$ is an $a_q^+$ domain if it satisfies conditions $a_q$ and $a_{q+1}$.  
 \begin{thm}\label{Thm::glob_hf_intro}
 Suppose that $M$ defined by a $C^2$ function $\rho$ in $\cL M$ satisfies condition $a_q^+$ with $q>0$. Let $V$ be a holomorphic vector bundle on $\cL M$.  Assume that  $ H^{q}(D,\cL O(V))$ vanishes. Let $r>0,\ve>0$ and $0\leq\delta\leq1/2$. Then
 $$
 \var =\db R_{q}\var +R_{q+1}\db \var,
$$ 
    where $\var$ is a $V$-valued $(0,q)$ form with  $\var, \db\var\in \Lambda^r(M)$. Furthermore, for $j=q,q+1$,  
 \gan{}\|R_jf\|_{\Lambda^{r+\delta}(M)}\leq C_{r,\delta,\ve }  ( |\rho|^*_{r+2+\delta}|\rho|_{2+\ve}^{b}\|f \|_{\Lambda^{\ve}(M)}+\|f \|_{\Lambda^{r}(M)}). 
\end{gather*}
Here the positive constants $b, 
 C_{r,\delta,\ve }$ depend on $\nabla\rho,
\nabla^2\rho$ and are stable under small perturbation of $\rho$ as defined in Section~$\ref{sec  Settings}$.  
\end{thm} 
\rt{Thm::glob_hf_intro} on homotopy formulas for $a_q^+$ domains in a complex manifold is interesting in its own right. 
We briefly mention a few related results. The construction of $\db$ solution operators for strongly pseudoconvex domains in $\cc^n$ has a long history. Almost $1/2$ gain for $\db$ solutions was obtained by Kerzman~\ci{MR0281944} and  the exact $1/2$ gain estimate was  obtained by Henkin--Ramanov~\ci{MR0293121}. For higher order derivative estimate with exact $1/2$ gain on H\"older spaces, $\db$ solution operators were constructed by Siu~\ci{MR330515} for $(0,1)$ forms and by Lieb--Range~\ci{MR0835763} for $(0,q)$ forms for $q\geq1$ when the boundaries of the strongly pseudoconvex domains are sufficiently smooth. For a homotopy formula on strongly pseudoconvex domains in $\cc^n$ with $C^2$ boundary, Gong~\ci{MR3961327} obtained the $1/2$ gain for H\"older-Zygmund spaces and Shi~\ci{shiNN} obtained an almost $1/2$ gain for $L^p$ Sobolev spaces.  In this paper we will use  Rychkov extension introduced  to the study of $\db$-problem 
by Shi--Yao~\ci{SYajm}.  
We also mention that integral formulas and canonical solutions (with minimum $L^2$ norms) also play important roles in the study of  Kerzman's problem~\ci{MR0281944} on super-normal estimates for polydiscs and more general product domains; a problem was solved  recently  by Li~\ci{MR4855323} and Yuan~\ci{yuan} independently. The  interested reader should consult ~\cites{MR3961327,MR4688544,MR4855323,dong,yuan} and   references therein. 

It is worthy to give an example for which \rt{Global NN} applies. 
Using  a vanishing theorem on harmonic forms of Shaw~\ci{MR2885124}*{Thm.~2.2}, we can prove the following global Newlander-Nirenberg theorem on a shell domain. 
\begin{cor}\label{CnNN}
Let $D=D_2\setminus{\ov D_1}$, where $D_1,D_2$ are two bounded strongly pseudoconvex domains with $C^2$ boundary in $\cc^n$, $\overline{D_1}\subset D_2$, and  $n \geq 4$. Assume that  $r_0>5/2$. There exists $\del>0$ satisfying the following:
if $r_0 \leq r < \infty$ and  $A\in \Lambda_{(0,1)}^r(D, \Theta)$ satisfies $\db A=-\f{1}{2}[A,A]$ and $\|A\|_{\Lambda^{r_0}(D)}<\del$, then the almost complex structure on $D$ defined by $A$ can be transformed into the standard complex structure on $F(D)$ via an embedding $F\in\Lambda^{s}(D)$ from $\ov D$ into $\cc^n$ for all $s<r+1/2$. 
\end{cor}

We conclude the introduction with the following open problems.
\begin{prob}In \rt{Thm::glob_hf_intro}, replace $a_q^+$ condition by $a_q$ condition on $M$ with $C^2$ boundary.    Determine if there is a homotopy formula $f=\db A_qf+A_{q+1}\db f$    satisfying
$$
\|A_jg\|_{\Lambda^{r+1/2}(M)}\leq C_r\|g\|_{\Lambda^{r}(M)}, \quad  j=q,q+1, \forall \: 
r>0.
$$
\end{prob}

\begin{prob} Under conditions of \rt{Global NN},  determine if there is a  solution $F\in\Lambda^{r+1/2}(M)$ when $A\in\Lambda^r(M)$.
\end{prob}
One would expect that a solution to the last question would have applications. For instance, it might be helpful to settle the sharp regularity result
left open in the local CR embedding problem~\cites{MR995504,  MR1263172,    MR2868966} for dimension $7$ or higher.

\section{Outline of the paper} \label{sec  outline}

Section~\ref{sec  Settings} contains some elementary facts about the H\"older--Zygmund spaces $\La^r$. We construct a Moser type smoothing operator $S^t: C^0(\ov D) \to C^\infty(\Nc_t(D))$, where $D$ is a bounded Lipschitz domain. The important new feature here is that $S^t f$ is defined on a neighborhood $\cL N_t(D)$ of $D$. This fact is needed to estimate each iterating embedding $F^j$, which is defined as the time-1 flow of the vector fields $-S^{t_j} P^j A^j$ for a sequence of $A^j$ on $M_j$ for Sections~\ref{set-up} and~\ref{sect:hartogs}. 

In Section~\ref{sect:flow}, we derive estimates for the time-1 flow of a vector field. The main estimate is \re{dbwr}, where we estimate the $\db$ of the flow in terms of the $\db$ of the vector fields. We also prove the crucial estimate \re{nSt} for the commutator $[\na, S^t] = \na S^t - S^t \na$, which is the key to obtaining optimal convergence in the Nash-Moser iteration.  
All results in Sections~\ref{sec  Settings} and~\ref{sect:flow}    can be applied in general situations beyond the scope of this paper. 

In Section~\ref{sect:lhf}, we improve the local homotopy formulas constructed in~\ci{aq} for $a_q^+$ domains by using an estimate on the commutator $[\cL E,\db]$ for the Rychkov extension operator due to Shi--Yao~\ci{MR4688544,SYajm} on spaces $\Lambda^r(M)$ when $0<r\leq1$. 

Section~\ref{ghf} contains proof of \rt{Thm::glob_hf_intro}, i.e., the construction of a global homotopy formula on $a_q^+$ domains. 
First, we use the Grauert bumping method and the local homotopy formulas to construct an approximate global homotopy formula on some bigger $C^\infty$ domain $\wti M$ that is a small $C^2$ perturbation of $M$, with $M \Subset \wti M$ and  
\begin{equation} \label{ } 
   \var = \db T_q \var + T_{q+1} \db \var + K_q \var, \quad 
  \text{on $\wti M$.}    
\end{equation}
Here $K_q \var$ is a $(0,q)$ form with $C^\infty(\wti M)$ coefficients (See \rp{full-bumping}.) We can then apply the $\db$-Neumann formula $I = \Box_q N_q + H^0_q$ to $K_q \var$ on $\wti M$ to get 
\begin{equation} \label{Kq_approx_hf} 
  K_q \var = \db_{q-1} \db^\ast_{q-1} N_q K_q \var + \db_q^\ast \db_q N_q K_q \var + H^0_q K_q \var.     
\end{equation} 
The operator $H^0_q$ maps $L^2_{(0,q)}(\wti M, V)$ to the space of harmonic forms $\Hc_{(0,q)}(\wti M,V)$. 

A well-known result of H\"ormander states that the condition $H^{q}(M, \cL O(V))=0$ is stable under small perturbations of an $a_q$ domain and thus hold for $M$ replaced by $\wti M$. Using this fact we can prove that $\Hc_{(0,q)}(\wti M,V) = 0$ and thus $H^0_q K_q \var = 0$. See \rt{KRthm+} for a general statement allowing non-vanishing of $H^{q}(M,\cL O( V))$.  
 
Now $\wti M$ is still an $a_q^+$ domain since it is a small perturbation of $M$, and the $a_q^+$ condition allows us to commute $\db_q$ with $N_q$: $\db_q N_q = N_{q+1} \db_q $, thus the second term on the right-hand side of \re{Kq_approx_hf} becomes $\db^\ast_q N_{q+1} \db_q K_q \var$. Combining all these facts in \re{} and \re{Kq_approx_hf} we obtained (\rt{Thm::glob_hf_intro}) a global homotopy formula on $M$:  
\[
  \var = \db R_q \var + R_{q+1} \db \var,  
\]
where $R_q \var := T_q + \db \db^\ast N_q K_q \var$ and 
$R_{q+1} \db \var:= T_{q+1} \db \var + \db^\ast N_{q+1} \db K_q \var$. We remark that $\db K_q \var$ depends only on $\db \var$ and not on $\var$. This important fact follows from the Grauert-bumping procedure (Propoitions ~\ref{K-ext} and \ref{smooth-B}).

Section~\ref{set-up} contains the main step in a Nash--Moser iteration for the proof of \rt{Global NN}. We will finish the proof of the theorem in Section~\ref{sect:hartogs}.  The proof of \rt{Global NN} is more subtle when $M$ is not strongly pseudoconvex since \rt{Thm::glob_hf_intro} requires many derivatives for $bM$. We will solve this difficulty by combining the Hartogs's extension and the Nash--Moser iteration scheme.

\section{Preliminaries}\label{sec  Settings}
\setcounter{thm}{0}\setcounter{equation}{0}
 In this section, we first  recall elementary properties of H\"older norms $\norm{\cdot}_{C^a(M)}$ with $0\leq a <\infty$  and H\"older-Zygmund norms $\|{\cdot}\|_{\Lambda^a(M)}$ with $0< a <\infty$
where $M$ is a Lipschitz   domain.
For abbreviation, we also write $\|\cdot\|_{C^a(M)}$    as $\|\cdot\|_{M,a}$  and   $\|\cdot\|_{\Lambda^a(M)}$  as   $|\cdot|_{M,a}$.
 We then describe some properties of  
  the Nash-Moser smoothing operator.

A domain $D\Subset\rr^d$ is called a {\it Lipschitz  domain},  if  there is a finite open covering $\{V_i\}_{i=1}^N$ of $\ov D$ such that
$V_i$ is relatively compact in $D$ or $V_i\cap D$ is given by $1>y_d>h_i(y')$ where $(y',y_d)$ is a permutation of the standard coordinates of $\rr^d$ and $h_i$ is a Lipschitz function on some ball $V_i'$.  
We say that a constant $C( D)$  is {\it stable} under  perturbations of $D$ if $C(D)$ depends only on upper bounds of $N$ and  Lipschitz norms of $h_1,\dots, h_N$.

    A constant $C(\tilde \rho)$, depending on   functions $\tilde \rho$,  is {\it stable} under smaller perturbations of a $C^2$ function $\rho$ if $C(\tilde\rho)$ can be chosen independent of $\tilde\rho$ provided   $\|\tilde\rho-\rho\|_{C^2}<\del$  for some $\del$ depends only on $\rho$. Such a constant will be denoted by $C(\nabla\rho,\nabla^2\rho)$. 
\subsection{Notations}Let $\nn$ be the set of non-negative integers.
Throughout the paper we assume the following.
\bppp
    \item
 $M$ is a relatively bounded domain in a complex manifold $\cL M$ and $ b M$ is at least $C^2$.
\item The almost complex structures $H$ on $\ov M$ have  finite distance from the standard complex structure   on $\ov M$ and $H\in C^1(\ov M).$
\item If $g$ is a $C^1$ mapping from $V\subset\cc^n$ to $\cc^n$, we define Jacobean matrices $$g_z:=\left(\DD{g^\all(z)}{z^\beta}\right)_{1\leq \all,\beta\leq n}, \quad g_{\ov z}:= \left(\DD{g^\all(z)}{\bar z^\beta}\right)_{1\leq \all,\beta\leq n},\quad g_{(z,\ov z)}=(g_z,g_{\bar z}).
$$
\eppp
\subsection{Covering of coordinate charts for domains and norms}

 We will fix a finite open covering $\{U_1,\dots, U_{k_0}\}$ of a neighborhood  of $\overline M$. We may assume $\cL M=\cup U_j$.
For each $k$, we assume
\bppp
\item  there is a holomorphic coordinate map $\psi_k$ from   ${U_k}$ onto  $\hat U_k\Subset \cc^n$, and  $\psi_k$ is biholomorphic on an open set $\tilde U_k$ containing $\ov{U_k}$;
\item
$M_k:=U_k\cap M\neq\emptyset$,
   $\overline{U_k}\subset M$ if $\ov{U_k}\cap b M\neq\emptyset$, and $D_k:=\psi_k(U_k\cap M)$ are Lipschitz domains;
\item there is a partition of unity $\{\chi_k\}$ such that $\supp \chi_k\Subset U_k$, $\chi_k^{1/2}\in C^\infty(\cL M)$, 
     and $\sum\chi_k=1$ on a neighborhood $\cL U$ of $\overline M$, and set $\hat\chi_j=\chi_j\circ\psi_j^{-1}$;
\item  
for a subset  $A$ of $E=\rr^d$ or $\cc^n$, we use $\cL N_\del(A):=\{x\in E\colon \dist_A(x)<\del\}$ and 
  $\dist_A$ is the distance from $A$. 
\eppp
We will write $z_k=x_k+iy_k=\psi_k=(z_k^1,\dots, z_k^n)$ as local  holomorphic coordinates on $U_k$. We will write $\nabla^\ell$ as the set of partial derivatives of order $\ell$ in local real coordinates.
When $M'$ is a small $C^2$ perturbation of $M$, the above conditions are still satisfied for $M'$, without changing $\psi_k, U_k, \chi_k, \cL U$.


\medskip

We now use the local charts to define norms.

 Let $r=k+\beta$ with $k\in\nn$ and $0<\beta\leq1$. Let $D$ be a bounded Lipschitz domain in $\rr^d$. Let $\|u\|_{C^r(\ov D)}$  denote the standard H\"older norm for a function $u\in C^r(\ov D)$.   The  H\"older-Zygmund space $\Lambda^\beta( D)$ is the set of functions $ u\in C^0(\ov D)$ such that
$$
\| u\|_{\Lambda^\beta(D)}:=\sup_{x\in D}
| u(x)|+\sup_{x,x+h,x+2h\in D,h\neq0}|h|^{-\beta}|\Delta^2_hu(x)|<\infty.
$$
Here $\Del_hu(x)=u(x+h)-u(x)$ and $\Del^2_hu(x)= u(x+h)+ u(x+2h)-2u(x+h)$. When $r=k+\beta$ with $k\in\nn$ and $0<\beta\leq1$, define
$\|u\|_{\Lambda^r(D)}=\|u\|_{C^k(D)}+\|\pd^ku\|_{\Lambda^\beta(D)}$.

We also write $\|\cdot\|_{C^a(D)}$ as $\|\cdot\|_{D,a}$ and $\|\cdot\|_{\Lambda^a(D)}$ as $|\cdot|_{D,a}$.

 For discussions on various equivalent forms of the norm $|\cdot|_{D,r}$ on bounded Lipschitz domains $D$,  see   \ci{aq}*{sect.~5}. For instance, a theorem in~\cite{MR0521808}*{Thm.~18.5, p.~63} says that $u\in \Lambda^{k+\beta}(D)$ if and only if
 \eq{frrn}
\sup_{x\in D}
| u(x)|+\sup_{x,x+h,\dots,x+(k+2)h\in D,h\neq0}|h|^{-k-\beta}|\Delta^{k+2}_hu(x)|<\infty.
\eeq
Further, the latter and $|u|_{D,k+\beta}$ are equivalent norms.

For $r\in[0,\infty)$, we say that an $L^2$ function $f$ on $\rr^d$ belongs to $H^r(\rr^d)$ if
$$
\|f\|^2_{H^r(\rr^d)}:=\int_{\xi\in\rr^d}(1+|\xi|^2)^r|\hat f(\xi)|^2\, d\xi<\infty.
$$
If $D$ is a bounded  domain with Lipschitz boundary in $\rr^d$, let $H^r(D)$ be the set of functions $f\in L^2(D)$ such that
$$
\|f\|_{H^r(D)}:=\inf_{\tilde f|_D =f}\|\tilde f\|_{H^r(\rr^d)}<\infty.
$$

 Assume that $V$ is a holomorphic vector bundle on $\cL M$. We fix the coordinate charts so that $V$ over $U_k$ admits a holomorphic basis  
 $\{e_{k,1},\dots, e_{k,m}\}$. For the holomorphic tangent bundle $\Theta$ of $\cL M$, we always use $\DD{}{z_k^1},\dots, \DD{}{z_k^n}$ as its basis on $U_k$.

\begin{defn}\label{def-norm-f} When $u$ is a function on $\ov M$, define $u_k=u\circ\psi_k^{-1}$ and
$$
|u|_{M,r}:=\|u\|_{\Lambda^r(M)}:=\max_k|u_k|_{D_k,r},\quad \|{u}\|_{H^r(M)}:=\max_k\|{u_k}\|_{H^r(D_k)}.
$$
Let $\varphi$ be a section $V$-valued $(0,q)$ forms on $\ov M$ where $q=0,1,\dots, n$. Write
$
\varphi|_{U_k}=\sum_{j=1}^m\sum_{|I|=q} \varphi^j_{k,I}(z_k)(d\bar z_k)^I\otimes e_{k,j}(z_k)$.
Set
$$
|\var|_{M,r}:=\|{\varphi}\|_{\Lambda^r(M)}:=\max_{k,I,j}
|\varphi^j_{k,I}|_{D_k,r}.
$$
\end{defn}
To apply the Nash-Moser iteration,   we need to consider  domains $\widetilde M$ that vary  in $\cL M$. However, these domains will   be   small $C^2$ perturbations   of $M$. Thus, we can use the same coordinate charts $\{U_,\psi_k\}$   to define the norms $\znorm{\cdot}_{\widetilde M,r}$.
We will also need to define norms  of a diffeomorphism $F$ from  $M$ onto its image that is a small perturbation of the identity mapping $\id_{M}$.
Thus, we can define the following.
\begin{defn}\label{map-norm} Let $M,\widetilde M$ be domains in $\cL M$.
Let $F\colon M\to \widetilde M$ be a diffeomorphism. Set
$$
F_k(z_k):=\psi_k F\psi_k^{-1}(z_k)=z_k+f_k(z_k),
$$
where $f_k$ are sufficiently small.
Define
$
\znorm{F-\id}_{M,r}=\max_k\znorm{f_k}_{D_k;r}. $
    \end{defn}

\subsection{Convexity of norms and stability of constants}

For bounded Lipschitz domains, the following H\"older estimates for interpolation, product rule and chain rule  are well-known.

\pr{zprod}  Let $M$, $\widetilde M$ be  Lipschitz  domains in  $\cL M$.
Let $\ve>0$. Then for $ a>0$ and functions $u,v$ on $M$, 
\al
\znorm{uv}_{M,a}&\leq C_{a}(|u|_{M,a}\|v\|_{M,0}+\|u\|_{M,0}|v|_{M,a}),
\label{zprule}
\\
\label{zqrule}
\znorm{\f{1}{u}}_{M,a}&\leq C_{a,\ve}(1+|u^{-1}|_{M,\ve})^{a+2}(1+|u|^{a+1}_{M,\ve})\znorm{u}_{M,a}. \\ \intertext{Suppose that $g$ maps $M$ into $\widetilde M$ and $\norm{g-\id}_{M;1}<C(M)^{-1}$. Then for   $a>1$}
\|u\circ g\|_{M,1}&\leq C_\ve\|u\|_{\widetilde M,1}(1+\znorm{g}_{M,1+\ve});
\label{chain1}
\\
\label{zchain}
\znorm{u\circ g}_{M,a}&\leq C_{a,\ve}(\znorm{u}_{\widetilde M,a}(1+|g|^2_{M, 1+\ve})
  +\znorm{u}_{\widetilde M,1
+\ve}
\znorm{g}_{M,a}+\norm{u}_{\widetilde M,
0}).
\end{align}
Here, $C_a,C_{a,\ve}, C_\ve$ are stable under a small $C^1$ perturbation of $M$.
The estimates   hold for $\ve=0$ if all norms are replaced by the corresponding $\|\cdot\|$ norms.
\epr
\begin{proof}The case $\cL M=\cc^n$ is proved in~\cites{MR2829316}. The general case is reduced to local coordinate charts via Definition~\ref{map-norm}.
\end{proof}

\le{chain_rule}
Let $\{M^{\ell}\}_{\ell=0}^\infty$ be a sequence of Lipschitz domains in $\cL M$. Assume $F^\ell$ maps $M^\ell$ into $M^{\ell+1}$ and all $M^\ell$ are small perturbations of $M^0$. Suppose that
$\|F^\ell-id\|_{M^\ell,1}<\f{1}{C(M^0)\ell^2}$ for all $\ell\geq0$.
Then
for $\ve>0,  r>1$
\aln{}
\znorm{u\circ F^{m-1}\circ\cdots\circ F^0}_{M^0,r}
&\leq {C_{r,\ve}^{m+1}}\bigl\{\znorm{u}_{M^{m},r}+ \sum_i
 \znorm{u}_{M^{m},1+\ve}\znorm{f^i}_{M^i,r}\\
 &\qquad \nonumber
 +\znorm{u}_{M^{m},r}\znorm{f^i}_{M^i,1 +\ve} ^{2}      \bigr\}.
\end{align*}
 The estimates hold for $\ve=0$ and $r\geq0$ if all norms are replaced by the corresponding $\|\cdot\|$ norms.
\ele
\begin{proof}
When $\cL M$ is $\cc^n$, the lemma is in~\cites{MR2829316}. In general case,
we need to use Definition~\ref{map-norm}. Thus, we write
$$
u\circ F^m\circ\cdots\circ F^0\circ\psi_k^{-1}=\tilde u\circ \tilde F^m\circ\cdots\circ\tilde F^0
$$
with $\tilde F_k^j=\psi_k \circ F^j\circ\psi_k^{-1}$ and $\tilde u_k=u\circ\psi_k^{-1}$. Then the estimate follows from the definition and the lemma for $\cL M=\cc^n.$
    \end{proof}

\subsection{Smoothing operators}

We now modify the smoothing operator in Moser~\ci{MR0147741} as follows.

\le{goodsupp}Fix an integer $m_*\geq0$. Let $\omega$ be an open set in $\rr^d$. There is a smooth function $\mu$ on $\rr^d$ with compact support in $\omega$ such that $\int_{\rr^d} x^I\mu(x)\, dx=0$ for $I\in\nn^d$ with $0<|I|\leq m_*$ and $\int_{\rr^d} \mu(x)\, dx=1.$
\ele
\begin{proof} Consider $d=1$. Given an interval $(a,a+\ve)$.    Let us find smooth functions $\chi_0,\dots, \chi_{m_*}$ with compact support in $(a,a+\ve)$ such that
$$
 \int_\rr \sum_j c_j\chi_j(t)\, dt=1,\quad
 \int_\rr\sum_j c_j\chi_j(t)t^i\, dt=0, \quad i=1,\dots, m_*.
$$
It suffices to find $\chi_0,\dots, \chi_{m_*}$ such that
$$
A_{m_*}:=\left(
  \begin{array}{ccc}
   \int_\rr \chi_0(t)\, dt & \cdots &   \int_\rr \chi_{m_*}(t)\, dt \\
    \vdots & \cdots & \vdots \\
     \int_\rr t^{m_*}\chi_0(t)\, dt & \cdots  &  \int_\rr t^{m_*}\chi_{m_*}(t)\, dt \\
  \end{array}
\right)
$$
is non-singular.

The case ${m_*}=0$ is trivial. Suppose   we have found $\chi_0,\dots, \chi_{{m_*}-1}$ so that
$A_{{m_*}-1}
$
is non-singular.
By row reduction, we find a matrix $B_{m_*}$ such that $B_{m_*}A_{m_*}=(c_{ij})_{0\leq i,j\leq {m_*}}$ is an upper triangular matrix with all diagonal entries being non-zero except for possible the last one:
$$
c_{{m_*}{m_*}}=\int_\rr \chi_{m_*}(t)p(t)\, dt,\qquad p(t)=t^{m_*}-b_0-\cdots-b_{{m_*}-1}t^{{m_*}-1},
$$
where $b_i$ depend only on $\int_\rr t^i\chi_j(t)\, dt$ for $i=0,\dots, {m_*}$ and $j<{m_*}$.
Since  $p$ is not identically zero on $(a,a+\ve)$, we can find a desired $\chi_{m_*}$. Now $\mu(t)=\sum_j c_j\chi_j(t)$ satisfies the condition when $d=1$.

For $d>1$, we can take $\mu(x)=\mu_1(x_1)\cdots\mu_d(x_d)$ for suitable $\mu_1,\dots, \mu_d$.
\end{proof}

We now prove additional properties for the Moser smooth operator. A useful feature of the smoothing operator is that it extends functions form $D$ to a fixed and slightly larger domain.
\begin{lemma} \label{Lem:StD}
Fix a positive integer $m_*$.
Let $D\Subset\rr^d$ be a Lipschitz domain.
 There exist operators
 $$
 S^t: C^0(\overline{D}) \to C^\infty(\cL N_t({D})), \quad t\in (0,1]
 $$
 such that
\begin{align}\label{StNt}
 | S^t u |_{\cL N_t({D}), r} &\leq C_{r,s} t^{s-r} | u |_{D, s},  \quad 0< s \leq r < \infty;\\
 | (I - S^t)E u |_{\cL N_t({D}), s} &\leq C_{r,s} t^{r-s} | u
|_{D,r}, \quad  0 < s \leq r +m_*;
\label{IStu}\\
\label{nSt}
    \left| [\nabla, S^t] u \right|_{\cL N_t({D}),s} &\leq C_{r,s} t^{r-s} | u |_{D,r}, \quad r>0,s>0, |r-s|<m_*.
\end{align}
Here the constants $C_{r,s}$ depend only on $r,s,m_*$ and the Lipschitz norm of $D$.
Further, when $u\in\Lambda^r(D)$ with $0<r\leq1$, the $[\nabla, S^t]u$ is defined by
\eq{def-comm}
[\nabla, S^t] u=\lim_{\ve\to0^+}[\nabla,S^t]S^\ve u.
\eeq 
The  same estimates hold for the H\"older norms including $r=0$.
\end{lemma}

\begin{proof}We will use a similar construction in~\ci{shiNN}.
We cover $\ov D$ by finitely many open balls $B_j$ where $\{B_j\}$ is a stable covering for a small perturbation on $D$, namely $B_j\Subset D$ if $B_j\cap \pd D\neq\emptyset$. Let $\{\chi_j\}$ be a partition of unity subordinate to $\{B_j\}$. Thus, $\sum\chi_j=1$ on $\cL U$
where $\cL U$ is a fixed neighborhood of $\ov D.$

Let $E\colon C^0(\ov D)\to C_0^0(\rr^d)$ be the Stein extension operator~\ci{MR0290095}. Note that $E$ is linear and
$$
|Ef|_{\rr^d,a}\leq C_a(D)|f|_{D,a},\  a>0; \quad \|Ef\|_{\rr^d,a}\leq C_a(D)\|f\|_{D,a},\  a\geq0.
$$
Let $f_j=\chi_jf$, which is well-defined on $\ov D$. Note that   $f_j$ vanishes on $\cL N_{c_1}(\ov D\setminus B_j)$ for some $c_1>0$.   We have
$$
\text{$\supp\chi_j\Subset D$\quad  or\quad  $B_j\cap \partial D\neq\emptyset$}.
$$
Thus we can find an open set $\omega_j$  such that
\eq{Nts}
\cL N_t(B_j\cap D)-t \omega_j\subset \ov D, \quad \forall t\in[0,1].
\eeq
Fix a compact set $K_j$ such that its interior $K_j^{o}$ satisfies
\eq{non-empty}
\emptyset\neq K_j^o\subset K_j\subset \om_j.
\eeq
By \rl{goodsupp}, we find a smooth function $\mu_j$ on $\rr^d$
such that $\int y^I\mu_j(y)\, dy=0$ for $|I|>0$, $\int\mu_j(y)\, dy=1$ and
\eq{supp-chi}
\supp \mu _j 
\subset  K_j^o.
\eeq
By \re{Nts}-\re{supp-chi}, for $ t\in[0,1]$
\eq{skt}
S_j^tf_j(x)=\int_{y\in\om_j}f_j(x-ty)\mu_j(y)\, dy
\eeq
is a well-defined smooth function on $
\cL N_t 
(B_j\cap D)$. Now by adjusting $\om_j$ if necessary we have $S_j^tf_j=0$ on $\cL N_{c_1}(\ov D\cap \pd B_j)$ for some $c_1>0$. Define $S_j^t 
 f_j 
=0$ on $\cL N_t(D\setminus B_j)$. Then $S^t_jf_j$ is smooth on $\cL N_t(D)$. We still have \re{skt} and  
\eq{dxj}[\nabla,S_j^t]f_j=0.\eeq

Let $E_jf=\chi_jEf$. By \re{Nts} and \re{skt},
$
S_j^tf_j=S_j^tE_jf$   on $\cL N_t(D)$.
In other words, on $\cL N_t(D)$,  $S_j^tE_jf$ is independent of $E_j$.

By properties of Moser's smoothing operator, we have
\gan{}
  | S^t_jE_jf |_{\cL N_t(D), r} \leq C_{r,s}t^{s-r} | E_jf |_{D_j, s}\leq C_{r,s}'t^{s-r} | f |_{D, s},  \  0< s \leq r < \infty,\\
 | (I - S^t_j) E_jf|_{\cL N_t(D), r} \leq  C_{r,s} t^{s-r} | E_jf
|_{D_j,s} \leq C_{r,s}' t^{s-r} | f
|_{D,s}, \   0 < r \leq s +m_*.
\end{gather*}
Note that these properties are usually stated for H\"older norms; via interpolation they hold for Zygmund norms too.
 The reader is referred to \ci{aq}*{sect.~5}, for instance, on the equivalence of norms when interpolation is applied.

On $\cL N_t({D})$, we  define
\eq{sm_op_vf}
S^tf:=\sum S_j^t(\chi_jf).
\eeq
We obtain \re{chain_rule}-\re{IStu} immediately from the estimates for $S^t_j$ and the inequality
$$
|\chi_jf|_{r}\leq C_r|f|_r
$$
via \re{zprule}.
We now prove \re{nSt}. We first consider the case that $r>1$. 
We have
$
S^t(\nabla  f )= \sum S_j^t(\chi_j \nabla   f).
$
On $\cL N_t({D})$, by \re{dxj}
$$
 \nabla    S^tf=  \sum S_j^t(\nabla   (\chi_jf) )=S^t\nabla   f+\sum S_j^t( f\nabla   \chi_j).
$$
To finish the proof, we use the crucial fact that
 $S^t_j(f\nabla   \chi_j)=S^t_j(Ef\nabla   \chi_j)$ by \re{Nts}. Therefore,
  \al{}\nonumber
  [ \nabla,   S^t]f &= \sum(S_j^t-\id)(Ef\nabla   \chi_j) +Ef \sum \nabla   \chi_j.
\end{align}
The last sum vanishes since $    \sum_j\chi_j=1$ on $\cL N_t(D)$. Now \re{nSt} for $r>1$ follows from \re{StNt}-\re{IStu}.

Consider now $0<r\leq1$. 
 By \re{IStu}, the $S^\ve u$ converges to $u$ in $C^0(\ov D)$ norm as $\ve\to0$. Since \re{nSt} holds when $u$ is replaced by $S^\ve u$, 
we have
$$
|h|^{-r}|\Del_h^2u(x)|=\lim_{\ve\to0}|h|^{-r}|\Del_hS^\ve u(x)|\leq C_{r,s}t^{s-r}\liminf_{\ve\to0}|S^\ve u|_s\leq C_{r,s}t^{s-r}|u|_s 
$$
for any $x,x+h,x+2h\in D$ and $h\neq0$. This verifies   \re{nSt} via equivalent norms. 
\end{proof}

Next, we extend the above definition and properties of $S^t$  for sections of holomorphic vector bundles on $M$. For our use, we will replace $\nabla$ by the global operator $\db$.

Recall that $\{U_k\}$ is an open covering of $\ov  M$ and
$z_k=\psi_k\colon U_k\to\hat U_k$ are coordinate maps. Let $\{\chi_k\}$ be a partition of unity subordinated to $\{U_k\}$ such that $\sqrt{\chi_k}\in C^\infty(\cL M)$,  $\sum\chi_k=1$ in a neighborhood $\cL U$ of $ M$ and $\supp\chi_k\Subset U_k$ and let $\hat\chi_k=\chi_k\circ\psi_k^{-1}$.  Let $D_k=\psi_k(M_k)$.

Set $M_t=\cup\psi_k^{-1}(\cL N_t(D_k))$ for $t>0$, which is a neighborhood of $\ov M$ in $\cL M$.
\begin{defn}\label{tvef} Let $V$ be a holomorphic vector bundle on $\cL M$ with holomorphic basis $\{e_{k,1},\dots, e_{k,m}\}$ on $U_k$.  Suppose that $v=\sum_{j,I}  v^j_{k,I}(z_k)  d\bar z_k^I\otimes e_{k,j}(z_k)$ is in $C_{0,q}^0(\ov M,V)$. Define an extension  $Ev \in C^0_{(0,q)}(\cL M,V)$ by
$$
E v:=\sum_{k}\sum_{|I|=q}\sum_{j=1}^m \tilde v^j_{k,I}(z_k)  d\bar z_k^I\otimes e_{k,j}(z_k),\quad |\tilde v^j_{k,I}|_{D_k;r}\leq C_r|v|_{M,r}.
$$
Here  $\tilde v^j_{k,I}(z_k)=\chi_k^{1/2}
E_{D_k}(\hat\chi_k^{1/2}v_{k,I}^j)(z_k)$  on $\cL M$ (and $\tilde v^j_{k,I}=0$ on $\cL M\setminus U_k$). 
\end{defn} 
  Let $S_k^t$ be a smoothing operator for functions on $D_k=\psi_k(U_k\cap M)$ satisfying \rl{Lem:StD}. Then $(S_{k}^t(\hat\chi_k\tilde v^j_{k,I}))(z_k)d \bar z_k^I\otimes e_{k,j}$ is  smooth on $U_k\cap\cL N_t(M)$ and vanishes on $\cL N_{c_2}(\cL N_t(M)\cap\pd U_k$ for some $c_2>0$.
   Hence, it is well-defined and smooth on $\cL N_t(M)$ by setting it to be zero on $\cL N_t(M))\setminus U_k\cap\cL N_t(M)$.
  \begin{defn}
Let $v\in C^0_{(0,q)}(\ov M,V)$. Define
  on
 $\cL N_t(M)$ 
 \eq{St*+}
S^t v=\sum_{k}\sum_{|I|=q}\sum_{j=1}^{m}\chi_k^{1/2}
 \Bigl(S_{k}^t(\hat\chi_k^{1/2} v^j_{k,I})\Bigr)(z_k)d \bar z_k^I\otimes e_{k,j}.
\eeq
Here  $\chi_k^{1/2}
 \Bigl(S_{k}^t(\hat\chi_k^{1/2}\hat v^j_{k,I})\Bigr)(z_k)$, vanishing on $\cL N_t(M)\setminus U_k$, are smooth on $\cL N_t(M)$. \end{defn}

Note that on $\cL N_t(M)$, the  $S^tE v=S^tv$ depends only on the values of $v$ on $M$. Therefore, we can express
\eq{Sttv}
S^t\tilde v-\tilde v=\sum_{k}\sum_{j,I}\Bigl (\hat\chi_k^{1/2}(S_{k}^t-\id)E_{D_k}(\hat\chi_k^{1/2} v^j_{k,I})\Bigr)(z_k)d \bar z_k^I\otimes e_{k,j}.
\eeq

\begin{lemma} \label{StM}
Let $S^t $ be given by \rea{St*+}.
Then
$$
S^t\colon C_{0,q}(\ov M;V)\to C_{0,q}^\infty(\cL N_t(M);V).
$$
For all $v \in \La^s(\ov M)$, the following hold
\al{}\label{Stvn}
 | S^tv |_{\cL N_t({M}), r} &\leq C_{r,s} t^{s-r} | v |_{M, s},
 \quad 0< s \leq r < \infty,\\
 | (I - S^t)E v|_{\cL N_t({M}), r}& \leq C_{r,s} t^{s-r} | v
|_{M,s}, \quad  0 < r \leq s +m_*,
\label{IStv}\\
\label{dbS}
  \left| [\db,S^t] 
   v\right|_{\cL N_t({M}),s}& \leq  C_{r,s} t^{r-s} |v|_{M,r}, \quad 0 < s \leq r < s+m_*.
\end{align}
Here $C_{r,s}$ are stable under smaller perturbations of Lipschitz domains of $M$, and   $[\db,S^t]v:=\lim_{\ve\to 0^+}[\db,S^t]S^\ve v$ when $s\in(0,1]$.
\end{lemma}

\begin{proof}
The first two inequalities follows from \re{St*+}, \re{Sttv}, and  the previous lemma immediately.

We now estimate $[\db, S^t]$. 
Recall that    on $  \cL N_t(M)\cap U_k$, $S^t(\db v)$ depends only on the value of $\db v$ on $M$.
Let $v= \sum_{j,I}   v_{k,I}^j(z_k)  d\bar  z_k^I\otimes e_{k, j}$. We have
$$
\db v= \sum_{|I|=q}\sum_{j=1}^m \db v_{k,I}^j(z_k)  \wedge d\bar  z_k^I\otimes e_{k, j}.
$$
Hence, \begin{align}\nonumber
S^t(\db v  )&=S^t(\db v  ) =  \sum_{k,j,I}\chi_k^{1/2} \Bigl(S_{k}^t( \hat\chi_k^{1/2}\db v^ j _{k,I}  ) \Bigr)(z_k)\wedge  d\bar  z_k^I\otimes e_{k, j}.
\end{align}
 On $U_k\cap \cL N_t(M)$, we also have $\db S^{t}_{k}=S^t_{k}\db$ on $\cL N_t(D_k)\cap \supp\hat\chi_k$. Therefore, on $  \cL N_t(M)$  we can express
\begin{align} 
  \db S^tv&=    \sum_{k,j,I}\db\Bigl(\chi_k^{1/2} S_{k}^t (\hat\chi_k^{1/2}v_{k,I}^ j )\Bigr)(z_k)  \wedge d\bar  z_k^I\otimes e_{k,j}\nonumber\\
&=\sum_{k,j,I}\db\chi_k^{1/2}\wedge\Bigl( S_{k}^t (\hat\chi_k^{1/2}v_{k,I}^ j )\Bigr)(z_k)  d\bar  z_k^I\otimes e_{k,j}\nonumber\\
 &\quad +   \sum_{k,j,I}\chi_k^{1/2}\Bigl( S_{k}^t\bigl( \db(\hat\chi_k^{1/2}v_{k,I}^ j )\bigr)\Bigr)(z_k) \wedge d\bar  z_k^I\otimes e_{k,j}\nonumber. 
\end{align}
Thus
\begin{align}\nonumber
 [\db, S^t]
 v&=\sum_{k,j,I}\db\chi_k^{1/2}\wedge\Bigl( S_{k}^t (\hat\chi_k^{1/2}v_{k,I}^ j )\Bigr)(z_k)  d\bar  z_k^I\otimes e_{k,j}\nonumber\\
 &\quad +   \sum_{k,j,I}\chi_k^{1/2}\Bigl( S_{k}^t\bigl(v_{k,I}^ j \db\hat\chi_k^{1/2})\bigr)\Bigr)(z_k) \wedge d\bar  z_k^I\otimes e_{k,j}\nonumber\\
 \nonumber
 & =\sum_{k,j,I}\db \chi_k^{1/2}\wedge\Bigl((S_{k}^t-I)\bigl( \hat\chi_k^{1/2}v_{k,I}^ j \bigr)\Bigr)(z_k)  d\bar  z_k^I\otimes e_{k, j}\\
 &\quad +  \sum_{k,j,I}\chi_k^{1/2}\Bigl(( S_{k}^t-I)\bigl(v_{k,I}^ j \db\hat\chi_k^{1/2}\bigr)\Bigr)(z_k) \wedge d\bar  z_k^I\otimes e_{k,j}\nonumber \\
 &\quad +   \sum_{k,j,I}\Bigl(\bigl( 2\hat\chi_k^{1/2}  \db\hat\chi_k^{1/2}\bigr)
 (v_{k,I}^ j)(z_k) \Bigr)\wedge d\bar  z_k^I\otimes e_{k, j}.\nonumber
\end{align}
The last summation  is zero, as it equals   $\sum_{k}\bigl( 2\chi_k^{1/2}  \db\chi_k^{1/2}\bigr)
\wedge v$ and $\sum\chi_k=1$.
Now the lemma follows from \re{IStu} and the smooth approximation of $v$ via $S^\ve v$ for \re{dbS} when $s\in(0,1]$.
\end{proof}

\section{Estimates of the flows}\label{sect:flow}

In this section, we describe  a formally integrable complex structure that has finite distance to the complex structure of $\cL M$ in a $C^1$ section of $\Theta$-valued $(0,1)$-form on $\ov M$.
We show  the time-one mapping $F^1$ of the smoothing $S^tv$ of a vector filed $v$ defined on $M$ can be estimated on the same domain $M$.
We also derive estimates for $\db F^1$ in terms of $\db v$ for non-holomorphic $v$.

\subsection{Almost complex structures in coordinate charts}\setcounter{thm}{0}\setcounter{equation}{0}

An almost complex structure $X$ on $\ov M$ that has a {\it finite} distance to the standard complex structure on $\cL M$
is given in local coordinates $z_k$ by 
$$
X_{k,\ov\all}=\DD{}{\ov z_k^\all}+\sum_\beta A_{k,\ov\all}^\beta(z_k)\DD{}{ z_k^\beta},\quad\all=1,\dots, n.
$$
By definition, the span of $\{X_{k,\ov\all}\}$ is the same as the span of $\{X_{j,\ov\all}\}$ on $U_k\cap U_j$.   Thus $A:=A^\all_{k,\ov\beta}(z_k) d\ov z_k^\beta\otimes \DD{}{z_k^\all}\in C_{(0,1)}^1(\ov M, \Theta)
$ is a global section. 
We have
    \[
  \db A =  \sum_{\beta < \gm} \Bigl( \DD{}{\ov z_k^\beta }
A^\all_{k,\ov \gm} - \DD{}{\ov z_k^\gamma }  A^\all_{k,\ov \beta}
 \Bigr) d \ov z_k^\beta \we d \ov z_k^\gm \otimes \DD{}{ z_k^\all}.
    \]
The formal integrability means that $[X_\all,X_\beta]$ is in the span of $X_{k,\ov 1},\dots, X_{k,\ov n}$. One can check that equivalently, we have $[X_{k,\ov\all},X_{k,\ov\beta}]=0$, i.e.  
$ 
\db A=-\f{1}{2}[ A,A],
$ 
where in local coordinates the Lie bracket is defined by
\gan{}
[ A , A ]:=\sum_{\all,\beta,\gamma}Q^\all_{k,\ov\beta\ov\gamma}( A ) d\ov z_k^\beta\wedge d\ov z_k^\gamma\otimes \DD{}{ z_k^\all}, \\
 Q^\all_{k,\ov\beta\ov\gamma}( A ):=\sum_\tau
 A ^\tau_{k,\ov\beta}(z_k) \DD{ A ^\all_{\ov\gamma}}{ z_k^\tau}- A ^\tau_{k,\ov\gamma}(z_k) \DD{ A ^\all_{k,\ov\beta}}{ z_k^\tau}.
\end{gather*}
The standard complex structure on $M$  corresponds to $ A =0$.

We will use the following lemma in Webster \cite{MR999729} showing how an almost complex structure changes with respect to a transformation of the form $F=\id+f$ where $I$ is the identity mapping. It  provides a useful formula for $A'$ in terms of $A$ and $F$.

We will   use matrix $A_k=(A_{k,\ov\beta}^\all)_{1\leq\all,\beta\leq n}.$ Denote by $I$ the $n\times n$ identity matrix. 
\begin{lemma}\label{elem_pp}
Let $M$ be a relatively compact domain  in $\cL M$ with $C^2$ boundary and let $M'\Subset\cL M$ be a small $C^2$ perturbation of $M$.
Let $\{X_{\overline\a}\}_{\a=1}^n$ be a $C^1$ almost complex structure on an open set $\ov M$ that has finite distance to the standard complex structure on $M$. Let $F$ be a $C^1$ mapping from $\ov M$ onto $\ov M'$. Let $F_k=\psi_kF\psi_k^{-1}=\id +f_k$.  There is a constant $\del$ depends on $\|A\|_{0}$ such that if
$$\|F-\id\|_{M,1}<\delta,$$ then $F$ is a $C^1$ diffeomorphism   and $(F_k)_*X_{k,\ov\all}$ are in the span of
$$
X'_{k,\ov\all}=\DD{}{\ov z_k^\all}+\sum_\beta{A'}_{k,\ov\all}^\beta(z_k)\DD{}{ z_k^\beta},\quad\all=1,\dots, n.$$
Furthermore,
 $
X_{k,\overline\a}F_k^\b=(X_{k,\overline\a}F_k^{\overline\g})({A'}_{k,\overline\g}^\b\circ F_k).
$ 
Equivalently, as matrices,
\eq{AdfA}
A_k(z_k)+ ( f_k)_{\bar z_k}+A_k(z_k) (f_k)_{z_k}=(I+(\bar f_k)_{\bar z_k}+A_k(z_k) (\bar f_k)_{z_k})A_k'\circ F_k(z_k).
\eeq
\end{lemma}

Therefore, a $C^2$ diffeomorphism from $M$ onto $M'$ that is close to the identity transforms $A\in C^{1}_{(0,1)}(M,\Theta)$ into $A'\in C^{1}_{(0,1)}(M',\Theta)$ that still satisfies $\db A'=-\f{1}{2}[A',A']$ on $M'$. Thus, the solution to the global Newlander-Nirenberg theorem is meant to find an $F$ with a desired regularity such that $F$ transforms $A$ into $A'=0$, according to the relation \re{AdfA}.

\subsection{Time-one maps of vector fields of type $(1,0)$}
The advantage of expressing an almost complex structure via a global section $A\in C^1_{(0,r)}(M,\Theta)$ is to use the time-one map of a global section $v\in C^1(M,\Theta)$ to transform the structure. This however requires the vector fields to be defined on a domain slightly larger than $M$. 

We will study the flows of a type $(1,0)$ vector field $v$ on $\cL M$, i.e. the $v$ that has  the form
 $v=\sum v_k^\all(z_k)\f{\pd}{\pd z_k^\all}$, and let $F^\la $ be its flow. We will also estimate $\db F^\la$ in terms of $\db v$. Define
$$
\psi_kF^\la \psi_k^{-1}(z_k)=z_k+f^\la _k(z_k), \quad f_k^\la (z_k)=\la v_k(z_k)+w_k (z_k,\la).
$$

Recall that $\{U_k\}$ covers a neighborhood of $\ov M$. In the following, let us denote
$$
  |v|_{\cL N_a(M),r}=\max_{k,\all} |v_k^\all|_{\cL N_a(D_k),r}.
$$
Define $ 
\|v\|_{\cL N_a(M),r}$ analogously. Define
$$
|v|_r=\max |v_k|_{U_k,r},  \quad \|v\|_r=\max \|v_k\|_{U_k,r}.
$$

When we smooth a vector field $v$ on $M$, $S^tv$ is defined on a slightly large domain $\cL N_t(M)$. This allows us to estimate the time-one map of the smoothed vector field without shrinking the original domain.
 More precisely, we have the following.
\pr{Lem:flow}
Let $M$ be a relatively compact Lipschitz domain in $\cL M$.   Let $v=\sum v_k^\all(z_k)\f{\pd}{\pd z_k^\all}$ be a $C^1$ vector field defined on $\cL M$ and let $F^\la$ be its flow. There are positive constants $C_n,\del_0,\del_1,\delta_1', \del_2$ such that if $\|v\|_{1}\leq \delta_1 $, then the flow $F^\la$ of $v$ is well-defined on $\cL N_{\delta_0 }(M)$ for $\la\in[-1,1]$ and
\ga\label{Fla}
  F^\la (\cL N_\delta(M))\subset \cL N_{\delta+C_n\|v\|_0}(M), \quad 0\leq \delta\leq \delta_0 , \quad  \la\in[-1,1].\end{gather}
In particular, when $0\leq \del\leq \del_0$ and $C_n\|v\|_0\leq\del_0$,
\gan{}
M^1:=F^1(M)\subset\cL N_{C_n\|v\|_0}(M), \quad F^{-1}(M^1)\subset\cL N_{C_n(C_n+1)\|v\|_0}(M).
\end{gather*}   Further, for fixed $\ve>0$,   $w_k:=F^\la _k-\id-\la v_k$ satisfy the following

\bpp\item If  $|v|_{1+\ve}<\del_{1,\ve}'$ and $r>1$, then
\al \label{wkL}
|w_k(\cdot,\la)|_{\cL N_{\delta}({D_k}),r}& \leq C_{r,\ve}(M)  |v|^*_{ r},
\\ \label{wkL+}
|w_k(\cdot,\la)|_{\cL N_{\delta}({D_k}),r}& \leq C_{r,\ve}(M)  |v|^*_{ r+1} |v|^*_{1+\ve}.
\end{align}
\item If $|v|_{2+\ve}<\del_{2,\ve}$ and $r>1$, then
\al{}
\label{dbwr}
| \db w_k(\cdot,\la) |_{\cL N_{\delta}(D_k),r}&\leq C_{r,\ve}  |\db v|^*_{r+1}|v|^*_{1+\ve}+ C_{r,\ve} (M)|\db v|^*_{1+\ve}|v|^*_{r+1}.
\end{align}
\epp
Here $
|v|^*_{ r}=|v|_{\cL N_{\del+C_n\|v\|_{0}}(M), r}$. Further, $C_n$ depends only on $n$, and  $C_{r,\ve}(M),\del_0$, $ \delta_1$,$\del_{1,\ve}'$,$\del_{2,\ve}$ are stable under small perturbations of $ b M$.
The estimates \rea{wkL}-\rea{dbwr} hold for $\ve=0$ if all $|\cdot|$ norms are replaced by the corresponding $\|\cdot\|$ norms.
\end{prop}
\begin{proof}
The uniqueness and the local existence of $C^1$ flows for  $C^1$ vector fields is a standard theorem; see~\ci{MR1476913}*{Thm.~4.2}. We need to verify \re{Fla} and the rest of estimates.

Recall that $z_k$ maps $U_k, U_k'$ onto $\hat U_k,\hat U_k'$ respectively. In the coordinate chart $(U_k,z_k)$, let $v_k=(v_k^1,\dots,v_k^n)$ be the components of $v$.
Write $F^\la$ as
$$
F_k^\la(z_k)=z_k+\la v_k(z_k)+w_k(z_k,\la ).
$$
It suffices to solve for $\la \in[-1,1]$
\begin{equation}\label{e_fix_pt} w_k(z_k,\la ) = (T_kw_k)(z_k,\la ),
\end{equation}
where $T_k$ is the map defined by
$$
  (T_k w_k )(z_k,\la ):=\int_0^\la \{ v_k(z_k + sv_k(z_k)+
w_k (z_k,s))- v_k(z_k)\}\, ds.
$$
Let $B^0$ be the set of families $\{w_k\}$, where $w_k$ is a $C^0$ mapping from $\ov {\cL N_\del(D_k)}\times[-1,1]$ into the polydisc $\ov{\Del^n_{2a}}$  with $a=\|v_k\|_{U_k;0}$.
On $B^0$, we use the metric
\[
 d(\tilde w_k,w_k)=\|\tilde w_k-w_k\|_{0}.
\]
Here and in what follows, we denote $\|w_k\|_{  {\cL N_\delta(D_k)}\times[0,1],r}, |w_k|_{  {\cL N_\delta(D_k)}\times[0,1],r}$ by $\|w_k\|_{ r}$ and $|w_k|_r$ respectively.

We now choose $C$ so large that $z_k+s\tilde v_k+\tilde w_k\in\cL N_{\delta+C\|v_k\|_0}(D_k)$, when
$$
z_k\in \cL N_{\del}(D_k),\quad \tilde v_k \in \Delta_a^n,\quad\tilde  w_k \in \Del_{2a}^n.
$$
By the contraction mapping theorem, when $\|v_k\|_{\hat U_k,1}<\delta_1$, there is a unique solution $w_k\in B^0$ satisfying \re{e_fix_pt}. Thus, $w_k$ satisfies
\eq{ezt}
  w_k(z_k,\la ) = \int_0^\la  v_k(z_k+sv_k(z_k) + w_k(z_k,s)) - v_k(z_k)
  \, ds.
\eeq
   To simplify notation, we drop subscript $k$ in $w,v$. We write $w_z$ and $w_{\bar z}$ for  the Jacobean matrices in $z$ and $\bar z$, respectively.
Write $\bar v(z)=\overline{v(z)}$. We also write $|w(\cdot,\la)|_{\cL N_\del(D_k),a}$ as $|w(\cdot,\la)|_{a}$.
 Differentiating \re{ezt}, we get
\al{}\label{ezt+}
 w_z(z,\la )&
  = \int_0^\la v_z(z+sv(z) + w(z,s))
 -v_z(z)
  \, ds\\
  &\quad +\int_0^\la v_z(z+sv(z) + w(z,s)) (sv_z(z)+ w_z(z,s))\nonumber
  \, ds\\
  &\quad +\int_0^\la v_{\bar z}(z+sv(z) + w(z,s)) (s\bar v_{z}(z)+ \bar w_{z}(z,s))
  \, ds. \nonumber\end{align}
Similarly, for the $\bar z$ derivatives, we have
  \al{}
\label{ezt++}w_{\bar z}(z,\la )&
  = \int_0^\la v_{\bar z}(z+sv(z) + w(z,s))
 - v_{\bar z}(z)
  \, ds\\
  &\quad +\int_0^\la  v_{z}(z+sv(z) + w(z,s)) (sv_{\bar z}(z)+   w_{\bar z}(z,s))
  \, ds\nonumber
\\
  &\quad +\int_0^\la v_{\bar z}(z+sv(z) + w(z,s)) (s\bar v_{\bar z}(z)+\bar w_{\bar z}(z,s))
  \, ds\nonumber\\
  &=I_1+I_2+I_3.\nonumber
\end{align}
Applying the contraction mapping again, one can show that when $v\in\Lambda^r$ with $r>1$, $w_{(z,\ov z)}(\cdot,\la)$ are in $\Lambda^{r-1}( {\cL N_\del(D_k)}
)$ when $\lambda$ is fixed.

We will first derive estimate \re{wkL}.
Suppose $r>1$. By \re{ezt} and \re{zchain}, we get
$$
|w(\cdot,\la)|_r\leq C_{r,\ve}\{ |v|^*_r(1+|v|^*_{1+\ve}+\sup_{|s|\leq |\la|}|w(\cdot,s)|_{1+\ve})^{1+2\ve}+|v|_{1+\ve}^*(|v|^*_{r}+\sup_{|s|\leq |\la|}|w(\cdot,s)|_{r})\}.
$$
Using  
$\sup_{|s|\leq 1}|w(\cdot,s)|_{1+\ve}\leq C_\ve' |v|^*_{1+\ve}\leq C_\ve\del_{1,\ve}'<1/2,$
we solve  for $\sup_{|s|\leq 1}|w(\cdot,s)|_{r}$. This yields \re{wkL} immediately.

We now use  \re{wkL} to  derive better estimates.
Let us rewrite   \re{ezt} as
$$
w(z,\la):=\int_0^\la\int_0^{1} v_{(z,\bar z)} (z+s_1sv(z) +s_1 w(z,s)(sv(z)+w(z,s),s\bar v(z)+\bar w(z))^T\, ds_1
   ds.
$$
For $r>1$, we have
\aln{}
|w(\cdot,\la)|_{r}&\leq C_{r,\ve'}|v|^*_{1+\ve'}(|v|^*_{r}+\sup_{|s|\leq|\la|} |w(\cdot,s)|_{r})+C_{r,\ve'}(|v|^*_{\ve'}+\sup_{|s|\leq|\la|}|w(\cdot,s)|_{\ve'})\\
&\qquad\times ( |v|^*_{r+1}+|v|^*_{2+\ve'}(1+|v|^*_{r}+\sup_{|s|\leq|\la|} |w(\cdot,s)|_{r}))\\
&\leq C_{r,\ve'}|v|^*_{1+\ve'}|v|^*_{r}+C_{r,\ve'}'|v|^*_{1+\ve'} ( |v|^*_{r+1}+|v|^*_{2+\ve'}(1+|v|^*_{r})).
\end{align*}
Thus, we obtain $|v|^*_{2+\ve'}|v|^*_{r}\leq  C_{r,\ve'}\|v\|^*_{r+1-\ve'}|v|^*_{1+2\ve'}$. Taking $\ve'=\ve/2$, we get \re{wkL+}.  

To verify \re{dbwr}, we write the first integral in \re{ezt++} as
$$
I_1(z,\la):=\int_0^\la\int_0^{1}( v_{\bar z})_{(z,\ov z)}(z+s_1sv(z) +s_1 w(z,s))(sv(z)+w(z,s),s\bar v(z)+\bar w(z))^T\, ds_1
   ds.
$$
Using \re{zchain} and $|v|^*_\ve+ \sup_{|s|\leq|\la|} |w(\cdot,s)|_{\ve}<c$, we obtain
\aln
|I_1(\cdot,\la)|_{r}&\leq C_{r,\ve'}(|\db v|^*_{r+1}+|\db v|^*_{2+\ve'})(|v|^*_\ve+ \sup_{|s|\leq|\la|} |w(\cdot,s)|_{\ve})\\
&\quad +C_{r,\ve}|\db v|^*_{1+\ve}(|v|^*_r+\sup_{|s|\leq|\la|} |w(\cdot,s)|_{r})\\
&\leq C_{r,\ve'}'(|\db v|^*_{r+1}+|\db v|^*_{2+\ve'}|v|^*_r)|v|^*_{1+\ve'}  +C_{r,\ve'}|\db v|^*_{1+\ve'}|v|^*_r\\
&\leq C_{r,\ve'}''|\db v|^*_{r+1}|v|^*_{1+\ve'}+C_{r,\ve'}''|\db v|^*_{1+2\ve'}|v|^*_{r+1-\ve'}|v|_{1+\ve'}  +C_{r,\ve'}|\db v|^*_{1+\ve'}|v|^*_r.
\end{align*}
Here we have used $|\db v|^*_{2+\ve'}|v|^*_r\leq C_{r,\ve'}\|\db v\|^*_{r+1-\ve'}|v|^*_{1+2\ve'}+C_{r,\ve'}\|\db v\|^*_{1+2\ve'}|v|^*_{r+1-\ve'}$.
Using \re{zchain}, $|v|^*_{1+\ve}\leq \del_{1,\ve}'$, and $\sup_{|s|\leq 1}|w(\cdot,s)|_{1+\ve}<C$, we get
\aln{}
&|I_2(\cdot,\la)|_r\leq C_{r,\ve}(|v|^*_{r+1}+|v|^*_{2+\ve'}(1+2|v|_r^*))(|\db v|^*_\ve+\sup_{|s|\leq|\la|}|\db w(\cdot,s)|_\ve)\\
&\qquad +C_{r,\ve}|v|^*_{1+\ve}(1+|v|_{1+\ve}+\sup_{|s|\leq|\la|}|\db w(\cdot,s)|_{1+\ve})(|\db v|^*_r+\sup_{|s|\leq|\la|}|\db w(\cdot,s)|_r),\\
&|I_3(\cdot,\la)|_r\leq C_{r,\ve}(|\db v|^*_{r}+|\db v|^*_{1+\ve'}(1+2|v|_r^*))(|v|^*_{1+\ve}+\sup_{|s|\leq|\la|}|\db w(\cdot,s)|_{1+\ve})\\
&\qquad +C_{r,\ve}|\db v|^*_{\ve}(1+|v|_{1+\ve}+\sup_{|s|\leq|\la|}|\db w(\cdot,s)|_{1+\ve})(|v|^*_{r+1}+C_{r,\ve}\sup_{|s|\leq|\la|}|\db w(\cdot,s)|_{r+1}).
\end{align*}
Therefore,
\al\label{dbwdot}
&|\db w(\cdot,\la)|_{r}\leq C_{r,\ve}|\db v|^*_{r+1}|v|^*_{1+\ve'}+C_{r,\ve}|\db v|^*_{1+2\ve'}|v|_{r+1}\\
&\qquad+C_{r,\ve}(|v|^*_{2+\ve}+|v|_{r+1})\sup_{| s |\leq|\la|}|\db w(\cdot, s )|_\ve+C_{r,\ve}
|v|^*_{1+\ve} \sup_{| s |\leq|\la|}|\db w(\cdot, s )|_r.
\nonumber
\end{align}
Taking $r=\ve'$ and using $|v|^*_{2+\ve'}<\del_2$, we first get
$$
\sup_{|s|\leq 1}|\db w|_{\ve'}\leq C_{\ve'}|\db v|_{1+\ve}^*| v|_{1+\ve}^*+C_{\ve'}|v|^*_{2+\ve'}|\db v|^*_{\ve}.
$$
Using the above and solving for $\sup_{|s|\leq 1}|\db w(\cdot,s)|_{r}$ from \re{dbwdot} with $\la=1$, we obtain
\aln
\sup_{|s|\leq 1}|\db w(\cdot,s)|_{r}&\leq C_{r,\ve} (|\db v|^*_{r+1}|v|^*_{1+\ve'}+|\db v|^*_{1+2\ve'}|v|^*_{r+1}).
\end{align*}
This shows \re{dbwr}.
We can replace $\ve'=0$ in the above computation when we take $\ve=0$ and replace all norms in the above computation   by    the H\"older norms.
We have verified the lemma.
\end{proof}

\setcounter{thm}{0}\setcounter{equation}{0}

\section{Local  homotopy formulas}\label{sect:lhf}

Let $M$ be a relatively compact  domain in $\cL M$ defined by $\rho<0$, where $\rho$ is a $C^2$ defining function  with $d\rho(p)\neq0$ when $\rho(p)=0$.   Following Henkin-Leiterer~\ci{MR986248}, we say that $ b M$ is \emph{strictly $q$-convex} at $p\in b M$,  if the Levi-form $L_p\rho$, i.e. the restriction of the complex Hessian $H_p\rho$  on  $T_p^{1,0}( b M)$,  has  at least $q$ positive eigenvalues. We say that $ b M$ is \emph{strictly $q$-concave} at $\zeta\in b M$, if   $L_\zeta\rho$ has at least  $q$ negative eigenvalues. Thus, a  domain is strongly pseudoconvex if and only if it is strictly $(n-1)$-convex.

Thus, $M$ is an $a_q$ domain if and only if $ b M$ is either $(n-q)$ convex or $(q+1)$ concave, in which case we decompose
$
bM=b^+_{n-q}M\cup b^-_{q+1}M
$
where $b^+_{i}M$ (resp. $b^-_{i}M$) is the components of $bM$ that has at least $i$ positive (resp. $i$ negative) Levi eigenvalues.

In this section, we will derive the local homotopy formulas near each boundary point of an $a_q$ domains. One of the important tools is the commutator $[\db, E]=\db E-\db E$ for suitable extension operators $E$.   The  commutator $[\dbar, E]$ was introduced by Peters \cite{MR1001710} and has been used by Michel \cite{MR1134587}, Michel-Shaw \cite{MR1671846} and others.  It is also one of the main ingredients in the $\frac{1}{2}$-gain estimate~\cites{MR3961327, aq} for a homotopy operator on $a_q$ domains. In the above-mentioned papers, the Seeley or Stein extension operators are used. 
Recently, the Rychkov extension operator to the study of $\db$ problem was introduced by Shi--Yao~\cites{MR4688544,SYajm}. We also use the Rychkov extension to improve some results in~\ci{aq}.

\subsection{Existence of local homotopy formulas}

\begin{defn}\label{pck} Let $k\geq1$ and $r>1$. A relatively compact domain $M$ in a $C^k$ manifold $X$ is {\it  piecewise smooth} of class $C^k$
 (resp. $\Lambda^r$), if  for each $\zeta\in b M$, there are $C^k$ (resp. $\Lambda^r$) functions $\rho_1,\dots, \rho_\ell$ defined on a neighborhood $U$ of $\zeta$ such that $M\cap U=\{z\in U\colon\rho_j(z)<0,  j=1,\dots, \ell\}$, $\rho_j(\zeta)=0$ for all $j$, and
$$
d\rho_{1}\wedge\cdots\wedge d\rho_{\ell} \neq0.
$$
\end{defn}

 The main purpose of this section is to construct an approximate homotopy formula on a subdomain $M'$ of $M$, where $ b M'$ and $ b M$ share a piece of boundary containing a given point $\zeta_0\in b M$. The construction will be achieved in local coordinates. The following is proved in~\ci{aq} when $\cL M$ is $\cc^n$. The same proof is valid for the general case, since the result is local.

\begin{lemma}[\cite{aq}*{Lemma 3.1}]
\label{convex-rho}Let $M\Subset \cL M$ be a   domain defined by a $C^2$ function $\rho^0<0$ satisfying $\nabla\rho^0\neq0$ at each point of $U\cap b M$. Suppose that $ b M$ is $(n-q)$-convex at $\zeta\in U$.
\bpp
\item
There is a local biholomorphic mapping $\psi$ defined on an open set $U$ containing $\zeta$ such that $\psi(\zeta)=0$ while $D^1:=\psi(U\cap M)$ is defined by
\eq{qconv-nf}
\rho^1(z)=-y_{n}+\la_1|z_1|^2+\cdots+\la_{q-1}|z_{q-1}|^2+|z_{q}|^2+
\cdots+|z_{n}|^2+R(z)<0,
\eeq
where $|\la_j|\leq1/4$ and $R(z)=o(|z|^2)$.  There exists $r_1>0$ such that the boundary  $\pd\psi(U\cap M)$ intersects the sphere $\pd B_r$ transversally when $0<r<r_1$. Furthermore, the function $R$ in \rea{qconv-nf} is in $C^a(B_{r_1})$ $($resp. $\Lambda^a(B_{r_1}))$,
 when $\rho^0\in C^a(U)$ with $a\geq2$ $($resp. $\Lambda^a(U)$ with $a>2)$.
\item Let $\psi$ be as above. There exists $\delta(M)>0$ such that if $\widetilde M$ is defined by $\tilde\rho^0<0$ and $\|\tilde\rho^0-\rho^0\|_2<\delta(M)$, then  $\psi(U\cap\widetilde M)$ is given by
\eq{qconv-nf-t}
\tilde\rho^1(z)=-y_{n}+\la_1|z_1|^2+\cdots+\la_{q-1}|z_{q-1}|^2+|z_{q}|^2+
\cdots+|z_{n}|^2+\tilde R(z)<0
\eeq
with  $|\tilde R-R|_{B_{r_1},a}\leq C_a|\tilde\rho^0-\rho^0|_{U,a}$ for $a>2$ and  $\|\tilde R-R\|_{B_{r_1},a}\leq C_a\|\tilde\rho^0-\rho^0\|_{U,a}$ for $a\geq2$.
 There exists $r_1>0$ such that  the boundary  $\pd\psi(U\cap \widetilde M)$ intersects the sphere $\pd B_{r_2}$ transversally when $r_1/2<r_2<r_1$.
\epp
Here $\delta(M)$ depends on the modulus of continuity of $\pd^2\rho^0$.
\ele

  We now fix notation. Let $(D^1,U,\phi,\rho^1)$ be as in \rl{convex-rho}. Thus, $\rho^1$ is given by \re{qconv-nf} (or \re{qconv-nf-t}). Set
$$
\rho^2=|z|^2-r^2_2
$$
 where $0<r_2<r_1$ and $r_1/2<r_2<r_1$     for \rl{convex-rho} (a), (b). Let us define
$$
D^1\colon\rho^1<0, \quad D_{r_2}^2\colon\rho^2<0, \quad D_{r_2}^{12}=D^1\cap D_{r_2}^2.
$$
As in~\cite{aq}*{Def.~3.4}, we call $(D^1,D^2_{r_2},\rho^1,\rho^2)$ a $(n-q)$ convex configuration for which we take Leray maps
 Leray maps
 \gan{}\nonumber
 g^0(z,\zeta)=\ov\zeta-\ov z,\quad
g^2(z,\zeta)=(\f{\pd\rho^2}{\pd\zeta_1},\dots, \f{\pd\rho^2}{\pd\zeta_n})=\ov \zeta;
\\
\label{HL2pg81}
g^1_{j}( z,\zeta )=\begin{cases}
\DD{\rho^1}{z_j},&1\leq j\leq q+2,\\
\DD{\rho^1}{z_j}+\ov z_j-\ov \zeta_j,& q+3\leq j\leq n.
\end{cases}
\end{gather*}

 Let $g^j \colon D\times S^j\to\cc^n$ be   $C^1$ Leray   mappings for $j=1,\dots,\ell$.
Let $w=\zeta-z$.      Define  \gan
\omega^i=\f{1}{2\pi i}\f{g^i\cdot dw}{g^i\cdot w},
\quad
\Omega^i=\omega^i\wedge(\ov\pd\omega^i)^{n-1},\\
\Omega^{01}=\omega^0\wedge\omega^1\wedge\sum_{\alpha+\beta=n-2}
(\ov\pd\omega^0)^{\alpha}\wedge(\ov\pd\omega^1)^{\beta}.
\end{gather*}
Here both differentials $d$ and $\db$ are in $z,\zeta$ variables.
In general, define
\gan
\Omega^{1\cdots \ell}=\omega^{g_1}\wedge\cdots\wedge\omega^{g_\ell}\wedge\sum_{\alpha_1+\cdots+\all_\ell=n-\ell}
(\ov\pd\omega^{g_1})^{\alpha_1}\wedge\cdots(\ov\pd\omega^{g_\ell})^{\all_\ell}.
\end{gather*}
Decompose $\Om^{\bigcdot}=\sum\Om_{(0,q)}^{\bigcdot}$,
where $\Om_{(0,q)}^{\bigcdot}$
  has type $(0,q)$  in $z$. Hence, $\Om^{i_1,\dots, i_\ell}_{(0,q)}$ has type $(n,n-\ell-q)$    in $\zeta$. Set
   $\Om^{\bigcdot}_{0,-1}=0$ and $\Omega_{0,n+1}^{\bigcdot}=0$.

Note that when $M\Subset \mathcal M$ is strictly $(q+2)$ concave at $\zeta$, $\mathcal M\setminus \ov M$ is strictly $(n-q-3)$ convex at $\zeta$. Thus, there exists a biholomorphic change of coordinates $\psi$ such that $D^1=\psi(M\cap U)$ is defined by
$$
\rho^1(z)=-y_{q+3}-|z_1|^2-\cdots-|z_{q+3}|^2+\la_{q+4}|z_{q+4}|^2+
\cdots+\la_n|z_n|^2+R(z).
$$
We can intersect  $D^1\cap D^2$ with a third domain
$$
D^3\colon \rho^3<0, \quad 0\in D^3,
$$
where
$$
\rho^3(z):=-y_{q+2}+\sum_{j=q+3}^n3|z_j|^2-r^2_3
$$
with $0<r_3<r_2/{C_n}$. As in~\cite{MR986248}*{p.~120}   define
\gan
g^{3}_j( z,\zeta )=\begin{cases}
0,&1\leq j<q+2,\\
i,&j= q+2,\\
3(\ov\zeta_j+\ov z_j),& q+3\leq j\leq n.
\end{cases}
\end{gather*}
As in~\cite{aq}*{Def.~4.3}, we call $(D^1,D^2_{r_2}, D^3_{r_3})$ a $(q+2)$ concave configuration.

Set
$$
U^1=D^2\setminus\ov{D^1}, \quad S^1=\pd D^1\cap D^2, \quad S^1_+=\pd U^1\setminus S^1,\quad
S^2=\pd D^{12}\setminus S^1.
$$
   Next, we introduce integrals on  domains and lower-dimensional sets:
$$
R_{D; q}^{i_1\dots i_\ell}f(z):=\int_{D}\Om^{i_1\dots i_\ell}_{(0,q)}(z,\zeta)\wedge f(\zeta), \quad
L_{i_1\cdots i_\mu; q}^{j_1\dots j_\nu}f:=\int_{S^{i_1\cdots i_\mu}}\Omega_{(0,q)}^{j_1\cdots j_\nu}\wedge f.
$$

Let $D$ be a bounded Lipschitz domain in $\cc^n$.
We now recall an important estimate on the commutator $[X,\mathcal E]$ where $\cL E$ is the Rychkov extension operator \cite{MR1721827}
satisfying
$$
\cL E\colon\Lambda^r(D)\to\Lambda^r(\cc^n),\quad r>0.
$$

\begin{prop}[\cites{SYajm}]\label{sz-comm}
Let $D$ be a bounded Lipschitz domain in $\cc^n$. Then the Rychkov extension operator $\cL E\colon \Lambda^r(D)\to\Lambda^r(\cc^n)\cap C^\infty(\cc^n\setminus\ov D)$ satisfies
\eq{SZ-comm}
\Bigl|[\nabla,\cL E]u(x)\Bigr|\leq C_r(D)|u|_{\Lambda^r(D)}\dist(x,D)^{r-1}, \quad r>0.
\eeq
\end{prop}
Note that as used in~\ci{aq}, \re{SZ-comm} holds for $r\geq1$ when $\cL E$ is replaced by the Stein extension, provided $r>1$. The above version for $0<r\leq1$ plays an important role in~\ci{shiNN} to obtain our sharp regularity results to avoid a   loss of derivatives which occurs in~\ci{GG}.

Using the Rychkov extension operator for $\cL E_{D^{12}}$, define
\al{}
\label{hq1} H^{(1)}_s f&:=R_{U^1\cup D^{12}; s-1}^0 \cL E_{D^{12}} f+R_{U^1;s-1 }^{01}\Bigl[\db,\cL E_{D^{12}}\Bigr] f,
 \\
 \label{hq2}
 H^{(2)}_sf&:=-R^1_{U^1;s-1}\cL E_{D^{12}}f+L_{1^+;s-1}^{01} \cL E_{D^{12}}f +L_{2;s-1}^{02} f+L_{12;s-1}^{012}f,\\
L^{01}_{1^+; s-1} f&:=\int_{S^1_+}\Omega^{01}_{0,s-1}\wedge \cL E_{D^{12}}f.
 \label{H0f}
\end{align}

\begin{prop}\label{conv-hf} Let $r\in(1,\infty)$ and $1\leq q\leq n-1$.
Let $(D^1,D^2)$ be an $(n-q)$ convex configuration. There exists a homotopy formula
$
\var=\db H_q\var+H_{q+1}\db \var
$
for $\var\in \Lambda^r_{(0,q)}(D^{12}_{r_2}, V)$ with $\db \var\in \Lambda^r_{(0,q)}(D^{12}_{r_2}, V)$. Here $H_s=H^{(1)}_{s}+H_s^{(2)}$ with $H_s^{(1)}$ and 
$H_s^{(2)}$ being defined by \rea{hq1}-\rea{H0f}.
\end{prop}

\begin{thm}\label{conv-est} Let $r\in(0,\infty)$. For $j=q,q+1$,
the homotopy operator $H_j$ in \rpa{conv-hf} satisfies
\al{}\label{hqfr12-c}
\|H_j\var\|_{\Lambda^{r+1/2}(D^{1}\cap D^2_{r_3})}&\leq C_r(\nabla\rho^1,\nabla^2\rho^1)\|\var\|_{\Lambda^r(D^{12})}, \quad r_1/2<r_3<3r_2/4.
\end{align}
Moreover,  $C_r(\nabla\rho^1,\nabla^2\rho^1)$ is stable under small $C^2$ perturbations of $\rho^1$.
Consequently, the homotopy formula in \rpa{conv-hf} can be extended uniquely to the case  $r\in(0,\infty)$ so that \rea{hqfr12-c} still holds. 
\end{thm}

\subsection{Estimates of homotopy operators}

The proof of \rt{conv-est} can be derived by the following.
\le{c2case}Let $r=k+\all$ with $k\geq0$ be an integer and $0<\all\leq1$. Then
\gan\int_{[0,1]^3}\f{s_1^{r-1}t^{2n-3}\, d s_1ds_2dt}{(\del+ s_1+s_2+t^2)^{k+2}(s_1+s_2+t)^{2n-3}}\leq C_r\del^{\all-1/2},\quad 0<\all<1/2;\\
\int_{[0,1]^3}\f{s_1^{r-1}t^{2n-3}\, d s_1ds_2dt}{(\del+ s_1+s_2+t^2)^{k+3}(s_1+s_2+t)^{2n-3}}\leq C_r\del^{\all-3/2},\quad 1/2\leq\all\leq1.
\end{gather*}
\ele
\begin{proof}See the estimates for $I_1,I_2,I_3$ in the proof of Theorem~6.1 in~\ci{aq}.
\end{proof}

As in~\cite{SYajm},   the last assertion of \rt{conv-est} follows via a smooth approximation.
\begin{lemma}\label{LMr}
Let $M$ be a relatively compact Lipschitz domain in $\cL M$. Let $\var\in\Lambda^a(M)$ with $a>0$. Suppose that $\db\var\in\Lambda^b(M)$ with $a\geq b>0$. There exist  $\var_\ve\in C^\infty(\ov M)$ such that
\gan{}
\lim_{\ve\to0}\|\var_\ve-\var\|_{\La^{a'}(M)}=0, \quad\forall a'<a; \qquad \|\var_\ve \|_{\La^a}\leq C_r(M)\|\var\|_{\La^a},\\
\lim_{{\ve\to0}}\|\db\var_\ve-\db\var\|_{\La^{b'}(M)}=0, \quad\forall b'<b; \qquad \|\db\var_\ve\|_{\La^b}\leq C_r(M)\|\db\var\|_{\La^b}.
\end{gather*}
\end{lemma}
\begin{proof}We consider a partition of unity $\var=\sum\var_k$ where $\var_k=\chi_k\var$ and $\sum\chi_k=1$. Consider
$
\var_{k,\ve}=\var_k\ast\tilde\chi_\ve.
$
We have $\db\var_{k,\ve}=(\db\var_k)\ast\tilde\chi_\ve=(\db\chi_k\wedge\var+\chi_k\db\var)\ast\tilde\chi_\ve$. Thus,
as $\ve\to0$,
$\var_\ve:=\sum\var_{k,\ve}$ tends to $\var$ in $\La^{a'}(M)$  and $\var_\ve$ tends to $\db \var$ in $\La^{b'}(M)$.
\end{proof}

Let us first derive a corollary as in \ci{SYajm} \begin{cor}\label{ext-hf}
\rpa{conv-hf} and \rta{conv-est} still hold for $r\in(0,1]$.
\end{cor}
\begin{proof}We can define $H_q\var=H_q\var_\ve$ and $H_{q+1}(\db\var)=\lim_{\ve\to0}H_{q+1}(\db \var_\ve)$. Then the estimates in \rt{conv-est} imply that the homotopy formula holds for $\var$; see the end of proof of \rl{Lem:StD} for more details. 
\end{proof}
  
\begin{prop}\label{cchf} Let $r\in(1,\infty)$ and $1\leq q\leq n-3$.
 Let $(D^1,D^2 ,D^3 )$ be a  $(q+2)$-concave configuration. Assume that $r_4<r_3/C_n$.  For $\var\in
\Lambda^r_{(0,q)}({D^{12}}, V)
$ with $\db\var\in
\Lambda^r_{(0,q+1)}({D^{12}})$, we have \gan{}
\var=\db H_q\var+  H_{q+1}\db\var \quad \text{ on $D^{1}\cap B_{r_4}$}, 
\end{gather*}
where for  $j=q,q+1$,  $H_j=H_j^{(1)}+H_j^{(2)}+H_j^{(3)}$ with  $ H^{(1)}_j$,   $ H^{(2)}_j$ being given by \rea{hq1}-\rea{H0f}
and
$$
H_{j}^{(3)}f=
L_{12;j}^{123}f-T_{B_{r_4};j}L^{23}_{12;j} f.
$$  
\end{prop}

 To derive our main estimates, we use the following.
 \begin{lemma}[\cite{aq}]\label{HL}Let $0<\beta\leq1$.
Let $ D\subset\rr^n$ be a bounded and connected Lipschitz domain. Suppose that   $f$ is in $C_{loc}^{[\beta]+1}( D)$ and
$$
|\nabla^{1+[\beta]} f(x)|\leq A\dist(x,\pd D)^{\beta-1-[\beta]}.
$$
Fix $x_0\in D$. Then
$
\|f\|^{in}_{ \Lambda^{\beta}(D)}\leq \|f\|_{C^0(D)}+C_\beta A,
$
where constants $C_0,C_\beta$ are stable under small perturbations of $D$.
\ele

 \le{henint}Let $\beta\in(-1,\infty),\mu_1\in[0,\infty)$, $0\leq\la\leq\mu_1$,  and $0<\del<1$. Set
$$ 
\beta':=\beta-\f{\mu_1+\la-3}{2}.
$$ 
 Then for $m\geq0$,
$$ 
\int_{[0,1]^{3}}\f{ s_1^\beta(\del+s_1+s_2+t^2)^{-1-\mu_1}t^{m}}{(\del+s_1+s_2+t)^{m+\la-\mu_1}
}\, ds_1ds_2dt
<
\begin{cases}
C\del^{\beta'},&\beta'<0;\\
C,&\beta'>0.
\end{cases}
$$ 
\ele
\begin{proof}The lemma is proved in \cite{aq}*{Lem.~6.2} for $\beta\geq 0$. The same proof without any changes  is valid for $\beta>-1$.
\end{proof}

\begin{thm}\label{concave-est} Let $r\in(0,\infty)$, $\ve>0$ and $0\leq\delta\leq1/2$, and let $1\leq q\leq n-2$.  Let $(D^1,D^2,D^3)$ be a $(
q+
2 
)$-concave configuration.
Assume that $\var\in\La^{r}(\ov {D^{12}})$ and $\ve>0$. Then the homotopy operators in \rpa{cchf} satisfy
\al{}
\label{hqfr12}
\|H_q\var\|_{\Lambda^{r+\delta}(D^{12}_{r_4/2})}&\leq C_{r,\ve,\del}  (|\rho^1|_{\Lambda^{r+2+\delta}}\|\var\|_{\La^\ve(\ov{D^{12}})}+  \|\var\|_{\Lambda^r(D^{12})}).
\end{align}
Moreover,  $C_{r,\ve,\delta}$, which is independent of $\var$,
 stable under small $C^2$ perturbations of $\rho^1$. Consequently, the homotopy formula in \rpa{cchf} can be extended uniquely to the case  $r\in(0,\infty)$ so that \rea{hqfr12} still holds.
\end{thm}
\begin{proof}The last assertion follows from estimates \re{hqfr12} by a smooth approximation as \nrc{ext-hf}. 
 
We now derive the estimate. To ease notation, write
$$
\|\var\|_{\Lambda^a(D^{12})},\quad \|\var\|_{C^a(D^{12})},\quad\|H_q\var\|_{\Lambda^{a}(D^{12}_{r_4/2})},\quad \|H_q\var\|_{C^{a}(D^{12}_{r_4/2})}
$$
as  $|\var|_{a}, \|\var\|_{a}, |H_q\var|_{ a}, \|H_q\var\|_{ a}$,  respectively. We also write $\|\rho^1\|_{\Lambda^a(U^1)}, \|\rho^1\|_{C^a(U^1)}$ as $|\rho^1|_a,\|\rho^1\|_a$.
    Recall the homotopy operator
$$
H_{q}\var=(H^{(1)}_q+H^{(2)}_q+H^{(3)}_q)\var,
$$
where $\var$ has type $(0,q)$ and
$$
H^{(1)}_q\var=R_{U^1\cup D^{12}}^0 E \var+R_{U^1 }^{01}[\db,E] \var.
$$
Near the origin, each term in $ H^{(3)}_q\var$ is a linear combination of integrals of $K_0f_0$, where $f_0$ is a coefficient of $\var$. The $K_0f_0$ is expressed as
$$
K_0f_0(z):=A(z,\nabla_z\rho^1,\nabla_z^2\rho^1)\tilde K_0f_0(z),
$$
where  $\tilde K_0$, which  involves only $\nabla\rho^1$,
is defined by
$$
\tilde K_0f_0(z):=\int_{S^I}\frac{ f_0(\zeta )(\zeta,\bar\zeta)^e\, dV(\zeta)}{
(g^1(z,\zeta)\cdot(\zeta-z))^a(g^2(z,\zeta)\cdot(\zeta-z))^b(g^3(z,\zeta)\cdot(\zeta-z))^c|\zeta-z|^{2d}}.
$$
Here $a,b,c,d$ are non-negative integers, $e\in\nn^{2n}$, and $S^I$ is actually $S_{12}$. Therefore, for the $g^1,g^2,g^3$ that appears in the kernel, we have
$$
|g^i(z,\zeta)\cdot(\zeta-z)|\geq c_0
$$
when $z\in D^{12}_{r_4/2}$ and $\zeta\in S^I$; see~\ci{aq} for instance.
Then, we have
$
|\tilde K_0f_0|_{r+\delta}\leq C_r|\rho^1|_{r+1+\delta}\|f_0\|_{0}
$
and
$$
|K_0f_0|_{r+\delta}\leq C_r|\rho^1|_{r+2+\delta}\|\tilde K_0f_0\|_{0}+C_r |\tilde K_0f_0|_{r+\delta}.
$$
Here and in what follows $C_r$ denotes a constant depending on $\nabla\rho,\nabla^2\rho$.
Therefore, we have
\gan{}
|K_0f_0|_{a}\leq C_r|\rho^1|_{2+a}\|f\|_0,\quad a>0.
\end{gather*}
Each term in $H^{(2)}_q\var$ is a linear combination of integrals of $K_0f_0$ which have been estimated or integrals 
over $S^I$ being one of $S^{12}, S^2, S_1^+$ or over $U^1$, whose kernels do not have singularity for $z\in D^{12}_{r_4/2}$. The only difference is that we need to use $\cL E\var$ instead of $\var$. Thus we need to replace $\|f\|_0$ by $|\cL Ef|_\ve$. 
This shows that 
$$
\|(H_q^{(2)},H_q^{(3)})\var\|_{\Lambda^{r+\delta}(D^{12}_{r_4/2})}\leq C_{r,\ve}  |\rho^1|_{\Lambda^{r+2+\delta}}\|\var\|_{C^\ve(\ov{D^{12}})}.
$$

 We now estimate the main term $H^{(1)}\var$, which is decomposed as
 \eq{dbE}
H^{(1)}\var(z)=\int_{D^{12}\cup U^1}\Om_{(0,q)}^{0}(z,\zeta)\wedge E\var(\zeta)+ \int_{U^1}\Om_{(0,q)}^{01}(z,\zeta)\wedge[\db,E]\var(\zeta).
 \end{equation}
Denote the first integral by $K_1\var$. By an estimate on Newtonian potential~\cite{MR999729}*{p.~316} (see \rl{LetUr}), we have
$$
|K_1\var|_{r+1}\leq C_r|\var|_{r}, \quad r>0.
$$
The last integral in \re{dbE} can be written as a linear combination of
 \ga \label{defnKf}
K_2f(z):=A(z,\nabla_z\rho^1,\nabla_z^2\rho^1 )\tilde K_2f(z),
\end{gather}
where $f$ is a coefficient of the form $[\db,\cL E]\var$.  By \rp{SZ-comm}, we have
$$
|f(\zeta)|\leq C_r(D^1)|\var|_r \dist(\zeta,D^1)^{r-1}, \quad r>0.
$$
Note that $f$ vanishes on $\ov{ D^{12}}$. Further
\gan{}
\tilde K_2f(z):=\int_{U^1} f(\zeta )\f{ N_{1}(\zeta-z )}
 {\Phi^{n-j}(z,\zeta)|\zeta -z |^{2j}}\, dV(\zeta), \quad 1\leq j<n,\\
 \Phi(z,\zeta)=g^1(z,\zeta)\cdot(\zeta-z).
 \end{gather*}

 Fix $\zeta_0\in\pd D^1$.  We first choose local coordinates such that $s_1(\zeta),S^t(\zeta)$ and $t(\zeta)=(t_3,\dots, t_{2n})(\zeta)$ vanish at $\zeta_0$, $D^1$ is defined by $s_1<0$,  and
 \ga \label{LbPhi}
  |\Phi(z,\zeta)|\geq  c_*(\dist(z,\pd D^1)+ s_1(\zeta)+|s_2(\zeta)|+|t(\zeta)|^2),\\  C|\zeta-z|\geq |\Phi(z,\zeta)|\geq c_*|\zeta-z|^2,\\
|\zeta-z|\geq c_*(\dist(z,\pd D^1)+ s_1(\zeta)+|s_2(\zeta)|+ |t(\zeta)|),
 \label{LbPhi+}
  \end{gather}
  for $z\in D^1,\zeta\notin D^1$.

We have
\aln{}
|K_2f|_{r+\delta}&\leq C_r(\|\rho^1\|_{2})(|\rho^1|_{r+2+\delta}\|\tilde K_2f\|_0+ |\tilde K_2f|_{r+\delta}),\quad \  r>0.
\end{align*}

The rest of the proof is devoted to the proof of
\begin{gather}\label{tK2f}
|\tilde K_2f|_{r+\delta}\leq C_{r,\delta,\ve}( |\rho^1|_{r+2}|\var|_{\ve}+ |\var|_{r}), \quad r>0, \ 0\leq\delta\leq 1/2.
\end{gather}
Then combining   above estimates yields the proof for \re{hqfr12}.

\medskip

Recall that  $0\leq \delta\leq 1/2$.
Let $r=k+\all$ with integer $k\geq0$ and  $0<\all\leq1$.
In the following cases, we will apply \rl{henint} several times. For clarity, we will specify values $(\beta,\beta')$ in \rl{henint} via  $(\beta_i,\beta_i')$ when we use the lemma.

\medskip

$(i)$ $0<\all+\delta<1$. Recall that  the kernel of $\tilde K_2$  involves only first-order derivatives of $\rho^1$ in $z$-variables, and it does not involve $\zeta$-derivatives of $\rho^1$.   Further, $\Phi$ is a linear combination of $\zeta_j-z_j$.  Let $\nabla^{\ell}\rho^1$ denote a partial derivative of $\rho^1(z)$ in $z,\ov z$ of order $\ell$.
 Since $\rho^1\in \Lambda^{k+\all+2+\delta}\subset C^{k+2}$, via the product and chain rules we  can express $\nabla^{k+1}\tilde K_2f$ as a sum of
\eq{rhotimesK}
K_{\mu,\nu}^{(k+1)}f:= \nabla^{1+\nu'_1}\rho^1\cdots\nabla^{1+\nu'_{\mu_1}}\rho^1K_\mu f.
 \end{equation}
Here $\mu_1$ is  the number of times that we differentiate $\Phi(z,\zeta)$ in the numerator of the kernel of $\tilde K_2f$. Also,
 \eq{Kmuf}
 K_\mu f(z):=\int_{U^1}\f{f(\zeta)N_{1-\mu_0+\mu_1-\nu_1''-\cdots-\nu_{\mu_1}''+\mu_2}(\zeta-z)}
{(\Phi(z,\zeta))^{n-j+\mu_1}|\zeta-z|^{2j+2\mu_2}}dV.
\end{equation}
Note that $1\leq j<n$,
$\nu_i''=0,1$, and
$$ 
\mu_0+
\mu_2+\sum(\nu'_i+\nu''_i)\leq k+1,\quad \nu_i'+\nu_i''\geq1.
$$ 
To estimate \re{rhotimesK}, we  use \re{LbPhi}-\re{LbPhi+}: For $z\in D^1$, $\zeta\in U\setminus D^1$ and $D^1$ defined by $s_1<0$, we have
\gan  
  |\Phi(z,\zeta)|\geq  c_*(\dist(z,\pd D^1)+ s_1(\zeta)+|s_2(\zeta)|+|t(\zeta)|^2), \\
 C|\zeta-z|\geq |\Phi(z,\zeta)|\geq c_*|\zeta-z|^2,\\ |\zeta-z|\geq c_*(\dist(z,\pd D^1)+ s_1(\zeta)+|s_2(\zeta)|+ |t(\zeta)|.
  \end{gather*}
Consequently, the worst term for $K_\mu f$ occurs when  $j=n-1$. Note that we use $|\nabla\rho^1(z)|\geq c$ in \re{rhotimesK}. Thus,
the worst term for $K_{\mu,\nu}^{(k+1)}f$ also occurs when $\mu_0+\mu_2$ is absorbed into $\sum\nu_i''$.
Therefore, it suffices to estimate terms with $\mu_0=\mu_2=0$ and $j=n-1$, which are assumed now.

Throughout the proof, we assume that $\nu_i''=0$ and hence $\nu_i'\geq1$ for $i\leq \mu_1'$ and $\nu_i''=1$ for $i>\mu_1'$. Thus
$$
\mu_1'+\mu_1''=\mu_1, \quad \sum_{i\leq \mu_1'}\nu_i'=\sum_{i\leq\mu_1}\nu_i'.
$$
Thus \re{Kmuf} is simplified in the form
\ga{}\label{Kmuf+}
K_\mu f(z):=\int_{U^1}\f{f(\zeta)N_{1+\mu_1'}(\zeta-z)}
{(\Phi(z,\zeta))^{1+\mu_1}|\zeta-z|^{2(n-1)}}dV,\\
\label{mu1'}
\mu_1\leq\mu_1''+\sum_{i=1}^{\mu_1}\nu_i'\leq \sum(\nu_i'+\nu_i'')\leq k+1.\end{gather}

When applying  \rl{henint}, we always take
 $$ 
 \la:=\mu_1-\mu_1',\quad m=2n-3.
 $$ 
We need $\beta_1>-1$ and $
\beta_1'\leq\beta_1-\f{\mu_1+\la-3}{2}
$
which are satisfied by
$$
\beta_1'=\all+\delta-1<0, \quad
\beta_1=\max\Bigl\{-1+\ve,\all+\delta-1+ \f{\mu_1+\la-3}{2}\Bigr\}.
$$
Assume that $0<\ve<r$. Clearly, $\beta_1>-1$.  Using $\la=\mu_1-\mu_1'$ and \re{mu1'}, we can verify that $\beta_1\leq k+\alpha+\delta-3/2\leq r-1$. By  \rl{henint},    we obtain
 for $z\in D^{12}_{r_4/2}$,
\eq{Kmuf0}
|K_\mu f(z)|\leq |\var|_{\beta_1+1}\dist(z,\pd D^1)^{\beta_1'}.
\end{equation}
Thus,
\al \label{Kmu<}
|K_{\mu,\nu}^{(k+1)} f(z)|&\leq C\|\rho^1\|_{1+\nu_1'}\cdots\|\rho^1\|_{1+\nu_{\mu_1}'}|\var|_{\beta_1+1}
 \dist(z,\pd D^1)^{\all+\delta-1}.
\end{align}
When  $\beta_1=\ve-1$ or all   $\nu_i'\leq1$, we obtain \re{tK2f} from
$$
|K_{\mu,\nu}^{(k+1)} f(z)|\leq C\max \{|\rho^1|_{k+2}|\var|_\ve,|\var|_r\}
 \dist(z,D^1)^{\all+\delta-1}.
 $$

 Assume  now that $\beta_1>\ve-1$  and  $\nu_1'\geq2$. Thus $\mu_1'\geq1$.
 Let $x_+=\max\{0,x\}$ and
\eq{def-gamma}
\gamma:=\sum(\nu'_i-1)_+=-\mu_1'+\sum\nu_i'.
\end{equation}
  Then
$
\|\rho^1\|_{1+\nu'_1}\cdots\|\rho^1\|_{1+\nu'_{\mu_1}} \leq \|\rho^1\|_{2+\gamma}.
$
Then
\aln{}
(\beta_1+1)+\gamma&\leq\Bigl[\all+\delta+ \f{\mu_1+\la-3}{2}\Bigr]-\mu_1'+{\textstyle\sum}\nu_i'=\all+\delta- \f{\mu_1'+1}{2}+\mu_1''+{\textstyle\sum}\nu_i'\\
&\leq \all+\delta-\f{\mu_1'+3}{2}+k\leq r-\f{1}{2},
\end{align*}
where the second inequality is obtained by \re{mu1'} and $\mu'_1\geq1$. Therefore, \re{tK2f} follows from
\aln{}
&\|\rho^1\|_{1+\nu_1'}\cdots\|\rho^1\|_{1+\nu_{\mu_1}'}|\var|_{\beta_1+1}\leq C_r \|\rho^1\|_{2+\gamma}\|\var\|_{\beta_1+1} \\
&\qquad\leq C_r'\|\rho^1\|_{2}\|\var\|_{1+\beta_1+\gamma}+C_r'\|\rho^1\|_{3+\gamma+\beta_1}\|\var\|_0\leq
C_r''|\var|_{r}+C_r''\|\rho^1\|_{r+3/2}\|\var\|_0.
\end{align*}

\medskip

$(ii)\  1<\all+\delta\leq3/2$. In this case we can take an extra derivative since $\rho^1\in\Lambda^{k+\all+2+\delta}\subset C^{k+3}$. Write
$\nabla^{k+2}\tilde K_2f$ as a sum of
$$
 K_{\mu,\nu}^{(k+2)}f:=\nabla^{1+\nu'_1}\rho^1\cdots\nabla^{1+\nu'_{\mu_1}}\rho^1K_\mu f,
 $$
where $K_\mu f$ is defined by \re{Kmuf}.
As before, the worst term occurs for $j=n-1$ and $\mu_0=\mu_2=0$ in $K_\mu$, which are assumed now.
  Then \re{mu1'} in which  $k$ is                                                                                                             replaced by $k+1$  becomes
  $$  
\mu_1\leq\mu_1''+\sum_{i=1}^{\mu_1}\nu_i'\leq k+2.
$$

To apply \rl{henint},
we  need to   $\beta_2>-1$ and
$
\beta_2'\leq\beta_2-(\mu_1+\la-3)/2.
$  
Thus, we take
$$
\beta_2'=\all+\delta-2<0,
\quad \beta_2=\max\Bigl\{\ve-1,\all+\delta-2 + \f{\mu_1+\la-3}{2}\Bigr\}.
$$
 Recall $\la=\mu_1-\mu_1'$. We can verify that $\beta_2\leq r-1$.  We obtain for $z\in D^{12}_{r_4}$
\aln{} 
|K_{\mu,\nu}^{(k+2)} f(z)|&\leq C_r\|\rho^1\|_{1+\nu_1'}\cdots\|\rho^1\|_{1+\nu_{\mu_1}'}|\var|_{\beta_2+1}
 \dist(z,\pd D^1)^{\all+\delta-1}.
\end{align*}
When $\beta_2=\ve-1$ or all  $\nu_i'\leq1$, we have \re{Kmuf0}-\re{Kmu<} in which $\beta_1',\beta_1$ are replaced by $\beta_2',\beta_2$ respectively. Thus, we get
$$
|K_\mu(z)|\leq C_r|\var|_{r}\dist(z,D^1)^{\all+\delta-1}.
$$

Assume now   $\beta_2>\ve-1$ and  $\nu_1'\geq2$. Thus $\mu_1'\geq1$. Let
 $\gamma$ be given by \re{def-gamma}.
  Then
we have
\aln{}
(\beta_2+1)+\gamma&\leq\Bigl[\all+\delta + \f{\mu_1+\la-5}{2}\Bigr]-\mu_1'+{\textstyle\sum}\nu_i'\\
&=\all+\delta- \f{\mu_1'+5}{2}+\mu_1''+{\textstyle\sum}\nu_i'\leq \all+\delta- \f{\mu_1'+1}{2}+k\leq r-\f{1}{2}.
\end{align*}
As in previous case, we get \aln{}
&\|\rho^1\|_{1+\nu_1'}\cdots\|\rho^1\|_{1+\nu_{\mu_1}'}|\var|_{\beta_2+1}\leq
C_r(|\var|_{r}+\|\rho^1\|_{r+3/2}\|\var\|_\ve).
\end{align*}

\medskip

$(iii)\  \all+\delta=1$. Thus $\alpha\geq1/2$.  We need to estimate $|\tilde K_2f|_{r+\delta}$ with $r+\delta=k+1$.

 Recall that $\rho^1\in\Lambda^{k+3}$.  On  $D^{12}_{r_4/2}$, we write $\nabla^k\tilde K_2f$ as
a sum of $$
 K_{\mu,\nu}^{(k)}f:=\nabla^{1+\nu'_1}\rho^1\cdots\nabla^{1+\nu'_{\mu_1}}\rho^1K_\mu f.
 $$
To estimate $|K^{(k)}_{\mu,\nu}f|_1$, we will use   Lemmas~\ref{HL} and \ref{henint} to estimate $|K_\mu f|_1$.

As before, the worst term occurs for $j=n-1$ and $\mu_0=\mu_2=0$ in $K_\mu$, which are assumed now.
Then $K_\mu f$ has the form \re{Kmuf+}, while \re{mu1'} has the form
\eq{mu1k}
\mu_1\leq\mu_1''+\sum_{i=1}^{\mu_1}\nu_i'\leq k.
\eeq

 By \re{zprule}, we   have up to a constant factor
\al\label{halfcase}
 |K_{\mu, \nu}^{(k)}f|_1
 &\leq
  |\nabla^{1+\nu'_1}\rho^1\cdots\nabla^{1+\nu'_{\mu_1}}\rho^1|_1\|K_\mu f\|_0\\
  &\quad
 +\|\nabla^{1+\nu'_1}\rho^1\cdots\nabla^{1+\nu'_{\mu'_1}}\rho^1\|_0|K_\mu f|_1=:\operatorname{I}+\operatorname{II}.
 \nonumber
 \end{align}

 We first estimate $\operatorname{I}$.
We set
 $$
 \beta_3'=\ve, \quad \beta_3=\max\Big\{\ve-1,\beta_3'+\f{\mu_1+\la-3}{2}\Bigr\}.
 $$
Then we use \rl{henint} for $\beta_3'>0$, which yields
$$
\|K_\mu f\|_0\leq |\var|_{\beta_3+1}.
$$
Here we take $0<\ve<r$ and recall $\la=\mu_1-\mu_1'$.
 Then we have $\beta_3\leq r-1$.
  Also,
 $$ \|\nabla^{1+\nu'_1}\rho^1\cdots\nabla^{1+\nu'_{\mu_1}}\rho^1\|_1\leq C_r\sum \|\nabla^{2+\nu'_i}\rho^1\|_0\prod_{j\neq i}\|\nabla^{1+\nu'_j}\rho^1\|_0\leq C_r'\|\rho^1\|_{2+\tilde\gamma}
 $$
 with $\tilde\gamma=\sum\nu_i'$. When $\beta_3=\ve-1$ or all $\nu_i'\leq1$, we get \re{Kmuf0}-\re{Kmu<} in which $\beta_1',\beta_1$ are replaced by $\beta_3',\beta_3$ respectively. Thus, we get
$$
\operatorname{I}\leq C_r\max\{\|\rho^1\|_{r+2}||\var|_\ve,|\var|_{r} \}.
$$
Assume that $\beta_3>\ve-1$ and some $\nu_i'\geq2$.  Then $\mu_1'\geq1$ and
 \begin{align*}
(\beta_3+1)+\tilde\gamma&\leq\f{\mu_1+\la-1}{2}+\ve +\sum\nu_i'
\\
&=\ve-  \f{\mu_1'}{2}-\f{1}{2}+\mu_1''+\sum\nu_i'
\leq\ve-  \f{\mu_1'}{2}-\f{1}{2}+ k\leq r-1+\ve.
\end{align*}
Thus,   we get
$$
\operatorname{I}\leq C_r\|\rho^1\|_{2+\tilde\gamma}\|\var\|_{1+\beta_3} \leq C_{r,\ve}'
\|\rho^1\|_{2}\|\var\|_{r-1+\ve}+C_{r,\ve}'\|\rho^1\|_{r+1+\ve}\|\var\|_0.
$$

Finally, we estimate $\operatorname{II}$ in \re{halfcase}.   \rl{HL} says that we can  estimate $|K_\mu f|_1$  via the pointwise estimate of $\nabla^2K_\mu f$.
Write $\nabla^2K_\mu f$ as a sum of
$$
 K_{\tilde \mu,\tilde \nu}^{(k)}f:=\nabla^{1+\tilde\nu'_1}\rho^1\cdots\nabla^{1+\tilde\nu'_{\tilde\mu_1}}\rho^1K_{\mu+\tilde \mu} f
$$
with
$$
\tilde\mu_1\leq\tilde\mu_1''+\sum\tilde\nu_i'\leq 2.
$$
The worst terms occur in the forms
$$
\operatorname{II}':=\nabla^3\rho^1K_{\mu+\tilde \mu}f,\quad
 \operatorname{II}'':=\nabla\rho^1\nabla^2\rho^1K_{\mu+\tilde{\tilde\mu}}f
$$
for $\tilde\mu=1$ and $\tilde{\tilde\mu}=2$, respectively.

We first estimate $\operatorname{II}'$ where $\tilde\mu=1$. We have
\eq{KmuTm}
K_{\mu+\tilde\mu} f(z):=\int_{U^1}\f{f(\zeta)N_{1+\mu_1'+\tilde\mu_1'}(\zeta-z)}
{(\Phi(z,\zeta))^{1+\mu_1+\tilde\mu_1}|\zeta-z|^{2(n-1)}}dV(\zeta).
\end{equation}
Note that the worst term occurs when $\tilde\mu_1=\tilde\mu=1$.
Let us combine the indices as follows.   Set $\hat \mu=\mu+\tilde\mu=\mu+1$, $\hat\mu_1=\mu_1+\tilde\mu_1=\mu_1+1$.  Set $\hat\nu_i'=\nu_i'$ for $i\leq\mu_1$, $\hat\nu'_{\hat \mu_1}=2$. The latter implies
$$\hat\mu_1'=\mu_1'+1\geq1.
$$
Set  $\hat\lambda=\hat\mu_1''=\hat\mu_1-\hat\mu_1'$. Then $\hat\mu_1''=\mu_1''$. We get
$$
\hat\mu_1\leq\hat\mu_1''+\sum\hat\nu_i'=\mu_1''+2+\sum\nu_i'\leq k+2.
$$
 Thus, we set
 $$
 \beta_4'=\ve-1, \quad \beta_4=\max\Bigl\{\ve-1,\beta_4'+\f{\hat\mu_1+\hat\la-3}{2}\Bigr\}.
 $$
  Then
 $$
 |\operatorname{II}'(z)|\leq C_r\|\rho^1\|_{1+\tilde\nu_1'}|\var|_{\beta_4+1} \dist(z,\pd D^1)^{\beta_4'}, \quad \tilde\nu_1'=2.
 $$
  For the new indices, let $
\hat\gamma=\sum(\hat\nu'_i-1)_+.
$
 We have $$\| \nabla^{1+\hat\nu'_1}\rho^1\cdots\nabla^{1+\hat\nu'_{\hat\mu_1}}\rho^1\|_0 \leq C_r\|\rho^1\|_{2+\hat\gamma}.$$
Assume first that $\beta_4>\ve-1$. Then by $\hat\mu'_1\geq1$, we get
 $$
(\beta_4+1)+\hat\gamma\leq\Bigl[\ve-\f{3}{2}+\f{\hat\mu_1+\hat\la}{2}\Bigr]- \hat\mu_1'+ \sum\hat\nu_i'=\ve-\f{\hat \mu_1'}{2}-\f{3}{2}+ (k+2)\leq r-\all+\ve.
$$
Since  $\alpha=1-\delta\geq 1/2$, we obtain the desired estimate for $|\operatorname{II}'(z)|$ from
 $$
\|\rho^1\|_{2+\hat\gamma}|\var|_{1+\beta_4}\leq C'_r \|\rho^1\|_{r+1/2}\|\var\|_1+C_r'\|\rho^1\|_{2}\|\var\|_{r-1/2+\ve}.
$$
When $\beta_4=\ve-1$,  with $k\geq0$ and \re{mu1k} we  simply use $\hat\gamma=1+\sum(\nu_i'-1)_+\leq k\leq r-1/2$. Thus $ \|\rho^1\|_{2+\hat\gamma}\|\var\|_{1+\beta_4}\leq C_r\|\rho^1\|_{r+3/2}\|\var\|_\ve$, which gives us the desired estimate.

We now estimate  $\operatorname{II}''$.
Then we have $\tilde{\tilde\mu}=2$, and hence
$K_{\mu+\tilde{\tilde\mu}} f(z)$ has the form \re{KmuTm} in which $\tilde\mu, \tilde\mu_1',\tilde\mu_1$ are replaced by $\tilde{\tilde\mu}, \tilde{\tilde\mu}_1',\tilde{\tilde\mu}_1$ respectively with $$
\tilde{\tilde\mu}_1\leq\tilde{\tilde\mu}_1''+\sum\tilde{\tilde\nu}_i'\leq 2.
$$
The worst term occurs when $\tilde{\tilde\mu}_1=2$.
Let us combine the indices as before.   Set $\hat \mu=\mu+2$, $\hat\mu_1=\mu_1+\tilde{\tilde\mu}_1$, $\hat\mu_1'=\mu_1'+\tilde{\tilde\mu}'$,  and $\hat\lambda=\hat\mu_1''=\mu_1''+\tilde{\tilde\mu}''$. Let $\hat\nu_i'=\nu_i'$ for $i\leq\mu_1$ and  $\hat\nu'_{\hat \mu_1}=2$. The latter implies
$\hat\mu_1'=\mu_1'+1\geq1.
$ Then
\eq{hatmu1}
\hat\mu_1''+\sum\hat\nu_i' \leq k+2.
\end{equation}
 Thus, we set
 $$
 \beta_5'=\ve-1, \quad \beta_5=\max\Bigl\{\ve-1,\beta_5'+\f{\hat\mu_1+\hat\la-3}{2}\Bigr\}.
 $$
 Then
 $
 |K_{\mu+\tilde{\tilde\mu}} f(z)|\leq C_r|\var|_{\beta_5+1} \dist(z,\pd D^1)^{\beta_5'}.
 $
The values of $\beta_5,\beta_5'$, $\hat\mu_1'$,  and condition \re{hatmu1} are the same as in the previous case. The same computation yields the desired estimate
for $\operatorname{II}'$.
%
%
%
%
\end{proof}

\section{Globalizing local homotopy formulas via  $\db$-Neumann operators}\label{ghf}
\setcounter{thm}{0}\setcounter{equation}{0}

In this section, we will apply the $C^\infty$  regularity of the $\db$-Neumann operators $N_q,N_{q+1}$ for $a_q^+$ domains. We will need only the interior regularity of $N_q, N_{q+1}$ and their boundedness in $L^2$ spaces.  

\subsection{Grauert bumping}
We start with the following.
\begin{lemma}[\cite{aq}*{Lem.~8.1}]\label{in-sm}
Let $ M\subset \cL M$ be a domain defined by a $C^2$ function $\rho<0$ and let $M^a$ be defined by $\rho<a$. Let $\rho_t=S^t\rho$, where $S^t$ is the Moser smoothing operator. Suppose that $M$ is an $a_q$ domain. Let $ (M_t)^a$ be defined by $\rho_t<a$.  There exist $t_0=t_0(\nabla^2\rho)>0$ and $C>1>c>0$ such that if $0\leq t<t_0$, then
\begin{gather*}
\pd (M_t)^{-t}\subset M^{-ct}\setminus M^{-Ct}, \quad\pd (M_t)^{t}\subset M^{Ct}\setminus M^{ct}, 
\end{gather*}
while $ M^{b}$ and $ (M_t)^{b}$ still satisfy the condition $a_q$ for $b\in(-t_0,t_0)$.
\ele

We also need the following family of smooth subdomains that touch  given boundary points of
a  $C^2$ domain. This special smoothing will be used in Subsection~\ref{subsec:8.4}.
\le{touch}Let $M,\rho$ be as above. Let $\zeta\in b M$ and $\ve>0$. There exists a real function  $\tilde\rho\in C^\infty(\cL M)$ satisfying  $\tilde\rho(\zeta)=0$, $\widetilde M:=\{\tilde\rho<0\}$ is contained in $M$,  and $\|\tilde\rho-\hat\rho\|_{C^2(\cL M)}<\ve$, where $\hat\rho$ is a suitable $C^2$ defining function of $M$. 
\ele
\begin{proof}   We will use a smooth coordinate map $\psi\colon U\subset \cL M \to (-1,1)^{2n}$ such that $\psi(\zeta)=0$ and $\rho_1=\rho\circ\psi^{-1}=u\hat\rho$  with
$$
\hat\rho(x)=-x_1+Q(x)+R(x),
$$
where $Q$ is a quadratic polynomial, $|R(x)|\leq \ve\e_1^2$ when $|x|<\e_1$, and $u>0$. Here $\ve_1$ and $\ve_2,\ve_3$ below are to be determined. Note that $R\in C^2$. Set
$$
\hat\rho_\ve(x):=-x_1+Q(x)+\ve|x|^2.
$$
Let $\chi\colon\rr^{2n}\to[0,1]$  be a smooth function having compact support in the unit ball $B_1$ and value $1$ on $B_{1/2}$. Let $\chi_{\ve}$ be the trivial extension of $\chi(\ve^{-1}\psi)$ being $0$ on  $M\setminus U$.  Define
$$
\rho_\ve=(1-\chi_{\e_1})(S^t\rho+\e_2)+\chi_{\e_1} \hat\rho_{\e_3}\circ\psi, \quad \rho_2:=(1-\chi_{\e_1})\rho+\chi_{\e_1} \hat\rho\circ\psi.
$$
Then $\rho_\ve-\rho_2=(\e_2+S^t\rho-\rho)(1-\chi_{\e_1})  +\chi_{\e_1}(\hat\rho_{\e_3}-\hat\rho )\circ\psi$.
 Thus, we have
\aln
\|\rho_\ve-\rho_2\|_{C^2}&\leq\e_2\|1-\chi_{\e_1}\|_2+C\|
S^t\rho-\rho\|_0\|1-\chi_{\e_1}\|_2\\
&\quad+C\|S^t\rho-\rho\|_{C^2}+\|\chi_{\e_1}(\hat\rho_{\e_3}-\hat\rho)\circ\psi\|_2.
\end{align*}
The first two terms are less than $\ve$, provided
$$
\ve_2\e_1^{-2}<\ve/C,\quad t^2\|\rho\|_2\e_1^{-2}<\ve/C.
$$
When $t<t_0$ for $t_0$ depending on the module of continuity of $\nabla^2\rho$, we have $\|S^t\rho-\rho\|_{C^2}<\ve.$
We have by a direct computation via the product rule
\al\label{3R}\|\chi_{\e_1}(\hat\rho_{\e_3}-\hat\rho)\circ\psi\|_{C^2}&\leq C(\e_1^{-2}\e_3+
\e_1^{-2}\|R\|_{C^0(B_{\e_1})}\\&\quad+\e_1^{-1}\|\nabla R\|_{C^0(B_{\e_1})}+\|\nabla^2 R\|_{C^0(B_{\e_1})}).\nonumber
\end{align}
We assume $\e_1^{-2}\e_3<\ve/C$. We also choose $\del(\nabla^2R)>0$ so that $\e_1<\del(\nabla^2R)$ implies that the last three term in \re{3R} are less than $\ve/C$.

On the support of $\chi_{\e_1}$ except at $x_0$, we have
$
\hat\rho_\ve-\hat\rho>0
$
 for  small $\e_1$. On $\ov M$, we have $S^t\rho+\e_2-\rho>\e_2-C_2\|\rho\|_2t^2>0$.
Therefore, $\{\rho_\ve<0\}$ is contained in $M$.
\end{proof}

\le{bumping} Let $q>0$, $r\in(0,\infty)$, $\ve>0$ and $0\leq\delta\leq 1/2$. Let $M\Subset\cL M$ be an $a_q^+$ domain with a $C^2$ defining function $\rho$.  For each $ p \in b  M$ and any open set $U_0$ containing $ p $, there are open sets $U_1,U_2$ such that  $p\in U_2\Subset U_1\Subset U_0$ and on $U_2\cap M$
$$
\var =\db T_{U_1,\rho;q}\var +T_{U_1,\rho;q+1}\db \var ,
$$
where $\var \in\Lambda_{(0,q)}^r(U_1\cap M,V)$ and $ \db \var \in\Lambda_{(0,q+1)}^r(U_1\cap M,V)$. Furthermore, There exists $M_1$ that is defined by $\rho_1<0$ with $|\rho_1-\rho|_{a}<C_a\ve$ for all $a>0$ satisfies the following:
\bpp
\item
 The domain $M_1$ satisfies the $a_q^+$ condition, $M\cup\om\subset M_1\subset M\cup U_0$ and $\om\cap b  M$ is independent of $\ve$.
\item Let $\tilde \var =0$ on $\ov{M_1\setminus M}$ and  
$
\tilde\var =\var -\db(\chi T_q\var )-\chi T_{q+1}\db \var $ on $\ov M$.
Then
\gan{}
     |\tilde\var |_{M_1,r}\leq C_{r,\ve}(\nabla\rho^1,\nabla^2\rho^1)(|\rho^1|^*_{r+2}|\var |_{M^{12},\ve}+|\var |_{r}),\\ 
    |\tilde\var |_{M_1,r+\delta}\leq C_{r,\delta,\ve}(\nabla\rho^1,\nabla^2\rho^1)(|\rho^1|^*_{r+2+\delta}
    |\var |_{M,\ve}+|\var |_{r}).
\end{gather*}
\epp
The same estimates hold if $\var ,\tilde\var $ are replaced by $\db\var ,\db\tilde\var $, respectively. 
 Moreover, $C_{r,\delta,\ve}(\nabla\rho^1,\nabla^2\rho^1),C_{r,\ve}(\nabla\rho^1,\nabla^2\rho^1)$ are stable under $C^2$ perturbation of $\rho$.
 \ele
\pf{  By a straightforward computation, we have
\begin{align*}
\tilde \varphi&=(1-\chi) \var-\db\chi\wedge T_q\var \quad \text{on $M$}, \quad   \tilde \varphi=0 \quad \text{on $U_2$}.
    \end{align*}
Then    $\db \tilde \varphi=(1-\chi)\db \var-\db \chi \wedge \var+\db\chi\wedge\db T_q\var$ on $M$. Thus 
 \aln{}
     \db  \tilde \varphi&=(1-\chi) \db \var-\db \chi\wedge T_{q+1}\db \var \quad \text{on $M$}.
\end{align*}
Thus, $\tilde \varphi,\db\tilde \varphi$ satisfy the same kind of relations. It suffices to estimate $\tilde \varphi$.

Choose a nonnegative $\tilde\chi\in C^\infty_0(U_2)$ such that $\tilde\chi\geq0$ and $\tilde\chi=1$ on an open set $U_3$ with $p\in U_3\Subset U_2$.   Let $\rho_1=\rho-\ve\tilde\chi$.

For any $\ve>0$, it is clearly that $M_1\colon \rho_1<0$ contains $U_3\cap b  M$ and the latter is independent of $\ve$.
Hence, we can find an open neighborhood $\om$ of $p$ such that $M\cup\om\subset M_1$ and $\om\cap b  M$ is independent of $\ve$ and contains $U_3\cap b  M$.
\qedhere}

\begin{prop}\label{K-ext}
Let $M,\rho$ be as in \rla{bumping}. Let $r\in(0,\infty)$, $\ve>0$ and $0\leq\delta\leq1/2$.  Then there exists an approximate homotopy formula
$$
\var=\db T_q\var+T_{q+1}\db\var+  K^0_q\var\quad\text{on $M$},
$$ 
where the operators $K^0_q,T_q,
T_{q+1}$ satisfy
\ga{}
 \label{K0var}    |K^0\var|_{\widetilde M,r}\leq C_{r,\ve}(\nabla\rho^1,\nabla^2\rho^1)|\rho|^{*b}_{2+\ve}|\var|_{M,r}, \\ 
  \label{dbK0}    |\db K^0\var|_{\widetilde M,r}\leq C_{r,\ve}(\nabla\rho^1,\nabla^2\rho^1)|\rho|^{*b}_{2+\ve}|\db\var|_{M,r}, \\
\label{Tjpsi}    |T_j\psi|_{M,r+\delta}\leq C_{r,\delta,\ve}(\nabla\rho^1,\nabla^2\rho^1)(|\rho|^*_{r+2+\delta}|\rho|^{*b}_{2+\ve}
    |\psi|_{D,\ve}+|\psi|_{M,r}).
\end{gather}
Here, $\widetilde M\Supset M$, and   constants $b\geq 1$, $C_{r,\ve}(\nabla\rho^1,\nabla^2\rho^1)$, 
and $ C_{r,\delta,\ve}(\nabla\rho^1,\nabla^2\rho^1)$ are stable under small $C^2$ perturbations of $\rho$.
\end{prop}
\begin{proof}Denote the $U_i,\om,\rho_1,\chi,\tilde\chi$ in \rl{bumping} by $U_{i}(p),\om_p,\rho_{1,p},\chi_p,\tilde\chi_p$ as they depend on $p$.
By \rl{bumping}, we can find finitely many points $p_1,\dots, p_N$ in $ b M$ that are independent of $\ve$ and open sets $\om_{p_1},\dots, \om_{p_N}$ such that $ b M\subset\cup\om_{p_j}$. Let $\var_0=\var$ and $\rho_0=\rho$. Let $\rho_{j}=\rho_{j-1}-\ve\tilde\chi_{p_{j}}$.

Now, $M^j$ is a small perturbation of $M$, so there exist operators $T_{j,q}, T_{j,q+1}$  with $\var = \db T_{j,q} \var + T_{j,q+1} \db \var$ on $U_1(p_j) \cap M^j$. 
For simplicity, we have written $T_{j,q}:= T_{U_1(p_{j}),\rho_{j-1};q}$ and $T_{j,q+1}:= T_{U_1(p_j),\rho_{j-1};q+1}$. 
 We now copy the definition of $\tilde\var$ via $\var$ in \rl{bumping} $(b)$.   For $j\geq0$, we thus define 
$$
\var_{j+1}= \begin{cases}
\var_j-\db(\chi_{p_{j+1}} T_{j+1,q}\var_j)-
\chi_{p_{j+1}} T_{j+1;q+1}\db\var_j, & \text{on $M^j$}, \\  =0, &\text{on $M^{j+1}\setminus M^j$}.
\end{cases}
$$
Then $K^0_q\var:=\var_N$ and 
$$
T_q\var := \sum_{j=1}^N \chi_{p_j}   T_{j,q} \var_{j-1}, \quad T_{q+1}\db\var := \sum_{j=1}^N \chi_{p_{j}} T_{j,q+1} \db\var_{j-1} $$
  have the desired estimates by   applying the estimates in \rl{bumping} inductively and taking $\widetilde M=M_N$.

 Note that the $b\leq CN$ is   determined by the number of bumps that  are used. Thus, $b$ is stable under $C^2$ perturbations of $\rho$. 
\end{proof}

Next, we want to achieve interior $C^\infty$ smoothing via the homotopy formula on the coordinate balls.
We need the an interior estimate
which follows easily from \ci{MR999729}.
\le{}\label{LetUr}
Let $U_r:= \{z\in \cc^n:\|z\|< r \}$. Then  for any $C^1$ $(0,q)$ form $\var$  on $U_1$ with $q>0$, we have
$
\var=\db B_{q}\var+B_{q+1}\db\var.$
Furthermore, for $i =q,q+1$ and $s>0$ we have
for $ 0< r_1< r_2<1$
\eq{Bspsi-}
|B_i \psi|_{U_{r_1},s+1}
\leq \frac{C_r}{(\min(1-r_2, r_2-r_1))^{s+C_n}}(|\psi|_{U_{r_2},s}+\|\psi\|_{U_1,0}).
\eeq
\ele
Note that by a partition of unity, we obtain for open sets $U'\Subset U''\Subset U_1$,
\eq{Bspsi}
|B_i \psi|_{U',s+1}
\leq \frac{C_r}{(\min(\dist(U',\pd U''), \dist(U'',\pd U_1)))^{s+C_n}}(|\psi|_{U'',s}+\|\psi\|_{U_1,0}).
\eeq
\begin{prop}\label{smooth-B} Let $M$ be a domain in a complex manifold $\cL M$ and let $\var $ be a $(0,q)$ form  defined on $\widetilde M$ with $q>0$. Suppose that $M$ is  relatively compact   in $\widetilde M$. Then there exist linear operators $B_j$ such that   $K_q^1\var:=\var -\db B_q\var-B_{q+1}\db \var $   satisfies
\ga{}
|K_q^1\var|_{\widetilde M,r}\leq C_{r}\|\var |_{\Lambda^r(\widetilde M)},\quad  
\|K_q^1\var \|_{M,r}\leq C_{r}\|\var \|_{\widetilde M,0}, \\
|\db K_q^1\var|_{\widetilde M,r}\leq C_{r}|\db\var |_{\widetilde M,r},\quad 
\|\db K_q^1\var \|_{M,r}\leq C_{r} \|\db\var \|_{\widetilde M,0}, \\ 
\|B_j\psi\|_{\Lambda^{r+1}(M)}\leq C_{r}|\psi |_{\Lambda^r(\widetilde M)},\  j=q,q+1, \quad \forall  r>0.
\label{Bjpsi}
\end{gather}
In particular, $K_q^1\var\in C^\infty(M)$ for $\var\in C^0(\widetilde M)$. 
\end{prop}
\begin{proof}We first find $p_1,\dots, p_N\in \ov M$ and coordinate balls $\om_j,\om_j', U_j$ such that $p_j\in\om_j\Subset\om_j'\Subset U_j$ and $\ov M\subset\cup_{j=1}^N\om_j\Subset\widetilde M$.   We may assume that $U_j\Subset\widetilde M$ are coordinate balls centered at $p_j$.  
 We also have a homotopy formula $\psi=\db B_{U_j,q}+B_{U_j,q+1}\db \psi$ on $U_j$. Choose $\chi_j\in C^\infty_0(\om_j')$ with $\chi_j=1$ in a neighborhood of $\ov{\om_j}$.    

Let $\var_0:=\var$ be a $(0,q)$ form on $\widetilde M$. We first define $\var_1,\dots, \var_N$ inductively by
 \eq{varj-def}
 \var _{j}=\var_{j-1}-\db(\chi_{j} B_{U_{j},q}\var_{j-1})-\chi_{j}B_{U_{j},q+1}\db \var_{j-1} \quad \text{ on $\widetilde M$}.
 \eeq
Then 
 \al{}\label{varj}\var_j&=(1-\chi_j) \var_{j-1} -\db\chi_{j}\wedge B_{U_j,q}\var_{j-1},\\
 \db   \varphi_j&=(1-\chi_j) \db \var_{j-1}-\db \chi_j\wedge B_{U_j,q+1}\db \var_{j-1}.
 \label{dbvarj}
 \end{align}
Since $\supp\db\chi_j\Subset U_j$, by \re{Bspsi} the two formulas immediately yield
\ga{}\label{vavaj}
|\var _j|_{\widetilde M,r}\leq C_r|\var_{j-1}|_{\widetilde M,r}\leq  C_r^j|\var_0|_{\widetilde M,r},
\quad |\db\var _j|_{\widetilde M,r}\leq   C_r^j|\db\var_0|_{\widetilde M,r}.
\end{gather}
Let us show that with $\hat\om_j:=\om_1\cup\cdots\cup\om_j$, we have
\ga{}
\label{varj0}\|\var_{j}\|_{\cL N_{\del_j}(\hat\om_j),r}\leq C_{r,\delta_j}\|\var_0\|_{\widetilde M,0}, \quad 
\|\db\var_{j}\|_{\cL N_{\del_j}(\hat\om_j),r}\leq C_{r,\delta_j}\|\db\var_0\|_{\widetilde M,0},
\end{gather}  
where $\delta_j>0$ will be chosen. 
In fact, $\var _1=0$ on $\cL N_{\del_1}(\om_1)$ if $\chi_1=1$ on $\cL N_{\delta_1}(\om_1)$ and hence the  inequality holds trivially for $j=1$; for the induction step,  we can only prove the above weaker form.

Suppose that \re{varj0} holds and we want to verify it when $j$ is replaced by $j+1$. We observe that $\chi_{j+1} \equiv 1$ on $\Nc_{\del_{j+1}}(\om_{j+1})$ when $\delta_{j+1}$ is sufficiently small and $0<\del_{j+1}<\delta_j/2$. Suppose $\delta\in(0,\delta_j)$. Then by \re{varj0}
$$
\|(1- \chi_{j+1}) \var _j\|_{\cL N_{\delta_{j+1}}(\hat\om_{j+1}),r}=\|(1- \chi_{j+1}) \var _j\|_{\cL N_{\delta_{j+1}}(\hat\om_{j}),r}\leq  C'_{r,\delta_{j}}\|\var_{0}\|_{\widetilde M,0}.
$$ 
Since $\supp \db\chi_{j+1}\subset\om_{j+1}'$ and latter contains  $\cL N_{\delta_{j+1}}(\om_{j+1})$ for small $\delta_{j+2}$, then  
\begin{align*}
\|\db  \chi_{j+1}&\wedge  B_{U_{j+1},q}\var _{j}\|_{\cL N_{\delta_{j+1}}(\hat\om_{j+1}),r}\leq \|\db  \chi_{j+1}\wedge  B_{U_{j+1},q}\var _{j}\|_{\cL N_{\delta_{j+1}}(\hat\om_{j})\cup\cL N_{\delta_{j+1}}(\om_{j+1}),r}\\
&\leq \|\db  \chi_{j+1}\wedge  B_{U_{j+1},q}\var _{j}\|_{\cL N_{\delta_{j+1}}(\hat\om_{j})\cap\om_{j+1}',r}
\leq C_r\|   B_{U_{j+1},q}\var _{j}\|_{\cL N_{\delta_{j+1}}(\hat\om_{j})\cap\om_{j+1}',r}.
\end{align*}
Note that 
$
\cL N_{\delta_{j+1}}(\hat\om_j)\cap\om_{j+1}'\Subset \cL N_{\delta_{j}}(\hat\om_j)\cap \cL N_{\delta_{j+1}}(\om_{j+1}')\Subset U_{j+1}.
$
By \re{Bspsi}, we obtain  
$$\|   B_{U_{j+1},q}\var _{j}\|_{\cL N_{\delta_{j+1}}(\hat\om_{j})\cap\om_{j+1}',r}\leq C_r(|\var_j|_{\cL N_{\delta_{j}}(\hat\om_j)\cap \cL N_{\delta_{j+1}}(\om_{j+1}'),r}+\|\var\|_{\widetilde M,0})\leq C_r'\|\var\|_{\widetilde M,0}.
$$
Combining the last three displayed estimates and using formulas \re{varj}, we see that \re{varj0} holds for $\var_{j+1}$. The estimate for $\db\var_{j+1}$ can be obtained by the same argument via formula \re{dbvarj}. 

Finally, as in \rp{K-ext} we have the desired estimates for $K^1_q\var:=\var_N$ and 
$$
B_q\var := \sum_{j=1}^N \chi_{j}   B_{U_j,q} \var_{j-1}, \quad B_{q+1}\psi := \sum_{j=1}^N \chi_{j} B_{U_j,q+1}\psi_{j-1}. $$
where $\var_j$ are defined by \re{varj} and $\psi_j$ is defined by \re{dbvarj} in which $\db\var_j$ is replaced by $\psi_j$ for $j=1,\dots, N$.\qedhere \end{proof}

 Applying \rl{smooth-B} in which $\var$ is $K_q^0\var$, we immediately obtain the following.

\begin{prop}\label{full-bumping} Let $r\in(0,\infty)$, $\ve>0$, $0\leq\delta\leq1/2$,  and $q\geq 1$.
Let $M\Subset\cL M\colon\rho<0 $ be a    domain with $C^2$ boundary satisfying the conditions $a_q^+$. Let $V$ be a holomorphic vector bundle of finite rank over $\cL M$. Then
$$
\var=\db H_q\var+H_{q+1}\db \var +K_q\var
$$
for any  $(0,q)$ forms $\var $ on $M$ such that $\var , \db \var \in\Lambda^r( M, V)$. Here $K_q\var$ are $\db$ closed  $V$-valued $(0,q)$ forms on $\tilde M\Supset\ov{M}$.  Further, $H_q,H_{q+1}, K_q$ are independent of $r$ and we have the following.
\bpp
\item If $ b M\in C^2$ is strictly $(n-q)$ convex, then
\aln{}
|H_j\var|_{M,r+1/2}&\leq C_{r,\ve}(\nabla\rho,
\nabla^2\rho)|\var |_{M,r}, \quad j=q,q+1;\\
\|K_q\var\|_{\tilde M,r}&\leq C_{r,\ve}(\nabla\rho,
\nabla^2\rho)\|\var \|_{M,\ve}.
\end{align*}
\item If $b_{q+1}^- M\in\Lambda^{r+2+\delta}$,  then
for $j=q,q+1$
\aln{}|H_j\var|_{M,r+\delta}&\leq C_{r,\delta,\ve}(\nabla\rho,
\nabla^2\rho) ( |\rho|^*_{r+2+\delta}\|\rho\|_{2+\ve}^{*b}|\var |_{M,\ve}+|\var |_{M,r}),\\
\|K_q\var\|_{\tilde M,r}&\leq C_{r,\ve}(\nabla\rho,
\nabla^2\rho)\|\rho\|^{*b}_{2+\ve}\|\var \|_{M,\ve}.
\end{align*}
\epp
   Furthermore, the positive constants $b, C_{r,\ve}(\nabla\rho,
\nabla^2\rho)$ and  $C_{r,\delta,\ve}(\nabla^\rho,\nabla^2\rho)$ are  stable under small $C^2$ perturbations of $\rho$.
\end{prop} 
\begin{proof}By Propositions~\ref{K-ext} and \ref{smooth-B}, we obtain
$$
\var=\db T_q\var+\db B_qK_q^0\var+T_{q+1}\db\var+B_{q+1}\db K_q^0\var+K_q^1K_q^0\var.
$$
Define $H_q\var=T_q\var+B_qK_q^0\var$ and $H_{q+1}\db\var=T_{q+1}\db\var+B_{q+1}\db K_q^0\var$. Define $K_q\var=K_q^1K_q^0\var$. Note that $H_{q+1}\var$ is well-defined. Indeed, if $\db\var=\db\tilde\var$ on $M$, we have $T_{q+1}\db\var=T_{q+1}\db\tilde\var$. Also, by \re{Bjpsi} and estimate \re{dbK0} for $K_q^0$, we have
$$
|B_{q+1}\db K_q^0\tilde\var-B_{q+1}\db K_q^0\var|_{M,\ve}\leq C_\ve|\db K_q^0\tilde\var-\db K_q^0\var|_{\widetilde M,\ve}
\leq C_\ve'|\db  \tilde\var-\db  \var|_{M,\ve}=0.
$$
The estimates follow from Propositions~\ref{K-ext} and \ref{smooth-B}. Note that   $\|\var\|_\ve$, instead of $\|\var\|_{\ve}$
is need in this proposition,  since estimates in \rp{K-ext} requires $r>0$.  
\end{proof}

\subsection{An application of $\db$-Neumann operators}

To construct the global homotopy formula by removing the compact operator in the approximate homotopy formula, we will use the $\db$-Neumann solution operators in Kohn~\cites{MR0153030,MR0208200} and Kohn--Rossi~\ci{MR177135}. Suppose that $\cL M$ is a hermitian manifold and $V$ is a hermitian holomorphic vector bundle on $\cL M$. Let $M$ be a relatively compact domain in $\cL M$ with smooth boundary. (We will take $M$ to be the $\widetilde M$ in our application.) Then $M$ is defined by $\rho<0$ with $|d\rho|=1$ at each point in $\pd M$.  Locally, for each $\zeta\in b M$, we can choose a special unitary frame $\om^1,\dots, \om^n$ for $(1,0)$ forms on a neighborhood $U$ of $\zeta$ such that on $ b M\cap U$, $\om^n=\sqrt 2\pd\rho$, and a unitary frame $e_1,\dots, e_m$ for $V$ such that a $V$-valued $(0,q)$ form $\var$ can be written as
$$
\var=\sum\var_{I}^j\bar\om^I\otimes e_j, \quad \db\var=\sum\var^j_{I;\all}\bar\om^\all\wedge\bar\om^I\otimes e_{j}.
$$
Let $dV$ be the volume on $\cL M$ with respect to the hermitian metric on $\cL M$. For $\var,\psi\in L^2_{(0,q)}(M,V)$, define
 $$(\var,\psi):=\int_M\jq{\var,\psi} \,  dV,\quad \jq{\var,\psi}=\sum\var_I^j\ov{\psi_I^j},\quad \|\var\|^2:=(\var,\var).
 $$
 Let $\db_q\colon L^2_{(0,q)}(M,V)\to L^2_{(0,q+1)}(M,V)$ be the standard closed and densely defined $\db$ operator.
 Let $\dom{\db_q}$ be the domain of $\db_q$.     Define $g\in\dom {\db_q^*}\subset L^2_{(0,q+1)}(M,V)$ and write $\db^*g=v$,  if $(\db_q u,g)=(u,v)$ for $u\in\dom\db_q$. Define $\db$-Neumann Laplacian $\Box_q=\db_{q}^*\db_q+\db_{q-1}\db_{q-1}^*$ for $q>0$, where we use convention   $\db_n^*=0$. 
We recall the following.  
 \begin{thm}[\cite{MR177135}, Thm.~3.9;   \cite{MR0461588}, Thm.~3.1.14, p.~51 and p.~77] \label{KRthm}
 If $bM$ satisfies the $a_q$ condition with $q>0$ and $ b M\in C^\infty$, then for any hermitian vector bundle $V$ on $\cL M$ there exists a bounded operator $N_{q}\colon L^2_{0,q}(M,V)\to L^2_{0,q}(M,V)$ satisfying the following
 \bpp\item $N_q L^2_{(0,q)}(M,V)\subset\dom(\Box_{q})$ and
 $$
\var =\Box_q N_q \var +H^0_q\var,\quad H^0_q\var\in\cL H_{(0,q)}(M,V), \qquad\forall \var \in  L^2_{(0,q)}(M,V),
 $$
 where $\cL H_{(0,q)}(M,V)$ is the space of $V$-valued harmonic $(0,q)$-forms. 
 \item $N_{q}$ commutes with   $\Box_q$ on $\dom\Box_q$. Assume further that $N_{q+1}$ is defined on $L^2_{(0,q+1)}(M,V)$. Then
  $N_{q+1}\db_q\var=\db_q N_{q}\var$ if $\var\in\dom{\db_q}$.
 \item $N_{q}(C^\infty_{(0,q)}(\ov M,V))\subset C^\infty_{(0,q)}(\ov M,V)$ and $\|N_q\var\|_{H^{s+1}(M)}\leq C_s \|\var\|_{H^s(M)}$ for all $\var\in C^\infty_{(0,q)}(\ov M)$.
 \epp
 \end{thm}
  \begin{rem}As mentioned in Folland-Kohn~\ci{MR0461588}*{p.~51},
 $N_{q+1}$ is defined whenever $\Box_{q+1}$ has closed range, whether the basic estimate (i.e. condition $a_{q+1}$) holds or not for $(0,q+1)$ forms.
 \end{rem} 

 \begin{thm}\label{KRthm+} Let $r>0,\ve>0$ and $0\leq\delta\leq1/2$.
 Suppose that $ b   M$ satisfies conditions $a_q^+$ with $q>0$. Assume that  $r\in(0,\infty)$. Then for $\var,\db\var\in\Lambda^r(M,V)$ we have on $M$
 \eq{HF-K}
 \var =\db R_{q}\var +R_{q+1}\db \var +\tilde H_{q}\var,
 \eeq
where $\tilde H_{q}\var\in\cL H_{(0,q)}(\widetilde M,V)$ and $\widetilde M\Supset M$ is an $a_q^+$ domain with smooth boundary.
\bpp
\item If $ b M\in C^2$ is strictly $(n-q)$ convex, then
 \al{}\label{Rjfe-sc}|R_s\var|_{r+1/2}&\leq C_r  |\var |_r,\quad 
\|\tilde H_{q}\var\|_{r}\leq C_{r,\ve} |\var \|_\ve.
\end{align}
\item 
  For $j=q,q+1$, we have
 \al{}\label{Rjfe}|R_j\var|_{r+\delta}&\leq C_{r,\delta,\ve}   ( |\rho|^*_{r+2+\delta}|\rho|_{2+\ve}^{*b}|\var |_\ve+|\var |_r),\\
\|\tilde H_{q}\var\|_{r}&\leq C_{r,\ve} \|\rho\|_{2+\ve}^{*b}\|\var \|_\ve.
\label{tHqv}
\end{align}
\epp
Furthermore, if $0<\del_0<\del_1$,
 then for any $\del',\del''$ satisfying $\del_0<\del'<\del''<\del_1$, we have
\al{}\label{Rjf}|R_s\var|_{M\setminus \cL N_{\del''}(b_{q+1}^-M), r+1/2}&\leq \f{ C_{r,\ve} }{ (\del''-\del')^{r+c_*}}\Bigl (  \|\var \|_\ve+ |\var |_{M\setminus \cL N_{\del'}(b_{q+1}^-M), r}\Bigr). 
\end{align}
Here  $ C_{r,\ve}=C_{r,\ve}(\nabla\rho,
\nabla^2\rho)$  are stable under small $C^2$ perturbations of $\rho$. 
 \end{thm}
 \begin{proof} By \rp{full-bumping}, we have $\var=\db H_q\var+H_{q+1}\db\var+K_q\var$, where $K_q\var$ is defined on $\widetilde M$. 
 By \rt{KRthm} $(a)$--$(b)$ applied to the domain $\widetilde M$, we decompose
 $$
 K_q\varphi=\db \tilde T_q\var+\tilde T_{q+1}\db\var+{ H}^0_{q}K_q\var
 $$
 with
 $$
\tilde T_q\var=\db^*N_qK_q\var, \quad \tilde T_{q+1}\db\var=\db^*N_{q+1}\db K_q\var.
 $$
 Here $N_s,H^0_q$ are the corresponding operators in \rt{KRthm} for the $a_q^+$ domain $\tilde M$. 
 Thus, we have obtained 
 $
 \var=\db R_{s}\var+R_{q+1}\db\var+\tilde H_q\var
 $
 with
 $
 R_s\var:=H_q\var+\tilde T_s\var$ and $\tilde H_q\var:={ H}^0_{q}K_q\var.
 $

 We now need to derive interior estimates for $N_j$ for $j=q,q+1$. Recall that since $N_j$ exists, we have
 $$
 \Box N_j\tilde \var_j =\tilde \var_j-H^0_j\tilde \var_j.
 $$
 We will take $\tilde\var_q=K_q\var$ and $\tilde\var_{q+1}=\db K_q\var$. 
  
  We now derive the estimates using local interior estimates for the elliptic system that is valid for any relatively compact subdomain of $\widetilde M$, by using the global $L^2$ estimate in \rt{KRthm} 
 $$
\| N_j\tilde \var_j\|_{L^2(\widetilde M)}\leq C_q(\widetilde M)\|\tilde \var_j\|_{L^2(\widetilde M)}.
 $$
By the regularity of $\db$-Neumann operator, we know that for $a_q$ domain $\widetilde M$ which has $C^\infty$ boundary, harmonic forms on smooth on $\ov{\widetilde M}$.   Therefore,
$$
|H_j^0\tilde \var_j|_{M,r}\leq C_{r}\|\tilde \var_j\|_{L^2(\widetilde M)}, \quad \forall r>0.
$$
 
The interior estimation is  local and we may assume that  $M$ is contained in $\cc^n$.   We still have
 $$
 \Box N_j\tilde \var_j=\tilde \var-H_j^0\tilde \var_j.
 $$
Let $\chi$ be a smooth function with compact support on $\cL N_\del(\ov M)$.
Let $u=\chi N_j\tilde \var_j$. Then $$\Box u=v,\quad 
v:=a(\nabla\chi)\nabla N_j\tilde \var+b(\nabla^2\chi)N_j\tilde \var_j+\chi\tilde \var_j-\chi H_{j}^0\tilde \var_j.
$$

Next,  we recall two interior estimates on systems of elliptic equations in Morrey~\cite{MR0202511}*{Thm 6.4.4., p.~246}:
\gan{}
\|\tilde u\|_{W^{k+2,p} }\leq C_{k,p} \|v\|_{W^{k,p}}+C_R\|\tilde u\|_{W^{0,1}},\  1<p<\infty;\\
\|\tilde u\|_{k+2+\all}\leq C_{k,\all}\|v\|_{k+\all}+C_R\|\tilde u\|_{Lip},
\end{gather*}
provided  the right-hand sides are finite and $ \supp\tilde u\subset B_R$. Note  that the latter is satisfied by our $\tilde u$. We can easily verify (see the proof of Prop.~8.5 in~\ci{aq}) that
$$
\|\tilde u\|_{k+2+\all}\leq C_{k}(\|\chi(\tilde \var-H_j^0\tilde \var_j)\|_{k+\all}+\|u\|_{L^2})\leq C_k'\|\chi\|_{k+1}\|\tilde \var_j\|_{k+\all}+C_k\|u\|_{L^2}.
$$
Since $\|u\|_{L^2}\leq C\|\tilde \var_j\|_{L^2}$,
we obtain 
$$
\|\tilde u\|_{k+2+\all}\leq    C_k'\|\chi\|_{k+1}\|\tilde \var_j\|_{k+\all}+C_k\|\var_j\|_{C^0(\tilde M)}.
$$
Combining this estimate with the estimates \rp{full-bumping}, we can obtain \re{Rjfe-sc}-\re{tHqv}.
The details are left to the reader. 

To estimate \re{Rjf}, we take a smooth function $\tilde \chi$ on $\cL M$ that has compact support in $\ov M\setminus \cL N_{\del'}(b_{q+1}^-M)$ and equals $1$ on $M\setminus \cL N_{\del''}(b_{q+1}^-M)$. We may assume that $\|\chi\|_{C^k(\cL M)}\leq C_k(r'-r')^{k+1}$ where $C_k$ is independent of $\del',\del''$.   Applying the above estimates to $u=\tilde \chi N_j\tilde\var$, this gives us \re{Rjf}. We left the details to the reader. \end{proof}

\subsection{Stability of cohomology groups and proof of \nrc{CnNN}}
We also need  to identify the cohomology groups after bumping. Let us collect some known results. Let $\overline H_{(0,q)}(M,V)$ be the quotient  of $\dbar$-closed $V$-valued $(0,q)$ forms on $M$ by $L^2_{(0,q)}(M,V)\cap\dbar L^2_{(0,q-1)}(M,V)$.
\begin{prop}\label{coh-g}
Let $M\Subset\cL M$ be an $a_q$ domain with $C^2$ boundary defined by $\rho<0$ and let $M^a$ be defined by $\rho<a$. Let $V$ be a hermitian holomorphic  vector bundle on $\cL M$. 
Then we have
\bpp \item   Suppose that $M^{-c}$ are $a_q$ domains for $0<c\leq c_*$. Then the restriction of  $\ov H_{(0,q)}(M,V)$ on $\ov H_{(0,q)}(M^{-c_*},V)$ is injective.
\item  $\ov H_{(0,q)}(M,V)$ and $\cL H_{(0,q)}(M,V)$ are isomorphic. 
\item If $H^{q}( M,\cL O(V))=0$, then $\cL H_{(0,q)}(\widetilde M,V)=0$, where $\widetilde M$ is given by \rta{KRthm+} via the Grauert bumping; hence, \rta{Thm::glob_hf_intro} holds.
\item 
If $M$ further satisfies $a_{q}^+$ condition, then  $ H^q(M,\mathcal O(V))$  and  $\cL H_{(0,q)}(M,V)$ are isomorphic.\epp
\end{prop}
\begin{proof}By  Dolbeault isomorphism, we have $H^{(0,q)}(\cL M,V)\approx H^q(\cL M,\cL O(V))$ for a holomorphic vector bundle $V$ on a complex manifold $\cL M$.

$(a)$   is in H\"ormander~\ci{MR0179443}*{Thm.~3.4.6} when $V$ is trivial and $\pd M\in C^3$; see also Appendix~B in~\ci{aq} for the general case. 

$(b)$ is in Folland-Kohn~\ci{MR0461588}*{p.~77} and $(d)$ is in Folland-Kohn~\ci{MR0461588}*{Thm.~4.3.1}, when $ b  M\in C^\infty$. The restriction for $ b  M\in C^\infty$ can be relaxed to $ b  M\in C^2$; see  Appendix~B in~\ci{aq}.

$(c)$ Suppose that $f$ is an $V$-valued $(0,q)$ form on $M'$ satisfying $\db_q f=0$ and $\db_{q-1}^*f=0$. 
Then $f\in C^\infty$. Now $H^{(0,q)}(M,V)=0$ implies that $f=\db u$ for some $u\in C^\infty(M)$. 
Since $u\in L^2(M^{-c})$, then $f=\db\tilde u$ for some $\tilde u\in L^2(M')$ by $(a)$. Thus
$({f,f})=({f,\db u})=({\db^*f,u})=0$.
\end{proof}

Before we prove \rt{Global NN}, let us first prove its corollary.

\begin{proof}[Proof of \nrc{CnNN}]
Shaw~\ci{MR2885124}*{Thm.~2.2} showed that when $D_1,D_2$ are bounded strongly pseudoconvex domains in $\cc^n$ with smooth boundary with $D_1\Subset D_2$ and  $n\geq3$,  the harmonic space $\cL H_{(0,1)}(D_2\setminus \ov{D_1})$ vanishes. Since the tangent bundle of $\cc^n$ is trivial, applying this result to $\tilde M$ when $M$ is contained in $\cc^n$, we obtain $\cL H_{(0,1)}(\widetilde M,\Theta)=0$. Therefore, \nrc{CnNN} follows from \rt{Global NN} directly. 
\end{proof}

\begin{rem}Note that we need  $n\geq 4$ instead of $n\geq3$ for \nrc{CnNN}, because our proof relies on regularity of $N_1$ and $N_2$. 
On $\cc^n$, there is no need of time-one maps. It is plausible that the requirement of $r_0>5/2$ in \nrc{CnNN} can be lowed to $r_0>3/2$ as in Shi~\ci{shiNN}. 
\end{rem}

As mentioned in Folland-Kohn~\ci{MR0461588}*{p.~77}, the relations between these important cohomology groups  are non-trivial. It seems that it is unknown if the conclusion in $(c)$ holds for $a_q$ domains.
For vanishing of cohomology groups, see Chakrabarti--Laurent-Thi\'ebaut--Shaw~\ci{MR3798858}*{Thm.~1.8}. 
For recent work on stability on spectrum of the $\db$-Neumann Laplacian for families of domains, see Fu--Zhu~\ci{MR4359484}.

\section{Estimates for the main inductive step}\label{set-up}
\setcounter{thm}{0}\setcounter{equation}{0}

\newcommand{\fa}[1]{\all(1-\f{#1-s}{m-s})}

In this section and next, we prove \rt{Global NN} and a detailed version \rt{Global NN mv}, using \rt{Thm::glob_hf_intro}. The proof does not need condition $a_1^+$ and it only requires the existence of homotopy formulas, the estimates of the homotopy operators and the $a_1$ condition to ensure the stability of the existence of homotopy formulas under small $C^2$ perturbations of the domain $M$ as stated in \rt{Thm::glob_hf_intro} and an use of Hartogs's extension in Subsection~\ref{subsec:8.4} is valid.

This section provides the inductive step for a Nash-Moser iteration. We finish the proof in Section~\ref{sect:hartogs}.
\subsection{Strongly pseudoconvex case}
Let $M$ be a strictly pseudoconvex  domain in $\cL M$. Given the initial integrable almost complex structure on $M$, we want to find a transformation $F$ defined on $\ov M$ to transform the structure to a new structure closer to the standard complex structure while $\ov M$ is transformed to a new domain that is still strictly pseudoconvex.

Our homotopy formula requires that $ A$ and $\db A$ to be $\Lambda^r$ for $r>0$. We need further initial conditions to derive estimates. Thus, we will assume the following initial conditions
\ga\label{t-a}
t^{-\all}|A|_s<1,\quad \all>1,\  s>0;
\\
\label{Lmte}
|A|_m\leq L_m,\quad m>1.
\end{gather}
Eventually, we will need $s>1$ and $m>2$ as we carry out the required estimates.
 Here and in what follows, we write $|\cdot|_{M,r}$ as $|\cdot|_r$.
Then
\begin{equation} \label{srm}
|A|_{m'}\leq  C_{s,m',m} t^{\fa{m'}}|A|_m, \quad m\geq m'\geq s,
\end{equation}
where  $t$ is the parameter of the smoothing operator which will converge to $0$ in the iteration.
We shall later verify condition \re{srm} at each induction step.

We will use $C_m$ (resp. $C_s,C_r$ etc.) to denote a constant depending on $m$, and which is   stable under small $C^2$ perturbations of the domain $M$.

We apply the homotopy formula   to $A$ on $\ov M$ so that on $M$
\[
 A = \dbar P A + Q \dbar A, \qquad A\in \La^r, \quad r>1,
\]
where $P=R_1$ and $Q = R_2$ are given by \rt{KRthm+}.  We have estimates
\ga\label{PA-4}
|P\var|_{ r+1/2}\leq C_r(\hat
\pd^2\rho) |\var|_{ r},\quad  r>0,\\
|Q\psi|_{ r+1/2}\leq C_r(\hat
\pd^2\rho) |\psi|_{ r},\quad  r>0.
\label{PA-4+}
\end{gather}
Here $PA$ is a global vector field of type $(1,0)$ of class $\Lambda^{r+1/2}$.

On $\ov M $, we use the time-one flow of $v$:
\[
    F = F_v^1\quad \text{on $M$},  \qquad v:=-S^t P  A\in C^\infty(\ov{\cL N_t(M)}; \Theta),
\]
where $S^t$ is the smoothing operator given by \re{sm_op_vf}. We extend $v$ from $\cL N_t(M)$ to $\cL M$ so that
$$
\|v\|_{\cL M,r}\leq C_r(M)\|v\|_{\cL N_t(M),r}, \quad |v|_{\cL M,r}\leq C_r(M)|v|_{\cL N_t(M),r}.
$$
Here $C_r$ is independent of $t$ and stable under small perturbations of $M$. To apply \rp{Lem:flow}, we need
\aln{}
C_0(M) \|v\|_0\leq t,\quad
C_1(M) \|v\|_{1+\ve}\leq \delta_{1,\ve},\quad
C_{2,\ve}(M)\|v\|_{\cL N_t(M),2+\ve}\leq\del_{2,\ve},
\end{align*}
(Here the parameter $\ve$ will be determined only via \re{last-c}-\re{last-c+} below.) Thus, the above conditions hold since   $C |v|_{s+1/2}  \leq CC_s|A|_{s}\leq CC_st^\all<t$ for $\all>1$ and
$t\in[0,t_0^*]$, where $t_0^*>0$ is fixed.  This allows us to replace
all norms $|v|^*_r$ in \rp{Lem:flow} by $|v|_{\cL N_t(M);r}$.
Then  \rp{Lem:flow} gives us
$$M^1=F(M)\subset\cL N_{C_n\|v\|_0}(M), \quad
F_v^{-1}\colon \cL N_{C_n\|v\|_0}(M)\to
\cL N_{C_n(C_n+1)\|v\|_0}(M),
$$
and $F^{-1}=F^{-1}_v|_{M^1}$. Furthermore, we can use the estimates in \rp{Lem:flow}.

In local charts $(U_k, \psi_k )$, we have
\[
  F_k = I + f_k, \quad f_k = v_k+w_k.
\]
 where $v_k (z_k) = (- S^t P A)_k$,  
 $z_k = \psi_k$, 
 and $w_k$ regarded as smaller than $v_k$ is given as in \rp{Lem:flow}. Thus, $w_k$ is defined on $\cL N_t( D_k)$, where $D_k = \psi_k(U_k \cap M)$. 
 Here by an abuse of notation, we denote $w_k(\cdot,1)$ by $w_k(\cdot)$.  We have
\[
  A_{k}=\sum A_{k,\ov\beta}^\all d\ov z_k^\beta \otimes \DD{}{ z_k^\all}, \quad f_k=(f_k^1,\dots, f_k^n).
\]
By \rl{elem_pp},
the new structure $A'$ takes the form
\begin{equation} \label{A'cF}
  A'_k \circ F_k = (I + \dbar \ov f_k + A_k \dbar f_k)^{-1} (A_k + \dbar f_k + A_k \pa f_k).
\end{equation}
  Thus on $D_k$, we have
\begin{align*}
A_k&+\dbar f_k + A_k \d f_k
=A_k- \dbar(S^t P  A)_k - A_k \pa (S^t P A)_k + A_k \pa w_k + \db w_k\\
&=A_k-(S^t \dbar P  A)_k - A_k \pa (S^t P A)_k + A_k \pa w_k + \db w_k + ([S^t,\db]PA)_k.
\end{align*}
We now apply the homotopy formula for a second time as in~\ci{MR2868966}, by writing
$$
\db  PA=A-Q\db A.
$$
We obtain for $A\in\La^r$ with $r>1$
\begin{align*}
A_k&+\dbar f_k + A_k \d f_k
\\ &= A_k- (S^t A)_k  + ( S^t Q \dbar A)_k - ( A \pa S^t PA)_k +A_k \pa w_k  + \db w_k + ([S^t,\db]PA)_k.
\end{align*}
We shall use the following notation
$$
\begin{array}{lll}
  I_1 = [(I - S^t)A]_k, &  I_2 = [ S^t Q \dbar A]_k, &
  I_3 = [ [S^t, \dbar]P A]_k,
 \vspace{.75ex} \\
I_4 = -(A \pa S^t P A)_k,&
 I_5= A_k \pa w_k + \db w_k, & \Kc_k= \ov{\pa f_k} + A_k \dbar f_k.
\end{array}
$$
We remark that among $I_1,\dots, I_4$, only $I_2$ requires $A\in\Lambda^r$ with $r>1$. The other terms are well-defined for $A\in\Lambda^r$ with $r>0$.
Of course, $I_j$ depend on $k$.
Consequently, we can rewrite \re{A'cF} as
\[
  \wti A_k = (I + \Kc_k)^{-1} \Bigl( \sum_{j=1}^5 I_j \Bigr), \quad \wti A_k = A'_k \circ F_k.
\]

We first estimate the $I_j$-s. By \rl{StM} and noting that $S^t EA|_{M} = S^t A|_{M}$, we get
\begin{equation} \label{I1_est}
 |I_1|_{ D_k;r} \leq |(I - S^t)A|_{ r}
  \leq C_{m,r} t^{m-r} |A|_{ m}, \quad 0<r\leq m<r+m_*,
\end{equation}
where $m_*$ is a fixed large number.
Thus
$$
|I_1|_{ D_k;s}<C_{s,m} t^{m-s}|A|_m
\leq C_{s,m} L_m t^{m-s}  \leq t^{d\all}
$$
provided
\ga\label{d1}
m-s>d\all+ \ve,\quad  C_{s,m}L_m t^{\ve}<1.
\end{gather}
We need to take
$$
d>1.
$$
By \re{I1_est}, we also have
$$
  |I_1|_{ D_k;r}\leq C_{r}|A|_{r}, \quad \forall r>0.
$$

For $I_2$,
we need to use the integrability condition $\dbar A =-\f{1}{2} [A, A]$. Recall that $[A,A]$ is a sum of products of $A$ and its first-order derivatives. Thus  
\eq{[A,A]}
 |[A, A]|_{r-1}\leq C_{r,\ve }|A|_\ve|A|_{r}, \quad r>1.
\eeq
 Since $M$ is strongly pseudoconvex, we have by \rt{KRthm+}
\eq{QdbA}
|Q \db A|_{r+\yh} \leq C_r |\db A|_r\leq  C_{r,\ve }\|A\|_\ve|A|_{r+1},\quad r>0.
\eeq
Applying \rl{StM}, we get
\begin{equation} \label{I2_est}
  \begin{aligned} 
   |I_2|_{ D_k;r} &= | S^t Q\dbar A|_{ r}
  \leq C_rt^{-\yh} | Q\dbar A|_{ r-\yh}
  \\ &\leq C'_r t^{-\yh} |\dbar A|_{  r-1} \leq C_{r,\ve}' t^{-\yh} |A |_{\ve} |A |_{ r},\quad r>1.\nonumber
\end{aligned}
\end{equation}
In particular, we have
\ga\label{s>1}
|I_2|_{ D_k;s}\leq  C_s t^{-\yh+2\all}<t^{d\all}
\end{gather}
by the last inequality in \re{d1},
 provided
\eq{d2}
2\all-\yh> d\all+\ve,   \quad s>1.
\eeq
Here we emphasize that \re{s>1} requires $s>1$.
Also
$$
| I_2|_{ D_k;r} \leq  C_{r}|A|_{r}, \quad\forall r> 1.
$$
 
By \rl{StM} and \re{PA-4} we get
  \begin{align*}
  |I_3|_{ D_k;m'} &= |[S^t, \dbar]PA|_{m'}
 \leq C_{m',r}   t^{r-m'}|PA|_{ r}
\\ &\leq C'_{m',r} t^{r-m'} |A|_{  r-1/2},
\nonumber
\quad   r>\f{1}{2},\  m'>0,\ |r-m'|<m_*.   
\end{align*}
Thus, we simply use the estimates for $I_1$ to estimate $I_3$. %

Next, we estimate $I_4$. By using \re{zprule}, \re{Stvn} and \re{PA-4} we have
\begin{align*} 
|I_4|_{ D_k;r}
&= |A \pa S^t P A|_{ r}
 \\ &\leq C_{r,\ve } \left( |A|_{ r} |\pa S^t P A|_{ \ve} + |A|_{ \ve} | \pa S^t P A |_{  r} \right)\nonumber
 \\ &\leq C'_{r,\ve } \left( |A|_{ r} | S^t P A |_{ 1+\ve} + |A|_{ \ve} | S^t P A  |_{  r+1} \right)\nonumber
 \\ &\leq C''_{r,\ve } \left( |A|_{ r}  t^{-1/2} |A|_{ \ve} + t^{-\yh} |A|_{ \ve} |A|_{ r} \right),\quad r>0.\nonumber
\end{align*}
Applying the above with $s$ in place of $r$, we get
$$
  |I_4|_{ D_k;s} \leq 2C_{\ve, s}'' t^{{-\yh+2\all}}<t^{d\all} 
$$
by \re{d2} for a possibly larger $C_s$. Also
$$
  |I_4|_{ D_k;r} \leq  C_{r}|A|_{r}, \quad r > 0.
$$

We now estimate $I_5 = A \pa w + \db w$.
Recall that $
 v =-S^t PA$.  By   \rt{KRthm+}, we obtain
\gan{}
|v|_{\cL N_t(M);r+1/2}\leq C_r |A|_{r},\quad r>0, \\ |v|_{\cL N_t(M);s+1/2}\leq C_st^\all, \quad
|v|_{\cL N_t(M);r+1}\leq C_rt^{-1/2} |A|_{r}.
\end{gather*}

Thus, we have for $r>0$
\begin{align} \label{e_linear}
|w_k|_{D_k;r+1}& \leq C_r |v|_{\cL N_t(M); r+1} \leq C_rt^{-1/2}|A|_{r},
\\ \label{wkvv}
|w_k|_{D_k;r}&\leq C_{r,\ve} |v|_{\cL N_t(M); r+1} |v|_{\cL N_t(M); 1+\ve} \leq {C_{r,\ve}'} t^{-1/2} {|A|_{r} } t^\all.
\end{align}
Thus, for $r>0$ we obtain
\begin{align*}
  |A_k \pa v_k|_{D_k;r} &\leq C_{r,\ve} (|A|_{ r} |\pa v_k|_{{D_k},\ve} + |A|_{ \ve} |\pa v_k|_{ D_k;r} )
\\ &\leq C_{r,\ve}' (|A|_{ r} |A|_{ s} + |A|_{ s} |v|_{\cL N_t(M);r+1})
 \leq  2C_{r,\ve}''t^{-1/2}|A|_{ s} |A|_{ r}.
\end{align*}
Similarly, 
\begin{align*}
|A_k (\db v_k,\pd w_k,\db w_k)|_{D_k;r} &\leq 2C_{r,\ve}''t^{-1/2}|A|_{{s}} |A|_{r,\ve},\quad r>0.
\end{align*}
It follows that
\eq{vws}
 |A_k (\pd v_k,\db v_k,\pd w_k,\db w_k)|_{ D_k;s}\leq C_{s,\ve} t^{2\all-\yh}<t^{d\all}
\eeq
by \re{d2}.
We also have
\begin{gather}  \label{vwm'}
|A_k (\pd v_k,\db v_k,\pd w_k,\db w_k)|_{ D_k;r} \leq  C_{r,\ve}|A|_{r}, \quad r > 0.
\end{gather}

We now estimate $\db w_k$ on $ D_k$.
By \rp{Lem:flow}, we have for $ r>0$
\al{} | \db w_k |_{ D_k;r}&\leq C_{r,\ve} |\db v|_{\cL N_t(M);r+1}|v|_{\cL N_t(M);1+\ve}
+C_{r,\ve}|v|_{\cL N_t(M);1+\ve}|v|_{\cL N_t(M);r+1}.\nonumber
 \end{align}

\medskip

Recall that $v=-S^tPA$ is well defined on the neighborhood  $\cL N_{t}(M)$ of $M$.

To apply estimate on time-one mappings stated in \rp{Lem:flow}, we need to estimate $\db v$ on $\cL N_{t}(M)$.
Applying the homotopy formula for the third time, we can express on $\cL N_{t}(M)$,
\aln
\db v&=  [S^t,\db] PA {-} S^t\db PA =[S^t, \db] PA - S^t A + S^tQ\db A.
\end{align*}
Recall \re{dbS} says that
$$
 \left| [\db,S^t] v\right|_{\cL N_t({M}),a} \leq  C_{a,b} t^{b-a} |v|_{M,b}, \quad 0 < a \leq b < a+L.
$$
 Thus
\aln
|\db v|_{\cL N_{t}(M); r+1}&\leq C_r (t^{-1/2}|PA|_{r+1/2}+t^{-1}| A|_{r}+ t^{-3/2}|Q\db A|_{r-1/2})\\ &\leq
C_{r,\ve }(2t^{-1}|A|_{r}+2t^{-3/2}|A|_{r} |A|_{\ve}), \quad r>1.
\nonumber
\end{align*} 
where we again use the integrability condition to get $|Q \db A|_{r-\yh} \leq{C_{r,\ve}}
 |A|_{\ve} |A|_r$ by \re{QdbA}. 
Since $|A|_\ve \leq t^\all$ for $\all>1$, we get $|\db v|_{\cL N_{t}(M); r+1}\leq C_{r,\ve}t^{-1}|A|_{r}$ for $r>0$ and hence
\al{}
|\db v|_{\cL N_{t}(M); r+1}|v|_{\cL N_{t}(M);1+\ve}&\leq 3C_{r,\ve}' t^{\all-1}|A|_{r},\quad r>1;\nonumber\\
|v|_{\cL N_{t}(M);1+\ve}|v|_{\cL N_{t}(M);r+1}&\leq C_{r,\ve}''t^{\all-1/2}|A|_{r}, \quad r>1.\nonumber
\end{align}
Thus
\al{} | \db w_k |_{D_k;r}\leq C_{r,\ve}  t^{\all-1}|A|_{r},\quad r>1.\nonumber
 \end{align}
Using the initial condition \re{t-a}-\re{Lmte}, we get
\begin{align*}
 |\db w_k|_{D_k;s} &\leq C_{s,\ve} t^{2\all -1}, \quad \all>1;\quad
 |\db w_k|_{D_k;r} \leq C_{r,\ve} t^{\all-1}|A|_{r}, \quad r>1.
\end{align*}
Combining them with \re{vws}-\re{vwm'}, we get for $I_5=\db w_k+A_k\pd w_k$,
\begin{align*} 
  |I_5|_{ D_k;s} &\leq C_{s,\ve} t^{2\all-1} <t^{d\all},
\\ 
  |I_5|_{ D_k;r} &\leq C_{r,\ve}|A|_{r}, \qquad r >1.
\end{align*}

Next, we estimate the low and high-order norms of $ (I + \Kc_k)^{-1}$. Recall that
$$\Kc_k= \ov{\pa f_k} + A_k \dbar f_k=\overline{\pd v_k}+\overline{\pd w_k}+A_k\db v_k+A_k\db w_k.$$
Thus
$$
  |\Kc_k|_{ D_k;r} \leq C_{r,\ve}(|v_k|_{ D_k;r+1} +|w_k|_{ D_k;r+1} +  |A_k\dbar v_k|_{ D_k;r}
  + |A_k \db w_k|_{ D_k;r}).
$$
By  \re{e_linear}, we get
\gan{}
|w_k|_{D_k;s+1}\leq C_{s,\ve}|v|_{s+1}\leq C'_{s,\ve}t^{-1/2+\all}, \quad |w_k|_{D_k;r+1}\leq C_{r,\ve}|v|_{  r+1}\leq C'_{r,\ve}t^{-1/2}|A|_r.
\end{gather*}
By \re{vws}-\re{vwm'}, we   get better estimates for the last two terms in $\cL K_k$.
Therefore,
\eq{Kks}
|\Kc_k|_{ D_k;s} \leq C_{s,\ve}t^{-1/2+\all}<1/2,
\eeq
provided
\eq{Cst}
C_{s,\ve}t^{-1/2+\all}<1/2.
\eeq
With this, we also have
\eq{Kkm}  |\Kc_k|_{ D_k; r} \leq C_{r,\ve} t^{-1/2}|A|_{r},\quad r>1.
\eeq

We now estimate $(I+\Kc_k)^{-1}$. Using the formula $(I+ \Kc_k)^{-1} = [\det (I+\Kc_k)]^{-1} B$, where $B$ is the adjugate matrix of $I + \Kc_k$, we see that every entry in $(I + \Kc_k)^{-1}$ is a polynomial in $[\det (I+\Kc_k)]^{-1}$ and entries of $\Kc_k$. By using the product and quotient rules \re{zprule}-\re{zqrule},  and \re{Kks} and \re{Kkm}, we get
$$
  |(I+\Kc_k)^{-1} |_{ D_k;s}<C_{s,\ve}, \quad
  |(I+\Kc_k)^{-1} |_{ D_k;r} \leq C_{r,\ve}t^{-1/2} |A|_{r}.
$$
Then
\aln{}
|(I+\Kc_k)^{-1} I_j|_{ D_k;s}&\leq C_{s,\ve}|I_j|_{ D_k;s}\leq C'_{s,\ve}t^{d\all},\\
  |(I+\Kc_k)^{-1} I_j|_{ D_k;r}&\leq C_{r,\ve}( |(I+\Kc_k)|_{r}|I_j|_{ D_k;s}+|(I+\Kc_k)|_{ D_k;s}|I_j|_{ D_k;r})\leq  C_{r,\ve}|A|_{r}.
\end{align*}
Thus $ \wti A_k = (I + \Kc_k)^{-1} ( \sum_{j=1}^5 I_j)$ satifies
\begin{equation} \label{ti_A}
 | \wti A|_{ s} \leq  C_{s,\ve}  t^{d\all}, \quad  | \wti A|_{r} \leq C_{r,\ve} |A|_{r},\quad r>1
\end{equation}
provided \re{d1} holds.

Finally, we estimate the norms of $A' = \wti A \circ G$, where $G=F^{-1}$. Let $G_k=I+g_k$.  Then $g_k=-(v)_k+\tilde w_k$ with $\tilde w_k(\cdot)=w_k(\cdot,-1)$. Set
$$
D_k^1:=\psi_k(U_k\cap M_1).
$$
By \rp{Lem:flow}, we have
\ga{}
|f_k|_{D_k;r+1/2}+|g_k|_{D_k^1;r+1/2}\leq C_{r,\ve} |A|_{r}.\nonumber
\end{gather}
Thus
\al{}\label{fkdk}
&|f_k|_{D_k;s+1/2}+|g_k|_{D_k^1;s+1/2}\leq  C_{r,\ve} t^\all, \\
& |f_k|_{D_k;m+1/2}+|g_k|_{D_k^1;m+1/2}\leq  C_{m,\ve}|A|_{m}.
\label{fkdk+}
\end{align}
With \re{d1}, we already require $m>2$. Thus, via interpolation,
$$
|f_k|_{D_k;2}+|g_k|_{D_k^1;2}\leq C_{s,m,\ve}t^{\f{m-3/2}{m-s}\all}|A|^{\f{3/2-s}{m-s}}_m.
$$
This will allow us to control the $C^2$ norm of $\rho\circ G$ of the new domain $M'=F(M)$.

By estimate \re{zchain} for composition, we have
\begin{align*}
|A'|_{M^1,r} &= |\wti A \circ G|_{M^1,r}
\leq C_{r,\ve} (|\wti A|_{ r} (1+|G|_{M^1, 1+\ve}^{2} ) + \| \wti A \|_{ 1+\ve} |G|_{M^1,r} + \| \wti A \|_{ \ve} ).
\end{align*}
Using estimates \re{ti_A} for $\wti A$ and \re{fkdk}-\re{fkdk+} for $g_k$, we obtain
\begin{align}  \label{A_s_est}
|A'|_{M^1, s} &\leq C_{s,\ve}  t^{d\all}, \\
\label{A's}
|A'|_{M^1,m} &\leq C_{m.\ve} |A|_{m}.
\end{align}

We now determine the parameters. 
 For \re{d1}, \re{d2}, \re{Cst} to hold,  we need
\ga{}\label{last-c}
d\all<m-s-2\ve,\quad
d\all<2\all- \yh -\ve,\quad m>2+2\ve,\\
 C_{s,m,\ve}L_{m} t^{\ve} \leq 1, \quad C_st^{-1/2+\all}<1/2.
  \label{last-c+}
\end{gather}
  So we take $d=1+\ve$, $\all=1+\ve$,  $s=1 +\ve$. Fix a small $\ve>0$.
Thus, it is easy to check that  conditions \re{last-c}-\re{last-c+} are satisfied by
$$
m=2+C\ve, \quad 0<t\leq t_0^*,
$$
where $t_0^*$ depends on $C_{s,m}, C_s$ which in turn depends on $M$. Further, $t_0^*$, $C_{s,m}, C_s$ are stable under small $C^2$ perturbations of $ b M$.
We can write \re{A_s_est}   as
$$
|A'|_{ s} \leq  t^{d\all},
$$
 which is possible provided $t_0^*$ is sufficiently small.   Thus, we can rewrite  \re{A's}   as
$$
|A'|_{M^1,m} \leq C_mL_m,\quad m>1,
$$
 which is possible as $m,s,\all,d,\ve$ are fixed now while $t_0^*>0$ can still be chosen  sufficiently small.

\subsection{A preview for the sequence of embeddings}

We now outline the iteration process in the next section. Set $$t_{j}=t_{j-1}^{d}=t_0^{{d}^j}, \quad  M^0=M, \quad   A^0=A\in \Lambda_{(0,1)}^m(M,\Theta).
$$
  By shrinking $t_0^*$ further, we will find a sequence of  strongly pseudoconvex domains $M^j$ in $\cL M$, a sequence of complex structures defined by vector fields $X^j:=\{X^j_1,\dots, X^j_n\}$ with an adapted frame determined by $A^j\in \Lambda_{(0,1)}^m(M^j,\Theta)$, and a sequence of smooth mappings $F^j$ such that
$M^{j+1}=F^j(M^j)$, $F^j$ transforms $X^j$ into $X^{j+1}$, and
\al{}\label{Ajs}
|A^j|_{M^j,s}&\leq t_j^s,\\
\label{Ajm}
|A^j|_{M^j,m}&\leq L_{j,m},\\
t_j^\ve C_{s,m,\ve}L_{j,m}&\leq 1,
\label{Ajmm}
\end{align}
provided \re{Ajs}-\re{Ajmm} hold for $j=0$.
Further, $F^j_k(z_k)=z_k+f_k^j(z_k)$ and $G^j_k(z_k)=z_k+g_k^j(z_k)$ with $G_k^j=((F^j)^{-1})_k$ satisfy
\al{}\label{fkdkj}
|f^j_k|_{D^j_k;s+1/2}+|g^j_k|_{D_k^{j+1};s+1/2}&\leq  C_s t_j^\all, \\ |f^{j}_k|_{D^{j}_k;m+1/2}+|g^j_k|_{D_k^{j+1};m+1/2}&\leq  C_{m}|A^j|_{m},
\label{fkdk+j}\\
|A^{j+1}|_{M^{j+1},m}&\leq C_{s,m}|A^{j}|_{M^{j},m},
\label{Ajm'}
\end{align}
where $D^j_k=\psi_k(M^j\cap U_k)$. Note that \re{Ajm'} shows that
$$
|f^{j}_k|_{D^{j}_k;m+1/2}+|g^j_k|_{D_k^{j+1};m+1/2}\leq C_{m} |A^j|_{M^j,m} \leq C_{m,s}'^{j+1}|A^0|_{M,m},
$$
 which grows slowly, whereas $$|f_k^j|_{D_k^j,s+1/2}+|g_k^j|_{D^{j+1}_k,s+1/2}\leq C_st^{\all d^{j}}$$ decays rapidly.
 Then the compositions $F^j\circ\cdots\circ F^0$ converge in $\Lambda^{m'}(M)$ norm for any $m'<m+1/2$.

\subsection{General case:  $b M\in \Lambda^{s_*}\cap C^2$ with $s_*\geq2$}
 When we deal with a domain $M$ whose boundary is less smooth, we first smooth defining function of the domain to get a   smooth and larger domain.  We then extend   the formally integrable complex structure $A$ to the larger domain and smooth the structure. Of course, both the extension and smoothing destroy the formal integrability of the structure. Nevertheless, the failure of formal integrability can be measured by a commutator of $[S^t,\db]$ that has a good estimate on the larger domain $\cL N_t(M)$.  This allows us to use the homotopy formula on the larger domain and construct smoothing diffeomorphism defined on a (larger) domain that contains $M$. Its restriction on the original $M$ will be used to transform the structure to new one that is still formally integrable on the new domain. Repeating this via a Nash-Moser iteration scheme, the composed diffeomorphism will transform the structure on $M$ to the standard complex structure on $\cL M$.

We apply  smoothing to the defining function $\rho$ of $M$ by setting
\eq{rhoHatt}
\rho_{\hat t}=\rho \ast\chi_{\hat t}, \quad \hat t=(c_*t)^{\f{1}{{ s_* }}}, \quad  { M_{\hat t}}^{t}=\{\rho _{\hat t}<  {t}\}.
\eeq
Here using $t^{\f{1}{s_*}}$ instead of $t$ reflects that $\rho\in \Lambda^{s_*}$. See Yie~\ci{MR2693230} and \ci{aq} for similar uses.
We now have
\ga
\|\rho_{\hat t}-\rho\|_{0}\leq C_{ 2 } \hat t^{s^* }|\rho|_{{s^* }},\label{rho1-1-25} \\
\|\rho_{\hat t}-\rho\|_{2}\leq \ve,\label{rho1-2c-25} \quad \hat t<\hat t^*(\nabla^2\rho),\\
  |\rho_{\hat t}|_b\leq C_{b,{ 2 }}\hat t^{s_*-b}|\rho|_{ s_*}, \quad b\geq  s_*.
\label{rho1-3-25}
\end{gather}
We also have $\hat t^{s_*}=c_*t$. When $c_*$ is sufficiently small, we have $C_{s_*}\hat t^{s_*}\|\rho\|_{s_*}<t/C$ by \re{rho1-1-25}. Thus
$$
M\subset \cL N_{t/{  {C_*}}}(M)\subset  { M_{\hat t}}^t\subset \cL N_{  {C_*}
t}(M).
$$
We emphasize that $\hat t^*(\pd^2\rho)$ in \re{rho1-2c-25}
depends on the {\it modulus of continuity} of $\pd^2\rho$.
 {When $s_*=2$, $|\rho|_{s_*}$ needs be replaced by $\|\rho\|_2$. Then  we can replace all norms in \re{rho1-1-25}-\re{rho1-3-25}, which are stronger inequalities.}

To simplify formulas, we set
$$\mu(a)=\f{(a-s_*)_+}{s_*}.
$$
Then \re{rho1-3-25} becomes $\|\rho_{\hat t}\|_b\leq C_{b}t^{-\mu(b)}\|\rho\|_{s_*}$, where $C_b$ depends on $s_*$ also.

To limit the loss of derivatives, we will use $ { M_{\hat t}}^t\subset \cL N_{ {C_*}
t}(M)$, which allows us to apply the estimates of $S^tv$ and $[\db, S^t]v$ on $\cL N_{C_*t}(M)$, as well as the homotopy formula on the smooth subdomain $ { M_{\hat t}}^t$ of $ \cL N_{C_*t}(M)$. Note that the original $S^tv$ is defined on $\cL N_{t}(M)$. Replacing $S^tv$ by $S^{t/C}v$ which is still denoted it by $S^tv$,  we may assume that $S^tv$ and its estimates in \rl{Lem:StD} hold on $\cL N_{C_*t}(M)$ for any given $C_*$.
 We will also have
$$
\|v\|_{C^0(\cL N_t(M))}<t/C.
$$
Since $ { M_{\hat t}}^t$ contains  $\Nc_{t/ {C_*}}(M)$,  the time-one map $F$ of $v$ is well-defined on the original domain $M$; consequently, we can define a new almost complex structure $A$ on $M'=F(M)$ that is still formally integrable.

We have  homotopy operators satisfying
\aln
|P_{ { M_{\hat t}} ^{t }}\var|_{ { M_{\hat t}} ^{t }, r}&\leq C_{r,\ve}(\nabla\rho_{\hat t },\nabla^2\rho_{\hat t}) (|\var|_{ { M_{\hat t}} ^{t },r-1/2}+|\rho_{\hat t }|_{r+2}|\rho_{\hat t}|_{2+\ve}^b|\var|_{ { M_{\hat t}} ^{t },\ve})\\
&\leq C_{r,\ve}(\nabla\rho_{\hat t },\nabla^2\rho_{\hat t}) (|\var|_{ { M_{\hat t}} ^{t },r-1/2}+t ^{-\mu (r+2)-b\ve}L_0|\var|_{ { M_{\hat t}} ^{t },\ve}), \quad r>1/2.
\end{align*}
The same estimate holds when $P_{ { M_{\hat t}}^t}$ is replaced by $Q_{ { M_{\hat t}}^t}$.

For the rest of this section, we denote $P_{ { M_{\hat t}}^t}, Q_{ { M_{\hat t}}^t}$ by $P,Q$. 

 Recall that $S^tA$ is smooth on $\cL N_{t/{C_*}}(M)$ and it depends only on the values of $A$ on $M$. 
  {To ease notation,  replace $S^t$ by $S^{C_*t}$, we may assume that $S^tA$ is defined on $\cL N_t(M)$.}
 Thus, we
set    $v =-PS^t A $, which is well-defined on $\cL N_t(M)$ without further extension.  
Notice that unlike the strongly pseudoconvex case, $S^t$ is applied to $A$ directly.
Let $F$ be the time-one map of $v $ on $\cL N_t(M)$. We have $f_k=v_k+w_k$.

In the following, we will denote by $\pd $ a partial derivative of order $1$. Then on $D_k$, we have
\begin{align*}
&A_k+\dbar f_k + A_k \d f_k
=A_k- \dbar(P  S^tA)_k - A_k \pa (P S^tA)_k + A_k \pa w_k +\db w_k\\
&\quad=((I-S^t)A)_k +(Q S^t\db A)_k - A_k \pa (P S^tA)_k+ (A_k \pa w_k + \db w_k)
+(Q[\db, S^t]A)_k\\
&\quad=I_1 +I_2-I_3+I_4+I_5
\end{align*}
with $A\equiv\cL E_MA$ and
$$
\begin{array}{lll}
  I_1 = ((I - S^t) A)_k, &  I_2 = (Q S^t\db A)_k , &
  I_3 =A_k \pa (P S^tA)_k,
 \vspace{.75ex} \\
I_4 = A_k \pa w_k + \db w_k,&
 I_5= (Q[\db, S^t]A)_k, & \Kc_k= \ov{\pa f_k} + A_k \dbar f_k.
\end{array}
$$

 We assume \ga\label{igs3}
|A|_s<t^{\all},  \quad \all>1, \  s>1;\\ |A|_m  \leq L_{m,\ve},  \quad  m> 3/2; \label{igm3}\\
\|\rho\|_{s_*}\leq R_{s_*}, \quad R_{s_*}\leq L_{m,\ve}, \quad m+1/2>s_*.
\end{gather}
(Here $\ve$ is determined via \re{cri-1}-\re{am52s}.)
Then by \re{igm3},   we have
$$
|I_1|_{ D_k;s}=|((I-S^t)A)_k |_{ D_k;s}\leq C_{s,m} t^{m-s}|A|_{m},\quad s\leq m<s+L.
$$
Assume that
$$
m-s>d\all+\ve, \quad t^\ve C_{s,m}L_m<1.
$$
Then we have
$$
|I_1|_{ D_k;s}\leq t^{d\all}, \quad |I_1|_{D_k,\ell}\leq C_\ell|A|_\ell, \quad\forall\ell>0.
$$
We will seek additional conditions to achieve
$$
|I_j|_{ D_k;s}\leq t^{d\all}, \quad |I_j|_{ D_k;m}\leq C_{m,\ve}L_m, 
\quad j=2,\dots, 5.
$$

Recall that
$
S^t\db EA=S^t\db A$ on $\cL N_t(M)$.
 Therefore, we obtain
\al\label{v1t}
|v |_{M^{t },r}&=|P_{M^{t }}S^{t } A |_{M^{t },r}\leq C_{r,\ve}' |\rho_t|_{2+\ve}^b |S^t   A |_{M^{t},\ve}|\rho_t|_{r+2}+C_{r,\ve}' |S^t   A |_{M^{t},r-1/2}\\
&\leq C_{r,\ve}t ^{-\mu(r+2)-b\ve}R_{s_*} |A |_\ve +C_{r,\ve} |S^tA |_{M^{t },r-1/2}, \quad r>1;\nonumber\\
|v |_{M^{t },\ve}&=|P_{M^{t }}S^{t } A |_{M^{t },\ve}\leq {C_\ve}( |\rho_t|_{2+\ve}^{b+1} |S^t   A |_{M^{t},\ve}+ |S^t   A |_{M^{t},\ve}) \leq C_\ve't^{\all-b\ve}.\label{vepsilon}
\end{align}
Note that $\mu(a+\ve)\leq \mu(a)+\ve$ and
\eq{}
\mu(3)\leq \f{1}{2}.
\label{mus<.5}
\eeq
 In particular,
$$
|w_k|_{ D_k;1+\ve}\leq {C_\ve}|v|_{1+\ve} \leq  2C_\ve'R_{s_*}t^{\all-1/2-b\ve}.
$$
 For $ r>1/2$, we have \begin{align*} 
   |I_2|_{ D_k;r} &= |Q S^t \dbar  A|_{D_k,r} \leq C''_{r,\ve}  \left(|\rho_1|_{r+2}|\rho|_{2+\ve}^b |S^t \dbar  A |_{M^{t},\ve}   +  |S^t \dbar  A|_{M^{t}, r-1/2}   \right).
\end{align*}
Since $s>1+\ve$, we have $ |S^t\db A |_{M^t,\ve}\leq {C_\ve}|A|_s^2$. Then
$$
|I_2|_{ D_k;s}<C_{s,\ve}t^{-\mu(s+2)-b\ve}R_{s_*}t^{2\all}+C_{s,\ve}t^{-1/2}t^{2\all}<t^{d\all},
$$
assuming
$$
2\all-\mu(s+2)>d\all+b\ve+\ve, \quad 2\all-\f{1}{2}>d\all, \quad C_{s,\ve}t^\ve<1.
$$
We also have a coarse estimate
 \eq{I2ell}
  |I_2|_{ D_k;\ell}\leq C_{\ell,\ve}(t^{2\all-\mu(\ell+2)-b\ve}R_{s_*}+|A|_\ell), \quad \forall\ell>0.
\eeq
In particular,
$$
  |I_2|_{ D_k;m}\leq C_{m,\ve}(t^{-\mu(m+2)-b\ve}t^{2\all}R_{s_*}+ |A|_m)\leq C_mL_m,
$$
assuming
$$
2\all-\mu(m+2)>b\ve,\quad   \quad C_mt^\ve\leq 1/2.
$$

We have  
  \begin{align*} 
  |I_3|_{ D_k;r} &= |A_k \pa (P_{M^t}S^tA_k)|_{{D_k},r}
 \leq C_{r,\ve}(  |A|_r|PS^tA|_{D_k,1+\ve}+|A|_\ve|PS^tA|_{D_k,r+1})
\\ &\leq C_{r,\ve}'|A|_r|\rho_t|_{2+\ve}^b |\rho_{\hat t}|_{3+\ve}t^{\all}+C_{r,\ve}'t^{\all}( |\rho_{\hat t}|_{r+3}|\rho_t|_{2+\ve}^bt^\all+ t^{-1/2}|A|_{r}). \nonumber
\end{align*}
Thus,   we get
$$
|I_3|_{ D_k;s}<3C_{s,\ve} (t^{2\all-\mu(s+3)-2b\ve}R_{s_*}+t^{2\all-1/2}L_m)<t^{d\all},
$$
assuming
$$
2\all-\mu(s+3)>d\all+2b\ve+\ve, \quad 2\all-1/2>d\all+\ve, \quad 3C_{s,\ve}R_{s_*}t^\ve<1.
$$
We have a coarse estimate
\eq{I3ell}
  |I_3|_\ell\leq C_{\ell,\ve}(2|A|_\ell+ t^{2\all-\mu(\ell+3)-2b\ve}R_{s_*}), \quad\forall \ell>0.
\eeq
In particular,
$$
  |I_3|_{m,\ve}\leq C_{m,\ve}(2|A|_m+ t^{2\all-\mu(m+3)-2\ve}R_{s_*})\leq 3C_{m,\ve}L_m,
$$
assuming
\eq{2am+1}
2\all>\mu(m+3)+2b\ve.
\eeq

We have by \re{v1t}  
\begin{align} \label{wvr+1}
|w_k|_{ D_k;r+1}
 \leq C_{r,\ve} |v|_{\cL N_t(M);r+1}
\leq C'_{r,\ve}(t^{ -\mu(r+3)-b\ve}R_{s_*}|A|_s+t^{-1/2}|A|_r)
\end{align} 
for $r>0$.  Thus
\begin{align*}
  |A_k \pa w_k|_{ D_k;r} &\leq C_{r,\ve} (|A|_{ r} \|\pa w_k\|_{ D_k;0} + \|A\|_{ 0} |\pa w_k|_{ D_k;r} )
\\ &\leq C_{r,s,\ve}'((3t^{\all-\mu(3)-b\ve}+2t^{\all-1/2  {-b\ve}})   |A|_r+ t^{\all-\mu(r+3)-b\ve}R_{s_*}|A|_s).
\end{align*} 
  It follows that
\aln{}
 |A_k \pa w_k|_{ D_k;s}&\leq C_{s,\ve} (t^{2\all-\mu(s+3)-b\ve}+t^{2\all-1/2-b\ve})R_{s_*},\\
  |A_k \pa w_k|_{ D_k;\ell} &\leq C_{\ell,\ve}  ( R_{s_*}|A|_\ell+ (t^{2\all-\mu(\ell+3)-b\ve}+t^{\all-1/2-b\ve})R_{s_*}).
\end{align*}
We also have a coarse estimate.
$$
|A _k\pa w_k|_{ D_k;\ell}  \leq C_{\ell,\ve}(|A|_{\ell}+t^{2\all-\mu(\ell+3)-b\ve}R_{s_*}).
$$
 In particular, we have
$$
 |A _k\pa w_k|_{ D_k;m}\leq C_{m,\ve}L_m.
$$

%

On $M^{t}$, we have $v=-PS^tA$ and hence by the homotopy formula
$$
\db v=-\db S^tA+Q\db S^tA.
$$
Note that $\db S^tA=  [ \db, S^t]A + S^t\db A$ implies that
$$
|\db S^tA|_{r+1}\leq C_{r,m}( 2t^{m-r-1}|A|_m+ 2t^{-2}|A|_s|A|_r),\quad r>1.
$$
Combining with estimate on $Q$, we obtain for $m\geq r\geq s>1$
\aln\label{}
|\db v|_{r+1}&\leq C_{r,m,\ve}( 2t^{m-r-1}|A|_m+ 2t^{-2}|A|_s|A|_r +|\rho_1|_{r+3}|\rho_t|_{2+\ve}^b|\db S^tA|_\ve)\\
&\leq C_{r,m,\ve}'(2t^{m-r-1}|A|_m+ 2t^{-2}|A|_s|A|_r+ t^{2\all}t^{-\mu(r+3)-b\ve}R_{s_*}).\nonumber
\end{align*}
We now estimate $\db w_k$ on $ D_k$. Recall that  \re{mus<.5} says $\mu(3)\leq 1/2$.
By \re{dbwr}, we have
\aln{}
| \db w_k |_{ D_k;r}&\leq C_{r,\ve} |\db v|_{\cL N_t(M);r+1}|v|_{\cL N_t(M);\ve}
+C_{r,\ve}|v|_{\cL N_t(M);1+\ve}|v|_{\cL N_t(M);r+1}\\
&\leq C_{r,\ve}t^{\all-b\ve} (t^{m-r-1}|A|_m+ t^{\all-2} |A|_r
+t^{2\all-\mu(r+3)-b\ve}R_{s_*})\\
&\quad +C_{r,\ve}t^{\all-1/2-b\ve}(t ^{-\mu(r+3)-b\ve} |A |_\ve +t^{-1/2}|A |_{r})\\
&\leq C_{r,\ve}'t^{\all-b\ve }(t^{m-r-1}|A|_m+t^{\all-2} |A|_r+t^{2\all-\mu(r+3)}R_{s_*}).
 \end{align*}
Recall that $I_4=\db w_k+A_k\pd w_k$. Thus
$$
|I_4|_{ D_k;s}\leq C_{s,\ve} t^{2\all-\mu(s+3)-b\ve}+C_{s,\ve}t^{3\all-2}+ C_{m,\ve} t^{\all+m-s-1-b\ve}|A|_m\leq t^{d\all},
$$
assuming
$$
2\all-\mu(s+3)>d\all+(b+1)\ve, \quad 3\all-2>d\all+\ve, \quad \all+m-s-1>d\all+(b+1)\ve.
$$
Hence, we have a coarse estimate
\eq{I4ell}
|I_4|_\ell\leq C_{\ell,\ve}(|A|_{\ell}+t^{2\all-\mu(\ell+3)-b\ve}R_{s_*}),\quad\forall\ell>1.
\eeq
Also by \re{2am+1},
we have $
  |I_4|_m\leq C_{m,\ve}L_m$.

We  have
\aln
|I_5|_{ D_k;r}&=|Q([\db, S^t]A)|_r\leq C_{r,\ve}(|\rho_1|_{r+5/2}|\rho_1|_{2+\ve}^b |[\db,S^t]A |_\ve+|[\db,S^t]A|_r),\quad r>0.
\end{align*}
Recall that $|[\db, S^t]A|_b\leq C_{a,b}t^{a-b}|A|_a$ for $a\geq b$.
We can estimate the last term like $I_1$. Also, $|[\db, S^t]A|_\ve\leq C_\ve t^{s-\ve}t^\all$.  Therefore,
$$
|I_5|_{s}\leq t^{d\all}, \quad |I_5|_{m}\leq C_{m,\ve}L_m,
$$
assuming
$$
s+\all -\mu(s+5/2)>d\all+b\ve+\ve,\quad s+\all -\mu(m+5/2)>b\ve+\ve.
$$
We also have a coarse estimate
\eq{I5ell}
|I_5|_{ D_k;\ell}\leq   C_{\ell,\ve}( t^{s+\all-\mu(\ell+5/2)-b\ve}R_{s_*} + |A|_\ell),\quad\forall\ell>0.
\eeq

Next, we estimate the low and high-order norms of $ (I + \Kc_k)^{-1}$. Recall that
$$\Kc_k= \ov{\pa f_k} + A_k \dbar f_k=\overline{\pd v_k}+\overline{\pd w_k}+A_k\db v_k+A_k\db w_k.$$
Thus
\aln
  |\Kc_k|_{ D_k;r}& \leq C_{r,\ve}(|v_k|_{ D_k;r+1} +|w_k|_{ D_k;r+1} +  |A_k\dbar v_k|_{ D_k;r}
  + |A_k \db w_k|_{ D_k;r})\\
  & \leq C_{r,\ve}'( 2|v_k|_{ D_k;r+1} +2|w_k|_{ D_k;r+1} +  |A_k|_r( |v |_{\ve}+|w_k|_{ D_k;1+\ve})).
\end{align*}
Recall that $|v|_{1+\ve}+|w_k|_{ D_k;1+\ve}<C_{\ve}R_{s_*}t^{\all-\mu(3)-b\ve}\leq C_{\ve}'R_{s_*}t^{\all-1/2-b\ve}$. By \re{wvr+1}, we have
$$
|w_k|_{ D_k;r+1}\leq C_{r,\ve}|v|_{M^{t_1},r+1}\leq C_{r,\ve}'(R_{s_*}t^{ -\mu(r+3)-b\ve}|A|_s+t^{-1/2}|A|_r).
$$
Therefore,
\aln{}
|\Kc_k|_{ D_k;s} &\leq C_{s,\ve}R_{s_*} (t^{-\mu(s+3)-b\ve}+t^{-1/2-b\ve})t^{\all},\\    |\Kc_k|_{ D_k; r} &\leq C_{r,\ve}(t^{-1/2-b\ve}|A|_r+R_{s_*}t^{\all-\mu(r+3)-b\ve}).
\end{align*}
Assume that
\eq{ar+1}
\all>\mu(s+3)+b\ve+\ve, \quad C_{s,\ve}t^{\ve}<1/2.
\eeq

We now estimate $(I+\Kc_k)^{-1}$. Using the formula $(I+ \Kc_k)^{-1} = [\det (I+\Kc_k)]^{-1} B$, where $B$ is the adjugate matrix of $I + \Kc_k$, we see that every entry in $(I + \Kc_k)^{-1}$ is a polynomial in $[\det (I+\Kc_k)]^{-1}$ and entries of $\Kc_k$. By using the product rule \re{zprule} and \re{ar+1}, we get
\gan{}
  |(I+\Kc_k)^{-1} |_{ D_k;s}<C_{s,\ve}, \\
  |(I+\Kc_k)^{-1} |_{ D_k;\ell} \leq C_{\ell,\ve}(t^{-1/2-b\ve}|A|_\ell +R_{s_*}t^{\all-\mu(\ell+3)-b\ve}).
\end{gather*}
Then
\aln{}
|(I+\Kc_k)^{-1} I_j|_{ D_k;s}&\leq C_{s,\ve}|I_j|_s\leq t^{d\all},\\
  |(I+\Kc_k)^{-1} I_j|_{ D_k;r}&\leq C_{r}|(I+\Kc_k)^{-1}|_s|I_j|_{r}+C_{r} |(I+\Kc_k)^{-1}|_r|I_j|_{s}\\
  &\leq C'_{r,\ve} (|I_j|_{r}+R_{s_*}t_j^{d\all+\all-\mu(r+3)-b\ve}+R_{s_*}t^{d\all-1/2-b\ve}|A|_r).
\end{align*}
In particular, we have
$$
  |(I+\Kc_k)^{-1} I_j|_{ D_k;m}\leq C_{m,\ve}L_m
$$
since $d\all>1/2$ and
$$
d\all+\all -\mu(m+3)>b\ve.
$$
Therefore, $\tilde A_k=(I+\cL K_k)^{-1}\sum I_j$ satisfies
\begin{gather} \label{ti_A3}
 | \wti A|_{ s} \leq  C_{s,\ve}  t^{d\all}, \quad  | \wti A|_{ m} \leq C_{m,\ve} L_m,\\
 | \wti A|_{ \ell} \leq C_{\ell,\ve}(|A|_{\ell}+(t^{2\all-\mu(\ell+3)-b\ve}+ t^{s+\all-\mu(\ell+5/2)-b\ve})R_{s_*}),\quad\forall\ell>0.
 \label{ti_A3ell}
\end{gather}
where the last estimate is obtained from
 \re{I2ell}, \re{I3ell}, \re{I4ell}, \re{I5ell}.

Finally, we estimate the norms of $A' = \wti A \circ G$, where $G=F^{-1}$. Let $G_k=\id+g_k$.  Then $g_k=-(v)_k+\tilde w_k$.
By \rp{Lem:flow}, we have
\gan{}
|w_k|_{ D_k;r+1/2}\leq C_{r,\ve}|v|_{r+1/2}\leq   C_{r,\ve}'(R_{s_*} t^{ -\mu(r+5/2)-b\ve}|A|_s+|A|_r),\\
|g_k|_{ D_k^1
;r+1/2}\leq C_{r,\ve}|f_k|_{ D_k;r+1/2}\leq C'_{r,\ve}|v|_{r+1/2}.
\end{gather*}
 {Recall that we assume $R_{s_*}\leq L_m$.}
Thus
\al{}\label{gks1/2}
|g_k|_{ D^1_k;s+1/2}&\leq C_{s,\ve}|f_k|_{ D_k;s+1/2}\leq  C'_{s,\ve} R_{s_*} t^{\all-\mu(s+5/2)-b\ve}, \\
|g_k|_{ D^1_k;m+1/2}&\leq C_{m,\ve}|f_k|_{ D_k;m+1/2}\leq  C'_{m,\ve} {L_m},
\label{gks1/2+}
\end{align}
assuming
\eq{allm12}
\all-\mu(m+5/2)>b\ve+\ve.
\eeq
We also have a coarse estimate
$$
|g_k|_{ D^1_k;\ell+1/2}\leq C_{\ell,\ve}|f_k|_{ D_k;\ell+1/2}\leq  C'_{\ell,\ve}(|A|_\ell+ R_{s_*} t^{\all-\mu(\ell+5/2)-b\ve}).
$$

Then we obtain
\begin{align}\label{A'}
|A'|_{M^1,r} &= |\wti A \circ G|_{M,r}
\leq C_{r,\ve} (|\wti A|_{ r} (1+|G|_{M, 1+\ve}^{2} ) +  | \wti A  |_{ 1+\ve} |G|_{M,r} +  | \wti A  |_{ \ve} ).
\end{align}
Using estimates \re{ti_A3}  for $\wti A$, we obtain
\begin{align*}  
|A'|_{M^1; s} &\leq C_{s,\ve}  t^{d\all},
\quad
|A'|_{M^1;m} \leq C_{m,\ve} {L_m,}
\\
|A'|_{M^1;\ell } &\leq C_{\ell,\ve}(|A|_{M;\ell}+ (t^{2\all-\mu(\ell+3)-b\ve}+ t^{s+\all-\mu(\ell+5/2)-b\ve})R_{s_*}).
\end{align*}
Here we have used the fact that $\mu(\ell+3)\leq \mu(\ell+5/2)+\f{1}{2s_*}\leq  \mu(\ell+5/2)+\f{1}{2}$.
Our initial conditions need to control $\|\rho\circ G\|_{2}$, which can be achieved through
 \re{gks1/2}-\re{gks1/2+}. We also by  {$m+1/2>s_*$}
 and \re{gks1/2+}
$$
|\rho^1|_{s_*}\leq C_{s_*,\ve}(|\rho|_{s_*}|G|^2_{1+\ve}+\|\rho\|_2|G|_{s_*})\leq C'_{s_*,\ve}L_m.
$$
 We have obtained
\eq{Lm'}
\max\{|A'|_{m},|\rho^1|_{s_*}\}\leq C_{m,\ve}L_m.
\eeq

We need
$$
m>s+d\all, \quad s_*\geq 2,\quad m+ {\f{1}{2}}
>s_*>s.
$$
Thus, take $s=1+\ve$ and $d=1+\ve$, where $\ve>0$ is  sufficiently small. We want to achieve
\ga{}  \label{cri-1}
s+\all-\f{(s+5/2-s_*)_+}{s_*}>d\all+2b\ve+\ve,\\
\label{am52s}\all-\f{(m+5/2-s_*)_+}{s_*}>\ve+2b\ve.
\end{gather}
Since $s=1+\ve$, $\all>1$ and $s_*\geq2$, we always have \re{cri-1}.
Hence, we must have
$$
m-1>\all>\f{m+\f{5}{2}-s_*}{s_*},\quad \all>1.
$$
Assume that $m>2$ and
$$
m>\f{m+5/2}{s_*}, \quad i.e. \quad m>\f{5/2}{s_*-1}.
$$
Therefore, for the existence of $\all,\ve$ it is equivalent to have
\eq{ms*}
m>\max\Bigl\{s_*-\f{1}{2}, \f{5/2}{s_*-1}, 2\Bigr\}, \quad\text{i.e.}\quad
m>\max\Bigl\{\f{5/2}{s_*-1}, 2\Bigr\}.
\eeq
For instance, when $s_*=2$, we need $m>5/2$. When $s_*\geq9/4$, we just use $s_*=9/4$ and require $m>2$.

  To prove the theorem, we need to combine the Nash--Moser iteration and an application of  Hartogs's extension  when concavity is present.
  
 We will use the estimates in Section~\ref{ghf} to complete the proof in the next section.
To simplify notations, we will not indicate the dependence of various constants on $\ve$ in Section~\ref{ghf} since the $\ve$ has been fixed now.
\section{The convergence of iteration and Hartogs type extension}\label{sect:hartogs}
\setcounter{thm}{0}\setcounter{equation}{0}

In this section we set up the iteration scheme. We apply an infinite sequence of diffeomorphisms $F^j$ as constructed in the previous section. The goal is to show that the maps $F^j \circ F^{j-1} \circ \cdots\circ F^0$ converge to a limiting diffeomorphism $F$, while the error $A^j$ converges to $0$.

For this scheme to work we need to ensure that for each $i=0,1,2,\dots$, the  $F^i$ maps $M^i$ to a new domain $M^{i+1} = F^i(M^i)$ which is an $a_q$ domain with $C^2$ boundary. Hence, we need to control the $C^2$-norms of   $F^j$ and a defining function $\rho^j$ of $M^j$.

In what follows we assume that
\[
  M^0= \{ z \in \cL M\colon\rho^0(z) < 0 \}\Subset\cL U
\]
with $ b M^0=\{z\in\cL M\colon \rho^0(z)=0\}$, where $\rho^0$ is a $C^2$ function on $\cL M$ such that $d\rho^0(z)\neq0$ for each $z\in b M^0$ and $\rho^0-1$ has compact support in $\cL M$.

\begin{lemma}[\cite{GG}*{Lem.~7.1}] \label{map_iter} Let $M^0,\rho^0$ be as above.
For $j=0,1,\dots$, let
$F^j = I + f^j$ be a $C^m$ diffeomorphism which maps $\cL M$ onto $\cL M$ and maps $M^j$ onto $M^{j+1}$, where $M^{j+1}=F^j(M^j)$. 
Let 
$\rho^{j+1}=\rho^j\circ G^j$ with $G^j= (F^j)^{-1}$.
For any $\ve>0$, there exists
$\delta=\delta(\rho^0,\ve)>0$
such that if
$$
\|F^j-\id \|_{{\cL M},2}\leq \frac{\delta}{(j+1)^2},\quad \forall j\in\nn,
$$
then
 $\tilde F^j=F^j\circ\cdots\circ F^0$, $\tilde G^j=(\tilde F^j)^{-1}$ and $\rho^{j+1}$ satisfy
\gan{} 
\| \tilde F^{j+1}-\tilde F^j \|_{{\cL M},2}\leq C\frac{\delta}{(j+1)^2},\quad \forall j\in\nn,\\
\| \tilde G^{j+1}-\tilde G^j \|_{{\cL M},2}\leq C'\frac{\delta}{(j+1)^2},\quad\forall j\in\nn, \\
\| \rho^{j+1} - \rho^0 \|_{{\cL M},2} \leq \ve.
\end{gather*}
\end{lemma}

We now prove our main theorem.
\begin{thm} \label{Global NN mv}
Let
$M\colon\rho<0$ be a relatively compact domain in $\cL M$, where $\rho$ is a  $C^2$ defining function.
Assume that $H^{1}(M, \Theta)=0$. Assume that   $\del>0$, $\ve_0>0$ and
\eq{r0cond}
\ve_1:=r_0-\max\Bigl\{\f{5/2}{s_*-1}, 2\Bigr\}>0, \quad 2\leq s_*\leq \f{9}{4}, \quad 2<r_0<3.
\eeq
 There exist $\hat t_0>0$ and $   \beta>1$   such that   if $A\in\Lambda^r_{(0,1)}(M,\Theta)$ defines  a formally integrable almost complex structure on $\ov M$ and 
 \ga{}\label{sr0}
t_0^{-\beta}|A|_{1+\ve_0}<1, \quad  t_0^{\e_1}\max\Bigl\{|A|_{r_0}, \|\rho\|_2, \ |\rho|_{s_*}\Bigr\} <1,
\quad \text{ for some $t_0\in (0,\hat t_0]$},
\end{gather}
 there exists a  diffeomorphism $F\in\Lambda^{\ell}(M)$ for all $\ell<r_0+1/2$ from $M$ onto $M'\subset\cL M$ such that  $\|F-\id\|_{C^2(\ov M)}<\del$, and $F$ transforms the complex structure defined by $A$ into the given complex structure of $\cL M$, provided
\bpp\item    $M$ is  strongly pseudoconvex, $s_*=2$ and  $bM\in C^2$; or 
 \item  $M$ satisfies $a_1^+$ condition and $\rho\in C^2\cap\Lambda^{s_*}$.  
\epp
 Further,    $\hat t_0=\hat t_0(\rho,\ve_0,\ve_1,\del),\beta=\beta(\rho,\ve_0,\ve_1,\del)$ are stable under   small $C^2$ perturbations of $\rho$. Assume further that $r>r_0$ and $A\in\Lambda^r(M)$. Then the constructed $F$ is in $\Lambda^{\ell}(M)$ for all $\ell<r+1/2$.
\end{thm}
\begin{rem}\bpp
\item The first two cases cover the case where $ b M$ and $A$ are $C^\infty$.  The result was proved by Hamilton~\cites{MR0477158, MR594711} who did not specify the  value of $r_0$. \item When $\cL M$ is  $\cc^n$ and $M$ is strongly pseudoconvex, Shi~\ci{shiNN} can low $r_0$ to the range $(3/2,\infty)$, improving an early result of Gan-Gong~\ci{GG}. We need  $r_0>2$ even for the strongly pseudoconvex case  since we have to use flows to construct the diffeomorphism $F$.
\item \re{sr0} can be restated as
\eq{newsr0}
|A|_{1+\e_0}\leq\min\Bigl\{\hat t_0^\beta, 
\Bigl[\min\bigl\{|A|_{r_0}, \|\rho\|_2, \ |\rho|_{s_*}\bigr\}\Bigr]^{-\beta/{\e_1}}\Bigr\}.
\eeq
    \epp
\end{rem}
\begin{proof}[Proof of \rta{Global NN mv}]
 Assume that $H^{1}(M,\Theta)=0$.  Let $A\in\Lambda^r_{(0,1)}(M,\Theta)$ define  a formally integrable almost complex structure on $\ov M$, where $2<r\leq\infty$. We need to find    a  diffeomorphism $F$ from $M$ onto $M'\subset\cL M$, which transforms the complex structure defined by $A$ into the given complex structure of $\cL M$.  To derive the estimates, let us assume that    $r<\infty$.

\subsection{Convergence for the  strongly pseudoconvex case.}
%
 Fix any finite $r_0>2$. Let us recall how $\ve$ is choosen. Set
$
\all=1+\ve,  d=1+\ve$ and $ s=1+\ve.
$
When $\ve>0$ is sufficiently small, we have
$$
 d\all<r_0-s-2\ve,\quad d\all<2\all-1-\ve, \quad r_0>2+C\ve.
$$
After $\ve$ is chosen. We then choose $t_0$ such that
\eq{Ini-1}
0<t_0<1, \quad
C_{s,r_0}t_0^\ve<1, \quad C_{s}t_0^{-1/2+\all}<1/2.
\eeq
We now define
$
t_{j+1}=t_j^d.
$

We are given a formally integrable almost complex structure $A^0\in C^1(\ov{M^0},\Theta)$. Assume that
\eq{Ini-2}
|A^0|_{M^0,s}\leq t_0^\all, \quad |A^0|_{M^0,r_0}\leq L_{0,r_0}.
\eeq
We now apply the homotopy formula
$$
A^0=\db H_0A^0+H_1\db A^0, \quad \text{on $M_0$}.
$$
This gives a vector field $v^0=-H_0A^0\in \La^{s+1/2}(M^0)$. We now apply Moser's smoothing operator $S^{t_0}$ and define
$$
\tilde v^0=\chi E_{\cL N_{t_0}(M^0)} S^{t_0}v^0,
$$
where $\chi=1$ on $\cL N_{t_0}(M^0)$ and $\chi$ is compact support on $\cL M$.

 The extension and cut-off allow us to define a time-one map $F^0$ on the whole space $\cL M$ that is the identity map away from the support of $\tilde v^0$.
Since $\tilde v^0=S^{t_0}v^0$ on $\cL N_{t_0}(M_0)$, then $F^0|_{M^0}$ agrees with the time-one map of the flow of $S^{t_0}v^0=-S^tP_0A^0$ on $\cL N_{t_0/C}(M^0)$. Therefore, $F^0$ transforms the almost complex structure defined by $A^0$ into a almost complex structure defined by $A^1\in \Lambda^{r_0}(M^1,\Theta)$ where $M^1=F^0(M^0)$.

We have
$$
|\tilde v^0|_{\cL M,r}\leq C_r|v^0|_{M^0,r}.
$$
Further, $F^0_k(z_k)=z_k+f_k^0(z_k)$ and $G^0_k(z_k)=z_k+g_k^0(z_k)$ with $G_k^0=((F^0)^{-1})_k$ satisfy
\al{}\label{Ini-3}
t_0^{\ve} L_{0,r_0}&\leq 1,\\
\label{fkdkj-mv0}
|f^0_k|_{D^0_k;s+1/2}+|g^0_k|_{D_k^{0+1};s+1/2}&\leq  C_s t_0^\all, \\ |f^{0}_k|_{D^{0}_k;r_0+1/2}+|g^0_k|_{D_k^{1};r_0+1/2}&\leq  C_{r_0}|A^0|_{r_0}.
\end{align}
Here we may need to adjust $t_0$.

This gives us for $j=0$ the following 
\al{}\label{Ajs-mvj+1}
|A^{j+1}|_{M^{j+1},s}&\leq t_{j+1}^\all,\\
\label{Ajm-mvj+1}
|A^{j+1}|_{M^{j+1},r_0}\leq C_{r_0}|A^j|_{M_0,r_0}&\leq C_{r_0} L_{j,{r_0}},\\
t_{j+1}^\all L_{j+1,r_0}&\leq1,\\
\label{fkdkj-mv0j}
|f^j_k|_{D^j_k;s+1/2}+|g^j_k|_{D_k^{j+1};s+1/2}&\leq  C_s t_j^\all,\\
|f^{j}_k|_{D^{j}_k;{r_0}+1/2}+|g^j_k|_{D_k^{j+1};{r_0}+1/2}&\leq C_{{r_0}}L_{j,{r_0}}.
\label{fkdkj-mv0-mj+1}
\end{align}
The iteration can be repeated when $j$ is replaced by $j+1$,  if we have
$$
|f^j_k|_{D^j_k;2}+|g^j_k|_{D_k^{j+1};2}\leq \f{C_2}{(j+1)^2}.
$$
By \rl{map_iter}, the latter implies that the stability of constants in various estimates holds for $M^{j+1}\colon\rho^{j+1}=\rho^j\circ G^j$ with $G^j=(F^j)^{-1}$. Then \re{Ajs-mvj+1} implies that we can repeat the procedure to construct the sequence of $\{M^j,F^j,\rho^{j+1}\}_{j=0}^\infty$.

By \re{fkdkj-mv0j}, we have $|f^j_k|_{D^j_k;s+1/2}+|g^j_k|_{D_k^{j+1};s+1/2}\leq C_st_j^\all\leq C_st_0^{d^{j}}$.
This shows that $\tilde F^j=F^j\circ\cdots\circ F^0$ converges to $\tilde F$ and $\tilde G^j=G^0\circ\cdots G^j$ converges in $\Lambda^s$ norm.  By \re{fkdkj-mv0-mj+1}, we have $|f^j_k|_{D^j_k;{r_0}+1/2}+|g^j_k|_{D_k^{j+1};{r_0}+1/2}\leq C_{r_0}L_{j,{r_0}}\leq C_{r_0}^jL_{0,{r_0}}$. Using interpolation, we know that $|f^j_{k}|_{r_\theta}$ converges rapidly in $\Lambda^{r_\theta}$ norm for
$r_\theta=(1-\theta)s+\theta {r_0}+1/2$ for any $0<\theta<1$.

\medskip
We now apply a standard bootstrapping argument.
For the above fixed parameters, we want to show that $F$ is actually in $\Lambda^{\ell}$ for any $\ell<m+1/2$, provided
$$
r_0<m<\infty, \quad A^0\in\Lambda^{m}(M_0).
$$
In this case, define
$
L_{j,m}=|A^0|_{M^j,m}
$
Then \re{Ajm-mvj+1}, \re{fkdkj-mv0-mj+1} still hold for $j\geq0$, if we replace $r_0$ by $m$. Thus,
\aln{}
|A^{j+1}|_{M^{j+1},m} \leq C_{m} L_{j,{m}},\quad
|f^{j}_k|_{D^{j}_k;{m}+1/2}\leq C_{m}L_{j,{m}}.
\end{align*}
Consequently, $L_{j+1,m} \leq C_{m} L_{j,{m}}\leq C_m^{j+1}L_{0,m}$ and $|f^{j}_k|_{D^{j}_k;{m}+1/2}\leq C_m^jL_{0,m}$.
We know that $|f^{j}_k|_{D^{j}_k;{s}+1/2}$ converges rapidly. By interpolation, we conclude that $|f^{j}_k|_{D^{j}_k;\ell}$ converges rapidly for any $\ell<m+1/2$.

\subsection{Convergence in finite many derivatives for the general case}
Let us recall the choice of parameters in subsection 6.4. Assume that
$$
r_0>s+1,   \quad s_*\geq 2,\quad r_0>\f{5/2}{s_*-1}.
$$
Take $s=1+\ve$ and $d=1+\ve$, where $\ve>0$ is  sufficiently small. Thus, we assume
$$
r_0>\max\Bigl\{2,\f{5/2}{s_*-1}\Bigr\},\quad s_*\geq 2.
$$
We now define $
t_{j+1}=t_j^d
$
and choose $t_0$ such that
$$
0<t_0<1, \quad
C_{s,r_0}t_0^\ve<1, \quad C_st_0^{-1/2+\all}<1/2.
$$

We are given a formally integrable almost complex structure $A^0\in\cL C^1(\ov{  {M^0}
},\Theta)$. Assume that
\ga{\label{A0M0}}
|A^0|_{M^0,s}\leq t_0^\all, \quad |A^0|_{M^0,r_0}\leq L_{0,r_0}.
\end{gather}
We now apply the homotopy formula
$$
A^0=\db P_0A^0+Q_0\db A^0, \quad \text{on $M_0^{t_0}$}.
$$
Here  $M_0^{t_0}=(M^0)^{t_0}$,
 $P_0=P_{M_0^{t_0}}$ and $Q_0=Q_{M_0^{t_0}}$. 
This yields a vector field $v_0=-P_0A_0\in \La^{s+1/2}(M^0)$. We now apply Moser's smoothing operator $S^{t_0}$ and define
$$
\tilde v^0=\chi E_{\cL N_{t_0}(M^0)} S^{t_0}v^0,
$$
where $\chi=1$ on $\cL N_{t_0}(M^0)$ and $\chi$ is compact support on $\cL M$.

 The extension and cut-off allow us define a time-one map $F^0$ on the whole space $\cL M$ that is the identity map away from the support of $\tilde v^0$.
Since $\tilde v^0=S^{t_0}v^0$ on $\cL N_{t_0}(M^0)$, then $F^0|_{M^0}$ agrees with the time-one map of the flow of $S^{t_0}v^0=-S^{t_0}P_0A^0$ on $\cL N_{t_0/C}(M^0)$. Therefore, $F^0$ transforms the almost complex structure defined by $A^0$ into a almost complex structure defined by $A^1\in \Lambda^{m}(M^1,\Theta)$ where $M^1=F^0(M^0)$.

We have
$$
|\tilde v^0|_{\cL M,r}\leq C_r|v^0|_{M^0,r}.
$$
Further, $F^0_k(z_k)=z_k+f_k^0(z_k)$ and $G^0_k(z_k)=z_k+g_k^j(z_k)$ with $G_k^0=((F^0)^{-1})_k$ satisfy
\al{}\label{fkdkj-mv0}
|f^0_k|_{D^0_k;s+1/2}+|g^0_k|_{D_k^{1};s+1/2}&\leq  C_s t_0^\all, \\ |f^{0}_k|_{D^{0}_k;r_0+1/2}+|g^0_k|_{D_k^{1};r_0+1/2}&\leq  C_{r_0}|A^0|_{r_0}.
\end{align}

This give us for $j=0$ the following \al{}\label{Ajs-mv}
|A^{j+1}|_{M^{j+1},s}&\leq t_{j+1}^\all,\\
\label{Ajm-mv}
|A^{j+1}|_{M^{j+1},{r_0}}\leq C_{r_0}|A_j|_{M^j,{r_0}}&\leq L_{j+1,{r_0}},\\
\label{fkdkj-mv0}
|f^j_k|_{D^j_k;s+1/2}+|g^j_k|_{D_k^{j+1};s+1/2}&\leq  C_s t_j^\all,\\
|f^{j}_k|_{D^{j}_k;{r_0}+1/2}+|g^j_k|_{D_k^{j+1};{r_0}+1/2}&\leq C_{m}L_{j,{r_0}}.\label{fkdkj-mv0-m}
\end{align}
The iteration can be repeated when $j$ is replaced by $j+1$,  if we have
$$
|f^j_k|_{D^j_k;2}+|g^j_k|_{D_k^{j+1};2}\leq \f{C_2}{(j+1)^2}.
$$
By \rl{map_iter}, the latter implies that the stability of constants in various estimate holds for $M^{j+1}\colon\rho^{j+1}=\rho^j\circ G^j$. Then \re{Ajs-mv} implies that we can repeat the procedure to construct the sequence of $\{M^j,F^j,\rho^{j+1}\}_{j=0}^\infty$.

By \re{fkdkj-mv0}, we have $|f^j_k|_{D^j_k;s+1/2}+|g^j_k|_{D_k^{j+1};s+1/2}\leq C_st_j^\all\leq C_st_0^{d^{j}}$.
This shows that $\tilde F_j=F^j\circ\cdots\circ F^0$ converges to $\tilde F$ and $\tilde G^j=G^0\circ\cdots G^j$ converges in $\Lambda^s$ norm.  By \re{fkdkj-mv0-m}, we have $|f^j_k|_{D^j_k;{r_0}+1/2}+|g^j_k|_{D_k^{j+1};{r_0}+1/2}\leq C_{r_0}L_{j,{r_0}}\leq C_{r_0}^jL_{{r_0},0}$. Using interpolation, we know that $|f^j_{k}|_{r_\theta}$ converges rapidly in $\Lambda^{r_\theta}$ norm for
$r_\theta=(1-\theta)s+\theta {r_0}+1/2$ for any $0<\theta<1$.

Suppose now that $b M^0\in C^{\ell+1/2}$ and $A^0\in \Lambda^{\ell}$ with $\ell\in (r_0,\infty)\setminus\f{1}{2}\nn$.
 We want to show that without further adjusting the parameters $\all,d,m,s$ the above sequence $f_j,g_j$ converge rapidly in $\Lambda^{\ell'}$ for any $\ell'<\ell+1/2$.

We still keep $s_*$ and $\hat t=(c_*t)^{1/{s_*}}$ in \re{rhoHatt} unchanged. However, we replace the $s_*$
 in  $|\rho|_{s_*}$ by any $\ell+1/2$ for the third inequality below and replace $\mu(a)$  by
$$
 \mu_\ell(a):=\f{(a-\ell-1/2)_+}{s_*}.
$$
Then we obtain
\ga
\|\rho_{\hat t}-\rho\|_{0}\leq C_{ s } \hat t^{s_*}\|\rho\|_{{s_* }},\label{rho1-1-25-s} \\
\|\rho_{\hat t}-\rho\|_{2}\leq \ve,\label{rho1-2c-25-s} \quad \hat t<\hat t^*(\nabla^2\rho),\\
  \|\rho_{\hat t}\|_b\leq C_{b,{ \ell }}t^{-\mu_\ell(b-\ell-1/2)}\|\rho\|_{ \ell+1/2}, \quad b>0.
\label{rho1-3-25-s}
\end{gather}

 Recall $R_{s_*}=|\rho|_{s_*}$ by definition.
 Then
$$
 \mu_\ell(a)\leq \mu_\ell(\ell+3)=\f{5}{2s_*}=\mu(s_*+\f{5}{2}),\quad  \forall a\leq \ell+3.
$$
By  \re{ti_A3ell} in which we take $s_*=\ell$ and $\mu=\mu_\ell$, we obtain
\aln{}
|\tilde A|_{\ell}&\leq C_\ell (|A|_{\ell}+(t^{2\all - \mu_\ell(\ell+3)-b\ve}+t^{\all+s-\mu_\ell(\ell+5/2)})|\rho|_{\ell+1/2}).
\end{align*}
By \re{ar+1} and \re{am52s}, we have
$$\all-\mu_\ell(\ell+3)=\all-\mu(s_*+5/2)>0,\quad \all+s-\mu_\ell(\ell+3)>0.$$ Hence
\aln{}
&|\tilde A|_{\ell}\leq C_\ell'  \max(|\rho|_{\ell+1/2}, |A|_{\ell}),\\
&|g_k|_{\ell+1/2}\leq C_\ell |f_k|_{\ell+1/2}\leq C'_\ell (t^{\all- \mu_\ell(\ell+3)}|\rho|_{\ell+1/2}+ |A|_{\ell})\leq C_\ell''  \max(|\rho|_{\ell+1/2}, |A|_{\ell}).
\end{align*}
By \re{A'}, we have
$$
|A'|_{\ell}\leq C_\ell (|\tilde A|_{\ell}+|g|_\ell+ L_0)\leq  C_\ell'  \max(|\rho|_{\ell+1/2}, |A|_{\ell}).
$$
We also have
$$
|\rho\circ G|_{\ell+1/2} \leq C_\ell(|\rho|_{\ell+1/2}(1+|G|_{1+\ve})+C_\ell |\rho|_{1+\ve}|G|_{\ell+1/2})\leq C_\ell' \max(|\rho|_{\ell+1/2}, |A|_\ell).
$$
Therefore, we have
\aln{}
\max\{|A^{j+1}|_{\ell}, |\rho^{j+1}|_{\ell+1/2})\}\leq C_\ell^j \max( |A^0|_{\ell},|\rho^0|_{\ell+1/2}).
\end{align*}
  By interpolation, we see that $|f_j|_{\ell'}$ and $|g_j|_{\ell'}$ converge rapidly for any $\ell'<\ell+1/2$.
  Therefore,  under conditions  \re{A0M0} and
  $$
  A\in\Lambda^{\ell}, \quad  b M\in \Lambda^{\ell+1/2}, \quad \ell>r_0>2
  $$
we have an embedding $F_\infty\in \Lambda^{\ell'}$ for any $\ell'<\ell+1/2$.

We remark that conditions \re{Ini-1}, \re{Ini-2}, \re{Ini-3} can be written as \re{sr0} when we replace $t_0$
by $\hat t_0^\mu$ for some constant $\mu>1$ that depends on $\ve,\ve_1$.

In summary, we have proved the following.
\begin{prop}\label{weaker}\rta{Global NN mv} holds if we assume additionally
 \ga{}\label{extra2}
\rho\in\Lambda^{r+1/2}.
\end{gather} 
for its last assertion.
\end{prop}
In next two subsections, we will remove the extra condition \re{extra2}. Then the proof of \rt{Global NN mv} is complete.
\subsection{Improving  regularity  away from concave boundary}
It is well-known that by the classical Newlander-Nirenberg theorem that $F_0$ gains one derivative on $M$.
To improve the regularity of $F_0$ in a neighborhood of the strongly pseudoconvex boundary $b^+_{n-1}M$, we only need to slightly modify the arguments in the previous subsection.

Fix a  neighborhood $\om^+$ of $b_{n-1}^+M$ in $\ov M$ and a neighborhood $\om^-$ of $b^-_{2}M$ in $\ov M$ so that $\ov{\om^+}\cap\ov{\om^-}=\emptyset$.
Recall that using Grauert bumping, we have
 $$\var_{j}=\var_{j-1}-\db(\chi_{j}H_{U_{j},q}\var_{j-1})
 -\chi_{j}H_{U_{j},q+1}\db\var_{j-1},\quad j=1,\dots, N.$$
We can choose $\om_j,\om_j'$ so that $\om_j'\subset\om^+$ (resp. $\om_j'\subset\om^-$)  if $\om_j\cap b_{n-1}^+M\neq\emptyset$ (resp. $\om_j\cap b_{2}^-M\neq\emptyset$). Consequently,
$$
|\var_N|_{\om^+,a}\leq C_a|\var|_{M,a},\quad |\db\var_N|_{\om^+,a}\leq C_a|\db\var|_{M,a},
$$ for any $a>0$ even if $ b M$ is merely $C^2$.

We still use \re{Rjfe}-\re{tHqv}. Hence, we still have all estimates the previous subsection. Next, we improve and in fact simplify the estimates away from $b_2^-M$, using \re{Rjf} on subdomains of $\om^+$ containing $b_{n-1}^+M$.

Recall that $M$ is defined by $\rho<0$.   Set
$$
M_{j,\ve}=M^j\setminus\Bigl\{\rho^j\geq-\ve+\f{\ve}{\sqrt{j+2}}\Bigr\},\quad
M^{\pm}_{j,\ve}=M^j\setminus\Bigl\{\rho_j\geq-\ve+\f{\ve }{\sqrt{j+1\mp 1/2}}\Bigr\},
$$
where $\ve$ is a small positive number. We have
$$
 M_{j+1,\ve}^+\subset M_{j,\ve}^-\subset M_{j,\ve}\subset M_{j,\ve}^+, \quad
F^{j}(M^-_{j,\ve})\supset M^+_{j+1,\ve}, \quad j>\ell_\ve.
$$
Consequently, 
$
G^{\ell_\ve}\cdots G^{\ell}\colon M_{\ell_\ve+1,\ve}^+\to M_{\ell,\ve}^+
$ 
is well-defined.

We now have
$$
|v^j|_{M_{j,\ve},r+1/2}\leq \f{C_r}{(j+1)^2}(|A^j|_{M^+_{j,\ve},r}+|A^j|_{M^j,0}).
$$

This allows us to carry over the estimates for strongly pseudoconvex domains. Further, we only need to improve estimates for the highest order derivatives. We have
\ga{}\label{fjM-}
|f^j|_{M^-_{j,\ve},m+1/2}\leq \f{C_m(j+1)^{m+C}}{\ve^{m+C}}(|A^j|_{M^+_{j,\ve},m}+\|A^j\|_{M^j,0}),\\
\label{gjM-}
|g^j|_{M^+_{j+1,\ve},m+1/2}\leq \f{C_m(j+1)^{m+C}}{\ve^{m+C}}(|A^j|_{M^+_{j,\ve},m}+\|A^j\|_{M^j,0}),\\
|A^{j+1}|_{M^+_{j+1,\ve},m}\leq \f{C_m(j+1)^{m+C}}{\ve^{m+C}}(|A^j|_{M^+_{j,\ve},m}+\|A^j\|_{M^j,0}).
\label{AjM-}
\end{gather}
By \re{AjM-} and the rapid convergence of $|A_{j+1}|_{s}$ via \re{Ajs-mv}, we obtain
$$
|A^{j+1}|_{M_+^{j+1,\ve},m}+\|A^{j+1}\|_{M^j,0}\leq C_m\ve^{-m-C}(j+1))^j(|A^{0}|_{M^+_{1,\ve},m}+\|A^0\|_{M^0,0}).
$$
By \re{gjM-}, \re{gks1/2}, and interpolation, we see that $|g^j|_{M_+^{j+1},a}$ converges rapidly for any $a<m+1/2$. This shows that  $\lim_{\ell\to\infty}
G^{\ell_\ve}\cdots G^{\ell}$ is in $\Lambda^a(M^+_{\ell_\ve+1,\ve})$. Therefore, $F_\infty$ is in $\Lambda^a$ on a neighborhood of $b_{n-1}^+(M)$ in $\ov M$ for any $a<r+1/2$.

\subsection{Improving regularity near concave boundary via Hartogs extension}\label{subsec:8.4}
Let us summarize the results we have proved so far when $r_0,s_*$ satisfy \re{r0cond}-\re{sr0}. Suppose
\eq{onM0}
t_0^\all\|A\|_{M,s}\leq 1, \quad t_0^{-\ve}\|A\|_{M,r_0}\leq 1.\eeq
\bpp
\item  On $M$, we have founded a global coordinate map  $F\colon \ov {M}\to \cL M$ for $A$ (i.e. that $F$ transforms the almost complex structure defined by $A$ into the standard complex structure on $F(M)$). We also have    $F\in C^2(\ov{M})\cap \Lambda^{\ell}(\cL N_\del( b_1^+M))$ for any $\ell<r_0+1/2$. Further, $F(M)$ is still an $a_1^+$ domain.
    \item Suppose $m>r_0$. Suppose $\widetilde M$ with $b\widetilde M\in\Lambda^{m+1/2}$ is a small $C^2$ perturbation of $b M$ and $\tilde A\in\Lambda^m(\widetilde M)$ defines a formal integrable almost complex structure with $|\tilde A|_{r_0}$ being sufficiently small. There is a global coordinate map $\tilde F\in\Lambda^\ell(\widetilde M)$ for $\tilde A$,
        for any $\ell<r+1/2$. Furthermore, $\|\tilde F-\id\|_2<\del$ and $\tilde F(M)$ is still an $a_1^+$ domain.
\epp
We want to show that without requiring $b_1^+M$ to be more smooth, the $F$ in $(a)$ is actually in  $\Lambda^{\ell}(M)$ for any $\ell<m+1/2$, assuming $A\in\Lambda^m$.
Here we apply Hartogs's extension. This technic has been used recently in~\ci{aq} for solving the $\db$ problem. 

Suppose $\zeta\in b M$.   By \rl{touch}, we can find a smoothly bounded  subdomain $ M^*_{\zeta}$ of $M$   that is a small perturbation of $M$ in class $ C^2\cap\Lambda^{s_*}$ such that $\zeta\in b M^*_{\zeta}$. Since $M_\zeta^*$ is a subdomain of $M$, \re{onM0} implies that
$$
t_0^\all\bigl|{A}|_{M_\zeta^*}\bigr|_{M_\zeta^*,s}\leq 1, \quad t_0^{-\ve}\bigl|A|_{M^*_\zeta}\bigr|_{M^*_\zeta,r_0}\leq 1.$$
Therefore, we can apply $(b)$ and find a coordinate map $F_\zeta$ for $A|_{M_\zeta^*}$ that is smooth on $\ov{M_\zeta^*}$ such that $|F_\zeta-\id|_{M_\zeta^*}<\ve$. It is clear that
$$
M=\bigcup_{\zeta\in b_1^-M}M_\zeta^*.
$$

To apply Hartogs's extension theorem, we take a local homomorphic coordinate map $\phi$ from a small neighborhood $U_\zeta$ of $\zeta$ onto $P_{\del,\ve}= \Del_\del^{n-1}\times ([-\del,\del]+i  [-\ve,\ve])$. We may assume that $\phi(\zeta)=0$. Further, we have
\aln{}
D&:=\phi(U_\zeta\cap M)=\{z\in P_{\del,\ve}\colon   y_n>-|z_1|^2+\sum\la_j|z_j|^2+R(x,y')\},\\
D_\zeta&:=\phi(U_\zeta\cap M_\zeta^*)=\{z\in  P_{\del,\ve}\colon y_n>-|z_1|^2+\sum\la_j|z_j|^2+R_\zeta^*(x,y')\},
\end{align*}
where $|\la_j|<1/4$,  $|R_\zeta^*(x,y')|+|R(x,y')|\leq |(x,y')|^2/4$.
Both domains contain the Hartogs domain
$$
\Gamma:= \{z\in P_{\del,\ve}\colon   y_n>-\f{1}{2}|z_1|^2+\sum\f{1}{2}|z_j|^2\}.
$$
Since $M_\zeta^*$ is contained in $M^0$, then the composition
$$
H:=\phi  F_\zeta
F^{-1}\phi^{-1}$$
 is a holomorphic map from $D$ into $\cc^n$. In particular, $H=(H_1,\dots, H_n)$ is holomorphic on $\Gamma$. By the Cauchy formula, we obtain
\eq{CauchyH}
H_j(z)=\f{1}{(2\pi i)^n}\int_{|\eta_1|=\ve}\int_{|\eta_2|=\del}\cdots \int_{|\eta_n|=\del}\frac{H_j(\eta)\, d\eta_1\cdots d\eta_n}{(\eta_1-z_1)\cdots(\eta_n-z_n)}
\eeq
on $\om:=(\Del^{n-1}_{\del/C}\times\Del_{\del})\cap D\cap\{y_n>\del/2\}$. By uniqueness theorem, \re{CauchyH} is valid on $\om$.  This shows that
$$
|H|_{\om, \ell}\leq C_{\del,\ve,\ell}\|H\|_{\Gamma\cap D,0}.
$$
Consequently, $H^{-1}$ is in $\Lambda^\ell$ on $H(\om)$.
Then $F=\phi^{-1}H^{-1}\phi F_\zeta$ satisfies
$$
|F|_{M_\zeta^*\cap \tilde \Del_{\del,\ve}(\zeta),\ell}\leq C_\ell|F_\zeta|_{M_\zeta^*,\ell}\leq C_\ell'|A|_{M_\zeta^*,\ell}\leq C_\ell'|A|_{M^0,\ell}.
$$
The latter is independent of $\zeta$. This shows that $|F|_{M,\ell}<\infty$   for any $\ell<m+\yh$.
\end{proof}

\begin{rem}In our main results on the global Newlander-Nirenberg problem, we only use estimate in \rt{Thm::glob_hf_intro} for $\delta=1/2$ when $bM$ has the concave component. It is likely that arguments in Sections~\ref{set-up} and~\ref{sect:hartogs}
can lead to new results, if we apply \rt{Thm::glob_hf_intro} for $0\leq\delta<1/2$. 
\end{rem}

\appendix

\setcounter{thm}{0}\setcounter{equation}{0}
\section{Stability of constants in equivalence norms}\label{HF-low-reg}

The main purpose in this appendix is to describe the stability of various constants in Rychkov~\ci{MR1721827}
and Dispa~\ci{MR2017700}. In this appendix, we use notation in~\cites{MR1721827,MR2017700}.

\subsection{Case of special Lipschitz domains}
We will only focus on the case that $\Om$ is a \emph{special Lipschitz domain} in $\rr^n$ given by
$$
x_n>\om(x'), \quad x'\in\rr^{n-1}
$$
with $|\om(\tilde x')-\om(x')|\leq A|\tilde x'-x'|$. Then the cone $K$ defined by $x_n>A|x'|$ satisfies 
$\Om+K\subset\Om$. We call $A$ a Lipschitz constant of $\Om$.

Assume that $s\in\rr$. We describe some properties on $\cL C^s(\Om)$, which is the Besov space $B^s_{\infty,\infty}(\Om)$ as defined in~\cites{MR1721827,MR2017700} when $\Om$ is a special Lipschitz domain and more general a bounded Lipschitz domain.

Let $\var_0$ be a smooth function on $\rr^n$ with compact support in $-K$ and $\int\var_0\, dx=1$. Define
$$
\var(x)=\var_0(x)-2^{-n}\var_0(x/2), \quad \var_{j}(x)=2^{jn}\var(2^jx), \quad j\geq1.
$$
We say that $\var$ satisfies the {\it moment condition} of order $L_\var$, namely
$$
\int x^\all\var(x)\, dx=0, \quad |\all|\leq L_\var.
$$
Analogously, suppose $\la_0\in \cL D(-K)$ also satisfies
 the moment condition of order $L_{\la}$. 

We have $\var_0+\cdots+\var_j=2^{jn}\var_0(2^jx)$. For $f\in\cL D'(\Om)$, we have  
$$
f=\sum_{j=0}^\infty\var_j\ast f.
$$

For $f\in\cL D'(\rr^n)$, define
$$
\|f\|_{\cL C^s}:=\sup_j\sup_{x\in\mathbb R^n}2^{sj}|\var_j \ast f(x)|.
$$
Note   that
$
\|f\|_{\cL C^s}=\sup_{x\in\mathbb R^n,j} 2^{sj}|\var_j \ast f(x)|.
$  
Define for $f\in\cL D'(\Om)$
$$
\|f\|_{\cL C^s(\Om)}=\inf\Bigl\{\|\tilde f\|_{\cL C^s}\colon  \text{$\tilde f\in\cL \cL D'(\rr^n),\tilde f|_\Om=f$ in $\cL D'(\Om)$}\Bigr\}.
$$
Let $\cL C^s(\Om)$ be the set of all such $f\in\cL D'(\Om)$ with $\|f\|_{\cL C^s(\Om)}<\infty$, i.e.  $\tilde f|_\Om=f$ for some $\tilde f\in \cL C^s(\rr^n)$. When $s>0$, $\Lambda^s(\rr^n)=\cL C^s(\rr^n)$ in equivalent norms (see~\cites{MR2017700} and references therein.)

The Peetre-type maximum functions of $f\in\cL D'(\Om)$ are
$$
\var_{j,N}^\Om f(x)=\sup_{y\in\Om}\f{|\var_j\ast f(y)|}{(1+2^j|x-y|)^N}, \quad x\in\Om.
$$
%
\begin{rem}\label{trivalmax}
It holds trivially that for $f\in \cL D'(\Om)$
$$
\|\var_{j,N}^\Om f\|_{L^\infty(\Om)}=\|\var_j\ast f\|_{L^\infty(\Om)},\quad N\geq0.
$$
\end{rem}

\le{the-comp} Suppose that $\min\{L_\var,L_\la\}\geq M$ and $M$ is finite. Then
$$
\|\la_\ell\ast \var_j \|_{L^1(\rr^n) }\leq C_{\var_0,\la_0,K,M}2^{-|j-\ell| (M+1)}
$$
for any $j,\ell\geq0$, 
\ele
\begin{proof}
This is a consequence of  estimate~\ci{MR1721827}*{(2.13)}
$$
\int |\la_\ell\ast\var_j|(z)(1+2^j|z|)^N\, dz\leq
\begin{cases} C  2^{-|\ell-j|(L_\var+1-N) }& j\geq \ell,\\
C 2^{-|\ell-j|(L_\la+1) }, &j<\ell
\end{cases} 
$$
for $\ell,j\geq0 $ and $ N\geq0$, where $C$ depends on $\la_0,\var_0$ and is independent of $\ell,j$. 
See \ci{SYajm}*{Lemma 4.3} for further details.
\end{proof}

\begin{prop}[\cite{MR1721827}*{Prop.~2.1}] \label{Rprop2.1} Fix $L\geq0$ and fix $\var_0\in \cL D'(-K)$ with $\int\var_0=1$.  There exists $ \psi_0,\psi\in \cL D'(-K)$ such that $\psi$ has vanishing moments up to order $L_\psi\geq L$ and  
$$
f=\sum_{j=0}^\infty \psi_j\ast\var_j\ast f, \quad \forall f\in\cL D'(\Omega).
$$
Here $\psi_{j+1}(x)=2^{jn}\psi(2^jx)$ for $j\geq 0$. 
\epr 
\begin{rem}\bpp
\item  Unlike $\var,\var_0$ the relation  $\psi(x)=\psi_0(x)-2^{-n}\psi_0(x/2)$ may not hold. In~\ci{SYajm},
$\{\var_j\}_{j=0}^\infty$ is called a {\it regular dyadic resolution}, while $\{\psi_j\})_{j=0}^\infty$ is called a {\it generalized dyadic resolution}.  
\item The proof of \ci{MR1721827}*{Prop.~2.1} does not use  \rp{Rprop1.2} below. 
This allows us to simplify the proof of \ci{MR1721827}*{Prop.~1.2}
\epp
\end{rem} 

\begin{prop}[\cite{MR1721827}*{Prop.~1.2}]\label{Rprop1.2}
Let $\la_0,\var_0\in\cL D(-K)$ with  $\min\{L_\la,L_\var\}\geq   |s|$. Then for any  $f\in\cL D'(\Om)$
$$
\sup_j2^{sj}\|\la_j\ast f\|_{L^\infty} =\sup_j2^{sj}\|\la^{\mathbb R^n}_{j,N}  f\|_{L^\infty} \leq
C_s\sup_j2^{sj}\|\var_j\ast f\|_{L^\infty}, \quad N>0.
$$
 The constants are for $\rr^n$ and depend only on the Lipschitz norm of $\Om$.
\end{prop}
\begin{proof} Write $\min\{L_\var,L_\la\}\geq M\geq|s|$ with $M<\infty$. 
 Applying \rp{Rprop2.1} to $\Om=\rr^n$, we find $\eta_0,\eta$ such that $L_\eta\geq M$ and 
$$
f=\sum_{j=0}^\infty\eta_k\ast\var_k\ast f.
$$
By  moment conditions on $\eta$ and $\psi$,  Young's inequality $\|u\ast v\|_{L^\infty}\leq \|u\|_{L^1}\|v\|_{L^\infty}$, and \rl{the-comp}, we obtain
\aln{}
\|\la_j\ast f\|_{L^\infty}&\leq   \sum_{k}
\|\la_j\ast\eta_k\ast\var_k\ast  f\|_{L^\infty} \leq   \sum_{k}
\|\la_j\ast\eta_k\|_{L^1}\|\var_k\ast  f\|_{L^\infty}\\
&\leq  C_M\sum_{k}2^{-|j-k|(M+1)} 
 \|\var_k\ast  f\|_{L^\infty}\\
&\leq C_M 2^{-js}(\sup_k2^{sk}\|\var_k\ast  f\|_{L^\infty}) \sum_k 2^{-|j-k|(M+1-|s|)}.
\end{align*}
The last sum is bounded by $3$ since $M\geq|s|$.
\end{proof}

\rp{Rprop2.1} allows Rychkov~\ci{MR1721827} to define an extension operator $\cL E$ for $\Om$.
\begin{thm}[\cite{MR1721827}*{Thms.~2.2, 2.3, 3.2}]\label{Rthm2.2} Let $\var, \psi_0,\psi$ be as in \rpa{Rprop2.1}.  Let $g_\Om$ be the trivial extension by $0$ on $\rr^n\setminus\Om$.  Suppose $\min\{L_\la,L_\var\}\geq   |s|$.
\bpp \item Suppose $f\in\S'(\Om)$, i.e. $|\jq{f,\gamma}|\leq C\sup|D^\all\gamma(x)|(1+|x|)^M
$
for some $M=M_f$ and $C=C_f$. Then 
$$
\|f\|_{\cL C^s(\Om)}\approx   \sup 2^{js}\|\var_j\ast f\|_{L^\infty(\Om)}.
$$
\item Let 
$\cL E\colon \S'(\Om)\to D'(\rr^n)$ be defined by $$
\cL Ef:= \sum\psi_j\ast(\var_j\ast f)_\Om. 
$$
Then
$
\|\cL Ef\|_{\cL C^s(\mathbb R^n)}\leq C_s(\Om)\|f\|_{\cL C^s(\Omega)}
$
for $f\in\cL \cL C^s(\Omega)$. 
\epp
\end{thm}
\begin{proof} (a) For any $g\in \cL C^s$ with $g|_\Om=f$, we have   $\var_j\ast g=\var_j\ast f$ on $\Om$. Thus $\sup 2^{js}\|\var_j\ast f\|_{L^\infty(\Om)}=\sup 2^{js}\|\var_j\ast g\|_{L^\infty(\Om)}\leq \|g\|_{\cL C^s}.$  This shows that
$\sup 2^{js}\|\var_j\ast f\|_{L^\infty(\Om)}\leq\|f\|_{\cL C^s(\Om)}$. 

To  show the inequality in the other direction, it suffices to verify $(b)$.
Note that the estimate in the proof of \rp{Rprop1.2} does not use the moment conditions on $\var_k\ast f$.  Replacing the latter by $(\var_k\ast f)_\Om$ and $\eta_j$ by $\var_j$,  we get
\gan{}
\var_j\ast\cL Ef(x)=\sum_k\var_j\ast\psi_k\ast  (\var_k\ast f)_\Omega(x),\\
2^{js}\|\var_j\ast\cL Ef\|_{L^\infty}\leq  C\sup 2^{ks}\|(\var_k\ast f)_\Omega\|_{L^\infty}.
\end{gather*}
We have verified $\|\cL Ef\|_{\cL C^s}\leq \sup 2^{ks}\|\var_k\ast f\|_{L^\infty(\Omega)}.
$
\end{proof}

\begin{thm}[\cite{MR2017700}*{Thm.~3.18}]\label{disp-equiv}
Let  $\infty>M>s>0$ and let $\Om$ be a special Lipschitz domain in $\rr^n$.  Then for any $f\in\cL C^s(\Om)$, 
\aln{}
\|f\|_{\cL C_M^s(\Om)}:=\|f\|_{L^\infty(\Om)}+\sup_{h,0<|h|\leq1}|h|^{-s} \|\Del^M_hf\|_{L^\infty(\Om)}
\end{align*}
is equivalent to $\|f\|_{\cL C^s(\Om)}$. That is if $f\in\cL C^s(\Om)$ then
$$
C_s^{-1}\|f\|_{\cL C^s(\Om)}\leq\|f\|_{C_M^s(\Om)}\leq C_s\|f\|_{\cL C^s(\Om)}
$$
where $C_s$ depends on only the Lipschitz norm of $\pd\Om$. 
\end{thm}
\begin{proof}Dispa computed $C_s$  via  constants in Theorems 2.3 and 3.2 in~\ci{MR1721827}. 
\end{proof}

\subsection{Case of bounded Lipschitz domains}

Let $\Om$ be a relatively compact Lipschitz domain in $\rr^n$. Let $M\subset\cup U_j$. Let $\{\chi_j\}$ be a partition of unity subordinate to the covering $\{U_j\}$.  Then for a function $f$ on $\Om$. 
We may assume that $\Om\cap\supp \chi_j$ is contained in a special Lipschitz domain $\Om_j\subset\rr^n$ and $U_j\cap b\Om\subset b\Om_j$. Thus, $\chi_j f$ is defined on $\Om_j$ by trivial extension of being $0$ on $\Om_j\setminus(\Om\cap U_j)$.  Define
$$
\|f\|^*_{\cL C^s(\Om)}:=\max_j\|\chi_jf\|_{\cL C^s(\Om_j)},\quad\|f\|^*_{\cL C^s_M(\Om)}:=\max_j \|\chi_jf\|_{\cL C^s_M(\Om_j)}.
$$
Then \rt{disp-equiv} implies that $\|\cdot\|^*_{\cL C^s(\Om)}$ and $\|\cdot\|^*_{\cL C^s_M(\Om)}$ are equivalent. 

\begin{defn}[\cites{MR1721827,MR2017700}] Let $D$ a bounded Lipschitz domain $\rr^n$. Let $\cL C^s(\Om)$ be the set  of functions $f$ on $\Om$ such that there is $\tilde f\in\rr^n$ such that $\tilde f|_\Om=f$. For $f\in\cL C^s(\Om)$, define
$$
\|f\|_{C^s(\Om)}=\inf\{\|\tilde f\|_{\cL C^s}\colon\tilde f\in\cL C^s(\rr^n),\tilde f|_\Om=f\}.
$$
\end{defn}
Finally,  we mention that  a classical theorem~\ci{MR0521808}*{Thm.~18.5, p.~65} says that when $s>0$ and $\Om$ is bounded Lipschitz domain in $\rr^n$, then  $\|f\|^*_{\cL C^s_M(\Om)}<\infty$ implies that $f\in\cL C^s(\Om)$, i.e. $f$ extends to a function in $\cL C^s(\rr^n)$.

\begin{cor}Suppose that $s>0$. \rta{disp-equiv} holds when $\Om$ is a bounded Liptschitz domain i.e.
$$
C_s^{-1}\|f\|_{C^s(\Om)}\leq \|f\|^*_{\cL C^s_M(\Om)}\leq C_s\|f\|_{C^s(\Om)}
$$
and $C_s$ depends only on the Liptschitz norm of $\pd\Om$. 
\end{cor}
\begin{proof}Fix $f\in C^s(\Om)$. By definition, there is $\tilde f\in \cL C^s(\rr^n)$ such that
$\|\tilde f\|_{\cL C^s}\leq 2\|f\|_{C^s(\Om)}$. On $\Om$ we have $f=\sum\chi_j\tilde f$. Inductively, we can verify that 
$$\Del^M_h(f\chi)(x)=\sum_{i,j,k}\sum_{j\leq i,j\leq M-k}C_{i,j,k}\Del^{i}_{h}\tilde f(x+jh)\Del^{M-i}_h\chi(x+kh).
$$
Therefore, $|\Del^M_h(f\chi)(x)|\leq C_s\sum_{i=0}^M|h|^{s-M+i}|\tilde f|_{\cL C^{s-i}}|h|^{M-i}\|\chi\|_{C^{M-i}}. $
On $\rr^n$, it is well-known that $\|\tilde f\|_{\cL C^a}\leq C_{a,b}\|\tilde f\|_{\cL C^b}$ for $a<b$. This shows that
$\|f\|^*_{\cL C^s_M(\Om)}\leq C_s\|\tilde f\|_{\cL C^s}\leq 2C_s\|f\|_{\cL C^s(\Om)}$. 

For the other inequality, we start with  
$$
\|f\|_{\cL C^s(\Om)}=\|\sum\chi_j f\|_{\cL C^s(\Om)}\leq \sum\|\chi_j f\|_{\cL C^s(\Om)}\leq \sum\|\chi_j f\|_{\cL C^s(\Om_j)}.
$$
Let $\cL E_j$ be the Rychkov extension for $\Om_j$. We have 
$$
\cL E_j(\chi_jf)=\sum_k\psi_{j,k}\ast(\var_{j,k}\ast(\chi_jf))_{\Om_j}.
$$
By \rt{Rthm2.2} (b), we have $\|\chi_jf\|_{\cL C^s(\Om_j)}\leq \|\cL E_j(\chi_jf)\|_{\cL C^s(\rr^n)}\leq C_s\|\chi_jf\|_{\cL C^s(\Om_j)}$. By \rt{disp-equiv} for special Lipschitz domains, we get $\|\chi_jf\|_{\cL C^s(\Om_j)}\leq \|\chi_jf\|_{\cL C_M^s(\Om_j)}$. Hence $\|f\|_{C^s(\Om)}\leq C_s \|f\|^*_{\cL C^s_M(\Om)}$. \qedhere  
\end{proof}

\newcommand{\doi}[1]{\href{http://dx.doi.org/#1}{doi:#1}}
\newcommand{\arxiv}[1]{\href{https://arxiv.org/pdf/#1}{arXiv:#1}}

  \def\MR#1{\relax\ifhmode\unskip\spacefactor3000 \space\fi%
  \href{http://www.ams.org/mathscinet-getitem?mr=#1}{MR#1}}

\nocite{}
\bibliographystyle{alpha}



\end{document}